%% file: main.tex
\documentclass[12pt]{article}

\usepackage[utf8]{inputenc} 
\usepackage[T1]{fontenc}    
\usepackage{hyperref}       
\usepackage{url}            
\usepackage{booktabs}       
\usepackage{amsfonts}       
\usepackage{nicefrac}       
\usepackage{microtype}      
\usepackage{bookmark}
\usepackage{fullpage}
\usepackage{pifont}
\usepackage{enumerate}
\usepackage{enumitem}
\usepackage{dsfont}
\usepackage{array}
\usepackage{authblk}
\usepackage{wrapfig}
\usepackage{bbm}
\usepackage{booktabs}
\usepackage[sort, numbers]{natbib}
\hypersetup{
	colorlinks = true,
	citecolor = blue,
	linkcolor = black
}

\usepackage{graphicx}
\usepackage{caption}
\usepackage{subcaption}
\usepackage{amsmath}
\usepackage{amsthm}
\usepackage{amssymb}
\usepackage{tikz}
\usepackage{mathtools}
\usepackage{tablefootnote}
\usepackage{multirow}
\usepackage{xparse}
\usetikzlibrary{arrows}

\allowdisplaybreaks[4]

\usepackage{mathrsfs}

\usepackage{algorithm}
\usepackage{algorithmic}
\usepackage{bm}
\def\tod{\overset{\text{d}}{\rightarrow}}

\def\toas{\overset{a.s.}{\to}}
\def\topb{\overset{p}{\to}}

\newtheorem{thm}{Theorem}[section]
\newtheorem{lem}{Lemma}[section]

\newtheorem{prop}{Proposition}[section]
\newtheorem{asmp}{Assumption}[section]

\newtheorem{rem}{Remark}[section]

\mathtoolsset{showonlyrefs}
\NewDocumentCommand{\mybar}{ O{0.8} O{0pt} m }{
    \mathrlap{\hspace{#2}\overline{\scalebox{#1}[1]{\phantom{\ensuremath{#3}}}}}\ensuremath{#3}
}

\input{tex/math_commands.tex}

\begin{document}

\title{
 Extending Subsampling to Sequential Stopping}
\author[1]{Jose Blanchet} 
\author[1]{Peter Glynn}
\author[1]{Wenhao Yang}
\affil[1]{{\normalsize Management Science and Engineering, Stanford University}}

\maketitle

\begin{abstract}%
    Fixed-width sequential stopping rules terminate a stochastic simulation once an estimated confidence interval reaches a prescribed width. Classical fixed-width theory typically relies on a strongly consistent estimator of the asymptotic variance. This makes the normalized stopping time asymptotically deterministic, allowing fixed-sample-size limit theory to be transferred to the estimator at termination. This mechanism can fail when simulation output has infinite variance or long-range dependence. Although self-normalization and subsampling can yield asymptotically valid confidence intervals at a fixed sample size, the scaling process and the stopping time may retain nondegenerate randomness, so fixed-sample-size quantiles need not provide valid coverage at termination. In this paper, we develop a unified framework based on a joint functional limit theorem for the estimation process and a scaling process. We characterize the asymptotic behavior of both the stopping time and the self-normalized estimator evaluated at termination, thereby obtaining asymptotically valid sequential confidence intervals in classical finite-variance and infinite-variance settings. Moreover, we introduce a sequential subsampling procedure that consistently estimates the distribution relevant at the stopping time without directly estimating nuisance parameters in the limit distribution. The framework is verified for heavy-tailed moving-average processes, stochastic approximation, and an M/G/1 queue with heavy-tailed service times.
\end{abstract}

\section{Introduction}
\label{sec: intro}
\input{tex/intro}

\section{Preliminary}
\label{sec: pre}
\input{tex/pre.tex}

\section{A General Sequential Stopping Rule}
\label{sec: frame}
\input{tex/frame.tex}
\section{Construction of Scaling Process \texorpdfstring{$Z_2(\cdot)$}{Z2}}
\label{sec: alt}
\input{tex/alt.tex}

\section{Examples}
\label{sec: eg}
\input{tex/example.tex}

\section{Proofs}
\label{sec: proofs}
\input{tex/appendix.tex}

\section{Conclusion}
\label{sec: conclusion}
\input{tex/conclusion.tex}

\section*{Acknowledgement}
This work was supported by the Air Force Office of Scientific Research under award number FA9550-
20-1-0397 and additional support is gratefully acknowledged from NSF 2229012, 2312204, and ONR N00014-24-1-2655.

\bibliographystyle{plainnat}
\bibliography{refer.bib}

\end{document}

%% file: tex/math_commands.tex
\def\1{\bm{1}}

\DeclareMathAlphabet{\mathsfit}{\encodingdefault}{\sfdefault}{m}{sl}
\SetMathAlphabet{\mathsfit}{bold}{\encodingdefault}{\sfdefault}{bx}{n}

\def\0{{\bf 0}}
\def\1{{\bf 1}}

\def\DM{{\mathcal D}}

\def\OM{{\mathcal O}}

\def\RB{{\mathbb R}}
\def\EB{{\mathbb E}}
\def\ZB{{\mathbb Z}}
\def\PB{{\mathbb P}}

\def\idx{\mathrm{id}_{\RB_{+}}}

\newcommand{\Var}{\mathrm{Var}}



%% file: tex/intro.tex
Stochastic simulation is widely used to estimate performance measures that are analytically intractable, including steady-state means, expected delays, and the outputs of stochastic approximation algorithms. A basic practical question is how long a simulation should be run. If the running length is fixed in advance, then the computational budget is known, but the precision of the resulting estimator is not. The final confidence interval may be too wide to be useful, or unnecessarily narrow because the simulation was continued beyond the required accuracy. This tension motivates sequential stopping rules, which monitor statistical precision during the simulation and terminate once a prescribed accuracy has been reached.

We begin with the fixed-sample-size problem. Let $\{X_i\}_{i\in\ZB}$ denote the output of a stochastic simulation, and suppose that the goal is to estimate its stationary mean $\mu=\EB[X_1]$ using the time average
\begin{align}
    \bar X_n=n^{-1}\sum_{i=1}^nX_i.
\end{align}
When the simulation output has finite variance and short-range dependence, the classical central limit theorem gives
\begin{align}
    \sqrt{n}\left(\bar X_n-\mu\right)\tod N(0,\sigma^2),
    \label{eq: intro_clt}
\end{align}
where $\sigma^2$ is the long-run variance, given by the sum of the autocovariances of the stationary output process. If $\widehat{\sigma}_n$ is a consistent estimator of $\sigma$, then an approximate $(1-\delta)$-level confidence interval for $\mu$ is
\begin{align}
    \text{CI}_n
    =
    \left[
        \bar X_n-q_{1-\frac{\delta}{2}}\frac{\widehat{\sigma}_n}{\sqrt n},
        \bar X_n+q_{1-\frac{\delta}{2}}\frac{\widehat{\sigma}_n}{\sqrt n}
    \right],
    \label{eq: intro_ci}
\end{align}
where $q_{1-\frac{\delta}{2}}$ is the $\left(1-\frac{\delta}{2}\right)$-quantile of the standard normal distribution. The confidence interval is asymptotically valid in the sense that
\begin{align}
    \PB(\mu\in \text{CI}_n)\to 1-\delta
    \qquad\text{as }n\to\infty.
    \label{eq: intro_fixed_validity}
\end{align}
Consistent long-run variance estimators under various dependence conditions have been studied extensively; see, for example, \cite{newey1986simple,flegal2010batch,politis2011higher,zhang2012inference}.

These fixed-sample-size results treat $n$ as chosen in advance. They control coverage at that $n$, but the width of the resulting confidence interval remains random; consequently, they do not tell the practitioner how long the simulation must run to achieve a prescribed accuracy. In many simulation studies, however, accuracy is the natural input and computational effort should adapt to the realized output. A fixed-width procedure therefore reverses the roles of sample size and precision: it specifies a target width $\varepsilon$ and continues the simulation until the confidence interval is sufficiently narrow. Let $\text{CI}_n$ be a confidence interval constructed from the first $n$ simulation outputs, and define the fixed-width stopping time
\begin{align}
    T(\varepsilon)=\inf\{n\ge n_0(\varepsilon):\operatorname{width}(\text{CI}_n)\le\varepsilon\}.
    \label{eq: intro_stop}
\end{align}
Here $n_0(\varepsilon)\to\infty$ is a minimum sample size that prevents premature termination due to unstable early estimates. By construction, the procedure attains the target width. We seek to establish that the confidence interval at termination also retains its nominal coverage:
\begin{align}
    \PB\bigl(\mu\in \text{CI}_{T(\varepsilon)}\bigr)\to 1-\delta
    \qquad\text{as }\varepsilon\to0.
    \label{eq: intro_seq_validity}
\end{align}
This is the desired asymptotic validity. It does not follow automatically from fixed-sample-size validity because $T(\varepsilon)$ is random and is determined from the same simulation output used to construct the final confidence interval. Consequently, a fixed-sample-size limit theorem cannot in general be evaluated at $T(\varepsilon)$ without additional process-level arguments.

The literature on sequential stopping procedures has developed over several decades. Sequential analysis was introduced by \citet{wald2004sequential} to improve sample efficiency in hypothesis testing. In stochastic simulation, early studies \citep{fishman1977achieving,law1979sequential,lavenberg1977sequential} provided empirical evidence for sequential procedures in steady-state and regenerative simulation, with applications including queueing and inventory systems. On the theoretical side, \citet{anscombe1952large,chow1965asymptotic,starr1966performance,nadas1969extension} laid the foundations for fixed-width sequential confidence intervals for the mean of i.i.d.\ random variables. \citet{glynn1992asymptotic} subsequently established general conditions for the asymptotic validity of sequential stopping rules for a broad class of simulation estimators, including steady-state simulation. Later work extended fixed-width procedures to Markov chain Monte Carlo \citep{jones2006fixed,flegal2015relative}. Classical stopping rules monitor the size of a confidence region and therefore require a strongly consistent estimator of the relevant variance or scaling quantity, which may be difficult to construct. To reduce this reliance, \citet{dong2019new} proposed a procedure based on independent replications, using variation across groups in place of a conventional variance estimator. \citet{dong2019asymptotic} used standardized time series to construct confidence regions for steady-state simulation, and related refinements were developed by \citet{alexopoulos2020steady,lolos2022sequential}.

Many stochastic simulations, however, do not fall into the classical finite-variance, short-range dependent setting. Heavy-tailed simulation noise can produce output with infinite variance. Long-range dependence can arise in network traffic, queueing systems, and other simulations with persistent temporal effects. In these settings, the square-root scaling and the classical Gaussian approximation may fail. Instead, for an appropriate normalization $a_n$, one may have
\begin{align}
    a_n(\bar X_n-\mu)\tod Z,
\end{align}
where $a_n$ may depend on unknown nuisance quantities, such as a tail index, a Hurst parameter, or a slowly varying function, and $Z$ may be a stable or other nonstandard limit. A substantial fixed-sample-size literature develops valid inference in such cases. Classical bootstrap procedures can fail under infinite variance \citep{athreya1987bootstrap,knight1989bootstrap}, while smaller bootstrap samples and subsampling can restore validity under suitable conditions \citep{kinateder11992invariance,arcones1989bootstrap,arcones1991additions,wu1990bootstrapping,romano1999subsampling}. Self-normalization and subsampling can also eliminate unknown nuisance quantities and cover both heavy-tailed and long-range dependent processes \citep{fan2010statistical,mcelroy2013distribution,bai2016unified}. Thus, methods are available for constructing asymptotically valid confidence intervals at a sample size fixed in advance, but these fixed-sample-size results do not by themselves provide a fixed-width sequential procedure that remains valid at termination.

Extending this fixed-sample-size validity to fixed-width stopping is not automatic. In classical settings, a consistent scale estimator makes the properly normalized stopping time converge to a deterministic limit, allowing a random-time-change argument to transfer the fixed-sample-size limit to the stopping time. This mechanism can fail when validity is obtained through random self-normalization. Under infinite variance, for example, the partial-sum process and its quadratic self-normalizer can converge jointly to dependent stable processes, so the scale does not stabilize around a deterministic variance. Under long-range dependence, the normalization may involve an unknown Hurst parameter or slowly varying function, and a consistent long-run variance estimator may be unavailable or may have the wrong order. In these cases, the properly normalized stopping time can have a nondegenerate random limit determined by the limiting scaling process. The estimation process and the stopping time therefore remain asymptotically dependent, and quantiles obtained from a fixed-sample-size limit distribution need not provide the desired coverage at termination.

In this paper, we develop a unified sequential stopping rule for this broader class of stochastic simulations. Our starting point is a joint functional limit theorem for an estimation process and a data-dependent scaling process. We derive the asymptotic behavior of the resulting stopping time and of the self-normalized estimator evaluated at termination. However, the limiting distribution at termination generally depends on the unknown data-generating mechanism and is not available in closed form. To make the stopping rule practical, we introduce a sequential subsampling procedure that repeatedly runs shorter stopping experiments while the main simulation is in progress. The empirical distribution of these completed subprocedures consistently estimates the distribution relevant at the main stopping time.

We provide several constructions of the scaling process. For the classical i.i.d. sequence case, the empirical sum of squares yields a common rule in both finite- and infinite-variance settings. When this normalization does not match the order of the estimation error, we construct alternatives based on independent sectioning, batch means, and random scaling.  Together, these results give a single sequential methodology across simulation settings in which the convergence rate, limiting process, and existence of a second moment may all be unknown.

The rest of the paper is organized as follows. Section~\ref{sec: pre} reviews the fixed-sample-size self-normalization and subsampling ideas that motivate our construction. Section~\ref{sec: frame} presents the general sequential stopping rule and establishes its asymptotic validity and feasible calibration. Section~\ref{sec: alt} develops several choices of the scaling process. Section~\ref{sec: eg} verifies the framework in a collection of stochastic simulation models. Section~\ref{sec: proofs} gives the proofs of the results. Section~\ref{sec: conclusion} concludes the paper.

\paragraph{Notation.} For a sequence of stochastic processes $Z_n(\cdot)$ converging weakly to $Z(\cdot)$ in a specified topology, we write $Z_n(\cdot)\Rightarrow Z(\cdot)$. For random variables, $\tod$, $\topb$, and $\toas$ denote convergence in distribution, in probability, and almost surely, respectively. We write $\lfloor x\rfloor$ for the greatest integer less than or equal to $x$. We write $\mathbb D:=D([0,\infty),\RB)$ and $\mathbb D_+:=D([0,\infty),\RB_+)$. For random variables $X$ and $Y$, their Kolmogorov distance is $d_{\mathrm{Kol}}(X,Y)=\sup_t|\PB(X\le t)-\PB(Y\le t)|$. For $i\in\{1,2\}$, $\mathrm{WM}_i$ denotes the weak $M_i$ topology, namely the product of the univariate $M_i$ topologies for the two coordinates.

%% file: tex/pre.tex
Before diving into our main results, we first consider the standard mean estimation problem with i.i.d.\ samples of infinite variance. We review the inference methodology for the fixed-sample-size setting and highlight the challenge of extending it to fixed-width sequential stopping. Let $\{X_i\}_{i\ge1}$ be i.i.d.\ simulation outputs with mean $\mu:=\EB[X_1]$. We impose the following domain-of-attraction condition.
\begin{asmp}
    \label{asmp: domain}
    The distribution of $X_1$ is in the domain of attraction of an
    $\alpha$-stable law; that is, there exist sequences $s_n>0$ and $d_n\in\RB$
    such that:
    \begin{align}
        \frac{\sum_{i=1}^nX_i}{s_n}-d_n\tod Z_{\alpha},
    \end{align}
    where $Z_{\alpha}$ is an $\alpha$-stable random variable.
\end{asmp}
An immediate example for Assumption~\ref{asmp: domain} is the power law distribution. In a more general case, the distribution can be characterized as $\PB(|X_1|\ge t)=\frac{L(t)}{t^{\alpha}}$, where $L(t)$ is a slowly varying function at $+\infty$. And it is easy to check that $\EB|X_1|^\alpha$ can be infinite or finite. In this subsection, we assume $\alpha\in(1,2)$, which implies the samples $\{X_i\}_{i\ge1}$ are of infinite variance but finite mean. Under Assumption~\ref{asmp: domain}, a central limit theorem holds by \cite{geluk2000stable}:
\begin{align}
    n^{-\frac{1}{\alpha}}\ell(n)
    \sum_{i=1}^n \left(X_i-\mu\right)\tod Z_\alpha,
\end{align}
where $\ell(\cdot)$ is a slowly varying function dependent with $L(\cdot)$. Then, it is possible to establish a confidence interval for $\mu$ as long as $\alpha$, $\ell(\cdot)$ and quantiles of $Z_\alpha$ are known. To achieve that, the additional characteristics for $Z_{\alpha}$ such as skewness, scale and location shift are required to be estimated via point estimation, which is extremely complicated. However, a more efficient way to establish the confidence interval for $\mu$ is self-normalization. By \cite{logan1973limit}, under Assumption~\ref{asmp: domain}, the joint central limit theorem holds:
\begin{align}
    \label{eq: clt}
    \left(n^{-\frac{1}{\alpha}}\ell(n)
    \sum_{i=1}^n \left(X_i-\mu\right),
    n^{-\frac{2}{\alpha}}\ell(n)^2
    \sum_{i=1}^n(X_i-\mu)^2\right)
    \tod \left(Z_{\alpha}, Z_{\frac{\alpha}{2}}\right),
\end{align}
where $Z_{\frac{\alpha}{2}}$ is an $\frac{\alpha}{2}-$stable random variable and has some implicit dependence with $Z_{\alpha}$. Thus, replacing $\mu$ by $\bar{X}_n=\frac{1}{n}\sum_{i=1}^nX_i$ in the second component of \eqref{eq: clt} and applying self-normalization yields
\begin{align}
    \label{eq: self}
    \frac{\sqrt{n}\left(\bar{X}_n-\mu\right)}{\sqrt{\frac{1}{n-1}\sum_{i=1}^n(X_i-\bar{X}_n)^2}}\tod\frac{Z_{\alpha}}{\sqrt{Z_{\frac{\alpha}{2}}}}=:W_{\alpha}.
    \end{align}
Let $F_\alpha$ denote the distribution function of $W_\alpha$. By \eqref{eq: self}, one could skip estimating the index $\alpha$ and slowly varying function $\ell(\cdot)$ and construct an approximate $(1-\delta)$-confidence interval of $\mu$ by:
\begin{align}
    \left[\bar{X}_n-q_{1-\frac{\delta}{2}}\frac{\widehat{\sigma}_n}{\sqrt{n}},\bar{X}_n-q_{\frac{\delta}{2}}\frac{\widehat{\sigma}_n}{\sqrt{n}}\right],
    \label{eq: ci}
\end{align}
where $\widehat{\sigma}_n:=\sqrt{\frac{1}{n-1}\sum_{i=1}^n(X_i-\bar{X}_n)^2}$ and $q_x$ is the $x-$quantile of $F_{\alpha}$. Then, the only thing left to be estimated is the quantiles $q_{1-\frac{\delta}{2}}$ and $q_{\frac{\delta}{2}}$. However, the limit distribution $F_{\alpha}$ has a complicated dependence with several nuisance parameters. And it is also difficult to characterize the dependence of $Z_{\alpha}$ and $Z_{\frac{\alpha}{2}}$. Instead, \cite{romano1999subsampling} proposed a sub-sampling approach to estimate the quantiles of $F_{\alpha}$. We detail the procedure in the following:
\paragraph{Sub-sampling \citep{romano1999subsampling}:}
\begin{enumerate}
    \item[(a)] Given the dataset $\DM=\{X_1, X_2,\cdots, X_n\}$, construct $N_{n,b}=\tbinom{n}{b}$ subsets $\DM_1,\cdots, \DM_{N_{n,b}}$.
    \item[(b)] Calculate $\bar{X}_b^{(i)}$ and $\widehat{\sigma}_b^{(i)}$ for each subset $\DM_i$, $i=1,2,\cdots,N_{n,b}$.
    \item[(c)] Calculate the quantile by $\widehat{q}_x=\inf\left\{t\left|\frac{1}{N_{n,b}}\sum_{i=1}^{N_{n,b}}\mathbbm{1}\left(\frac{\sqrt{b}(\bar{X}_b^{(i)}-\bar{X}_n)}{\widehat{\sigma}_b^{(i)}}\le t\right)\ge x\right.\right\}$.
    \item[(d)] Output the confidence interval $\left[\bar{X}_n-\widehat{q}_{1-\frac{\delta}{2}}\frac{\widehat{\sigma}_n}{\sqrt{n}},\bar{X}_n-\widehat{q}_{\frac{\delta}{2}}\frac{\widehat{\sigma}_n}{\sqrt{n}}\right]$.
\end{enumerate}
Under standard conditions, $\widehat q_x$ consistently estimates $q_x$, so the confidence interval in \eqref{eq: ci} has asymptotic coverage $1-\delta$ when $n$ is fixed in advance. Its width, however, remains random. A fixed-width procedure instead allows the sample size to adapt to the simulation output and stops at
\begin{align}
    T(\varepsilon)=\inf\left\{n\ge n_0(\varepsilon):
    \left(\widehat{q}_{1-\frac{\delta}{2}}-\widehat{q}_{\frac{\delta}{2}}\right)
    \frac{\widehat{\sigma}_n}{\sqrt{n}}\le\varepsilon\right\},
    \label{eq: pre_stop}
\end{align}
where $n_0(\varepsilon)\to\infty$ is a minimum sample size. The goal is for the confidence interval evaluated at $T(\varepsilon)$ to retain asymptotic coverage $1-\delta$ as $\varepsilon\to0$. This property holds under classical finite-variance conditions \citep{glynn1992asymptotic}, but can fail under Assumption~\ref{asmp: domain}. Indeed, under Assumption~\ref{asmp: domain}, the self-normalizer $\widehat{\sigma}_n$ has a nondegenerate limit distribution under the appropriate scaling. As a consequence, the correspondingly scaled stopping time $T(\varepsilon)$ converges weakly to a nondegenerate random variable rather than to a deterministic constant. The self-normalized statistic evaluated at $T(\varepsilon)$ therefore need not have the fixed-sample-size limit $W_\alpha$. For the same reason, the subsampling procedure above does not directly yield a confidence interval with correct coverage at the random stopping time. In the next section, we develop the joint process-level limit theory and sequential subsampling procedure needed to estimate this random-time distribution.

%% file: tex/frame.tex
This section develops a general framework for sequential stopping without
imposing a particular data-generating mechanism. We represent the cumulative
estimator and its scaling statistic by two stochastic processes
$Z_1(\cdot)$ and $Z_2(\cdot)$, respectively, and assume their joint
functional convergence. 

\subsection{Assumptions}
\label{sec: stop_time_asymptotics}
\begin{asmp}
\label{asmp: fclt}
    Let $Z_1(\cdot)$ and $Z_2(\cdot)$ have sample paths in $\mathbb D$ and
    $\mathbb D_+$, respectively, with $Z_1(0)=0$ almost surely. There exist a constant $h\in\RB$, an eventually
    positive slowly varying function $\ell(\cdot)$ at infinity, and limiting
    processes $Y_1(\cdot)\in\mathbb D$ and $Y_2(\cdot)\in\mathbb D_+$ such that, as
    $\varepsilon\to0^+$,
    \begin{align}
    \label{eq: fclt}
        \left(\varepsilon^{h}\ell(\varepsilon^{-1})\left(Z_1\left(\frac{\cdot}{\varepsilon}\right)-\idx\left(\frac{\cdot}{\varepsilon}\right)\mu\right),\varepsilon^{h}\ell(\varepsilon^{-1}) \sqrt{Z_2\left(\frac{\cdot}{\varepsilon}\right)}\right)
        \overset{\mathrm{WM}_2}{\Rightarrow}
        \left(Y_1(\cdot), \sqrt{Y_2(\cdot)}\right).
    \end{align}
    Let $\mathbb F^Y=(\mathcal F_t^Y)_{t\ge0}$ be the usual augmentation of the natural filtration generated by $(Y_1,Y_2)$. Moreover, the following conditions are satisfied:
    \begin{itemize}
        \item[(a)] $0<h<1$.
        \item[(b)] As $t\to\infty$,
        $\frac{Z_1(t)-\mu t}{t}\to0$ and
        $\frac{Z_2(t)}{t^2}\to0$ almost surely.
        \item[(c)] The pair $(Y_1(\cdot),Y_2(\cdot))$ is jointly self-similar with indices $h$ and $2h$, respectively: for every $a>0$,
        \begin{align}
            \left(Y_1(a\,\cdot),Y_2(a\,\cdot)\right)
            \overset{d}{=}
            \left(a^hY_1(\cdot),a^{2h}Y_2(\cdot)\right).
        \end{align}
        \item[(d)] The process $Y_1(\cdot)$ is quasi-left-continuous with respect to $\mathbb F^Y$; that is, for every finite $\mathbb F^Y$-predictable stopping time $\tau$,
        \begin{align}
            \PB\bigl(\Delta Y_1(\tau)=0\bigr)=1.
        \end{align}
        \item[(e)] Almost surely, $Y_2(t)>0$ for every $t>0$, and $Y_2(\cdot)$ has no negative jumps.
    \end{itemize}
\end{asmp}
The common normalization in~\eqref{eq: fclt} places $Z_1-\mu\idx$ and $\sqrt{Z_2}$ on the same asymptotic scale, so their common nuisance normalization cancels in the self-normalized statistic. Condition~(b) ensures that the stopping rule is well-defined. Condition~(c), together with~(a), makes the distribution of the terminal self-normalized statistic invariant under the induced time rescaling. Condition~(d) rules out a jump of $Y_1$ at the predictable limiting stopping time, allowing random-time evaluation. Finally, condition~(e) ensures positivity of the limiting denominator at strictly positive times, while the absence of negative jumps ensures continuity of $Y_2$ at a positive finite first downcrossing. Sections~\ref{sec: alt} and~\ref{sec: eg} provide several constructions and examples for which these conditions can be verified.

Fix a deterministic coefficient $c>0$. Stopping as soon as
$c\sqrt{\frac{Z_2(n)}{n^2}}<\varepsilon$ may lead to premature termination if the
scaling statistic is small early in the simulation. To ensure that the
simulation length diverges as $\varepsilon\to0$, we impose a deterministic
lower bound and define
\begin{align}
\label{eq: stop1}
     T_c(\varepsilon)
     :=\inf\left\{n\in\mathbb N:n>\varepsilon^{-\nu},
     \quad c\sqrt{\frac{Z_2(n)}{n^2}}<\varepsilon\right\},
\end{align}
where $0<\nu<\frac{1}{1-h}$. Define
\begin{align*}
    \Phi_c(u)
    :=\inf\left\{s>0:
    c\frac{\sqrt{Y_2(s)}}{s}<u\right\},
    \qquad u>0.
\end{align*}
Throughout, the infimum of the empty set is understood to be $+\infty$.
Assumption~\ref{asmp: fclt}(b) ensures that
$T_c(\varepsilon)<\infty$ almost surely for every $c,\varepsilon>0$.
The strict lower bound in~\eqref{eq: stop1} also makes
$u\mapsto T_c\left(\frac{\varepsilon}{u}\right)$ right-continuous.
To identify the natural time scale of the stopping rule, let
\begin{align*}
    A(x):=x^{1-h}\ell(x),
    \qquad
    a_\varepsilon:=A^{\leftarrow}(\varepsilon^{-1}),
\end{align*}
where $A^{\leftarrow}$ is an asymptotic inverse of $A$. Since
$A$ is regularly varying with positive index $1-h$, the restriction
$\nu<\frac{1}{1-h}$ implies
$\frac{\varepsilon^{-\nu}}{a_\varepsilon}\to0$. Thus, although the deterministic
lower bound forces the simulation length to diverge, it is negligible
relative to the natural stopping scale. The functional convergence in
Assumption~\ref{asmp: fclt} controls the scaling statistic on normalized
time intervals bounded away from zero, but does not by itself rule out a
crossing at a vanishing fraction of $a_\varepsilon$. We therefore impose
the following small-time condition.
\begin{asmp}
\label{asmp: early_stopping}
    For every $K>0$,
    \begin{align}
        \lim_{\eta\downarrow0}\limsup_{\varepsilon\downarrow0}
        \PB\left(
        \inf_{\substack{n\in\mathbb N:\,
        \varepsilon^{-\nu}\le n\le \eta a_\varepsilon}}
        \frac{\sqrt{Z_2(n)}}{n}
        \le K\varepsilon
        \right)=0.
        \label{eq: early_stopping}
    \end{align}
\end{asmp}
Assumption~\ref{asmp: early_stopping} ensures that no fixed-coefficient rule stops before a vanishing fraction of the natural stopping scale.
Together with Assumption~\ref{asmp: fclt}, it allows the inverse-mapping
argument to be localized away from time zero. The following theorem gives
the asymptotic behavior of $T_c(\varepsilon)$.
\begin{thm}
    \label{thm: stop_prop}
    Under Assumptions~\ref{asmp: fclt} and~\ref{asmp: early_stopping}, the
    following conclusions hold for every fixed $c>0$:
    \begin{enumerate}
        \item[(a)] As $\varepsilon\to0^+$,
        $a_\varepsilon^{-1}T_c\left(\frac{\varepsilon}{\cdot}\right)
        \overset{M_1}{\Rightarrow}\Phi_c\left(\frac{1}{\cdot}\right)$
        locally on $(0,\infty)$.
        \item[(b)] For any $t>0$ and $a>0$, $\Phi_c(at)\overset{d}{=}a^{\frac{1}{h-1}}\Phi_c(t)$. Moreover, $\Phi_c(t)<\infty$ almost surely.
        \item[(c)] $\varepsilon^\gamma T_c(\varepsilon)\topb0$ whenever $\gamma>\frac{1}{1-h}$.
        \item[(d)] For every $u>0$ at which $\Phi_c(\cdot)$ is almost surely continuous, $\Phi_c(u)$ is an $\mathbb F^Y$-predictable stopping time.
    \end{enumerate}
\end{thm}

\subsection{Asymptotic Validity at Stopping Time}
\label{sec: stopping_validity}
\noindent Theorem~\ref{thm: stop_prop} identifies the limiting stopping time
$\Phi_c(1)$, which need not be deterministic. For the classical i.i.d.\
finite-variance rule based on a consistent variance estimator, the normalized
stopping time instead converges to a deterministic constant. In that setting,
replacing a deterministic sample size by the stopping time does not alter the
limiting distribution of the self-normalized estimator
\cite{glynn1992asymptotic,chow1965asymptotic}. When the limiting stopping time
is random, the terminal distribution must be obtained by evaluating the joint
limit in Assumption~\ref{asmp: fclt} at this random time.
Whenever a terminal self-normalized statistic below has a zero denominator,
it is defined to be zero.
\begin{thm}
\label{thm: asymp_valid}
    Under Assumptions~\ref{asmp: fclt} and~\ref{asmp: early_stopping}, suppose that $\Phi_1(1)>0$ almost surely. Fix $c>0$ and suppose that $\Phi_c(\cdot)$ is almost surely continuous at $1$. Then, as $\varepsilon\to0^+$,
    \begin{align}
             \frac{Z_1(T_c(\varepsilon))-T_c(\varepsilon)\mu}{\sqrt{Z_2(T_c(\varepsilon))}}
             \tod
             \frac{Y_1(\Phi_c(1))}{\sqrt{Y_2(\Phi_c(1))}}
             =:W_c.
             \label{eq: limit_dist}
    \end{align}
    The distribution of $W_c$ does not depend on $c$. Fix
    $\delta\in(0,1)$. Let $F$ denote its
    common distribution function, and define
    \begin{align*}
        c_l=F^{\leftarrow}\left(\frac{\delta}{2}\right),\qquad
        c_u=F^{\leftarrow}\left(1-\frac{\delta}{2}\right),\qquad
        c^*=c_u-c_l,
    \end{align*}
    and suppose that $c_l<c_u$, $F$ is continuous at $c_l$ and $c_u$, and $\Phi_{c^*}(\cdot)$ is almost surely continuous at $1$. Then the confidence interval
    \begin{align}
    \label{eq: cover}
        \text{CI}(\varepsilon)
        :=\left[
        \frac{Z_1(T_{c^*}(\varepsilon))}{T_{c^*}(\varepsilon)}
        -c_u\sqrt{\frac{Z_2(T_{c^*}(\varepsilon))}{T_{c^*}(\varepsilon)^2}},
        \frac{Z_1(T_{c^*}(\varepsilon))}{T_{c^*}(\varepsilon)}
        -c_l\sqrt{\frac{Z_2(T_{c^*}(\varepsilon))}{T_{c^*}(\varepsilon)^2}}
        \right]
    \end{align}
    satisfies
    \begin{align}
        \PB\left(\mu\in\text{CI}(\varepsilon)\right)
        \to1-\delta.
        \label{eq: oracle_coverage}
    \end{align}
\end{thm}
Since $c^*=c_u-c_l$, the width criterion for $\text{CI}(\varepsilon)$ is exactly the criterion monitored by $T_{c^*}(\varepsilon)$. Theorem~\ref{thm: asymp_valid} is an oracle result: its stopping coefficient and quantiles are deterministic. The next subsection estimates these quantities online and establishes the validity of the resulting feasible procedure.

\subsection{Sequential Subsampling}
\label{sec: sequential_subsampling}
The oracle result in Theorem~\ref{thm: asymp_valid} depends on unknown
quantiles of the random-time limit. Ordinary fixed-sample-size subsampling
does not directly estimate this distribution. The invariance of
$W_c$ is essential here: since $W_c\overset{d}{=}W_1$ for every $c>0$, all
local procedures may use the fixed coefficient $1$, while their terminal
statistics still target the same distribution $F$. We therefore run a sequence of
shorter, independent stopping procedures while the main simulation is in
progress and use their terminal statistics to estimate the required
quantiles. The estimates are updated whenever a local procedure is completed.
Algorithm~\ref{alg: general} summarizes the resulting online procedure.
\begin{algorithm}[H]
    \caption{Sequential stopping with online subsampling}
    \label{alg: general}
    \begin{algorithmic}[1]
    \REQUIRE Accuracy level $\varepsilon>0$, target coverage $1-\delta$ with
    $\delta\in(0,1)$,
    exponents $r,\nu$ satisfying $0<r<1$ and
    $0<\nu<\frac{1}{1-h}$, and lower bound $\underline{c}>0$.
    \STATE Initialize the current local-procedure index $J=1$, its current
    sample size $m=0$, and an independent local simulator.
    \FOR{$n=1,2,\ldots$}
        \STATE Generate $X_n$ and update the global statistics $Z_1(n)$ and $Z_2(n)$.
        \STATE Generate $X_{m+1}^{(J)}$ from the current local simulator, set
        $m\leftarrow m+1$, and update $Z_1^{(J)}(m)$ and $Z_2^{(J)}(m)$.
        \IF{$m> \varepsilon^{-r\nu}$ and
        $\sqrt{\frac{Z_2^{(J)}(m)}{m^2}}<\varepsilon^{r}$}
            \STATE Set $T^{(J)}(\varepsilon^r)=m$ and store the corresponding
            terminal statistics.
            \STATE Set $J\leftarrow J+1$ and $m\leftarrow0$, and initialize a
            new independent local simulator.
        \ENDIF            
        \IF{$J\ge2$}
            \FOR{$j=1,2,\ldots,J-1$}
                \STATE Compute $\displaystyle \omega_{j,n}(\varepsilon)=
                \frac{Z_1^{(j)}(T^{(j)}(\varepsilon^r))-
                \frac{T^{(j)}(\varepsilon^r)}{n}Z_1(n)}
                {\sqrt{Z_2^{(j)}(T^{(j)}(\varepsilon^r))}}$.
            \ENDFOR
            \STATE Set $\displaystyle \widehat{F}_n(x;\varepsilon)=
            \frac{1}{J-1}\sum_{j=1}^{J-1}
            \mathbbm{1}\{\omega_{j,n}(\varepsilon)\le x\}$.
            \STATE Set $\displaystyle c_{l,n}(\varepsilon)=\inf\{x:\widehat F_n(x;\varepsilon)\ge\frac{\delta}{2}\}$ and $\displaystyle c_{u,n}(\varepsilon)=\inf\{x:\widehat F_n(x;\varepsilon)\ge1-\frac{\delta}{2}\}$.
            \STATE Set $c_n(\varepsilon)=\max\{\underline c,c_{u,n}(\varepsilon)-c_{l,n}(\varepsilon)\}$.
            \IF{$n>\varepsilon^{-\nu}$ and $\sqrt{\frac{Z_2(n)}{n^2}}<\frac{\varepsilon}{c_n(\varepsilon)}$}
                \STATE Set $T(\varepsilon)=n$ and $B(\varepsilon)=J-1$, and terminate.
            \ENDIF
        \ENDIF
    \ENDFOR
    \end{algorithmic}
\end{algorithm}
On nontermination, the normalized stopping time and terminal quantities in
the conclusions below are assigned fixed finite values. This convention is
asymptotically irrelevant under the conditions below.
\begin{thm}
\label{thm: consistency}
    In Algorithm~\ref{alg: general}, suppose that the following conditions hold:
    \begin{enumerate}
        \item[(a)] Assumptions~\ref{asmp: fclt} and~\ref{asmp: early_stopping} hold for the global and local simulation outputs;
        \item[(b)] the local simulation streams are independent copies and are independent of the global stream;
        \item[(c)] the common distribution function $F$ in
        Theorem~\ref{thm: asymp_valid} is continuous, and its
        $\frac{\delta}{2}$ and $\left(1-\frac{\delta}{2}\right)$ quantiles $c_l$ and $c_u$ are uniquely
        identified: for every $\xi>0$,
        \begin{align*}
            F(c_l-\xi)&<\frac{\delta}{2}<F(c_l+\xi),\qquad
            F(c_u-\xi)<1-\frac{\delta}{2}<F(c_u+\xi);
        \end{align*}
        \item[(d)] $r\in\left(0,1\right)$;
        \item[(e)] $0<\underline c<c_u-c_l$.
    \end{enumerate}
    Assume that $\Phi_1(1)>0$ almost surely and that $\Phi_1(\cdot)$ is almost surely continuous at $1$ and $(c^*)^{-1}$. Then, as $\varepsilon\to0^+$,
    \begin{align}
        \bigl(c_{l,T(\varepsilon)}(\varepsilon),c_{u,T(\varepsilon)}(\varepsilon)\bigr)
        \topb(c_l,c_u).
        \label{eq: terminal_quantile_consistency}
    \end{align}
    Define the feasible confidence interval
    \begin{align}
        \widehat{\text{CI}}(\varepsilon)
        :=\left[
        \frac{Z_1(T(\varepsilon))}{T(\varepsilon)}
        -c_{u,T(\varepsilon)}(\varepsilon)\sqrt{\frac{Z_2(T(\varepsilon))}{T(\varepsilon)^2}},
        \frac{Z_1(T(\varepsilon))}{T(\varepsilon)}
        -c_{l,T(\varepsilon)}(\varepsilon)\sqrt{\frac{Z_2(T(\varepsilon))}{T(\varepsilon)^2}}
        \right].
        \label{eq: feasible_ci}
    \end{align}
    Then
    \begin{align}
        \PB\bigl(\mu\in\widehat{\text{CI}}(\varepsilon)\bigr)
        \to1-\delta.
        \label{eq: feasible_coverage}
    \end{align}
\end{thm}
\begin{figure}[H]
    \centering
    \includegraphics[width=1.0\linewidth]{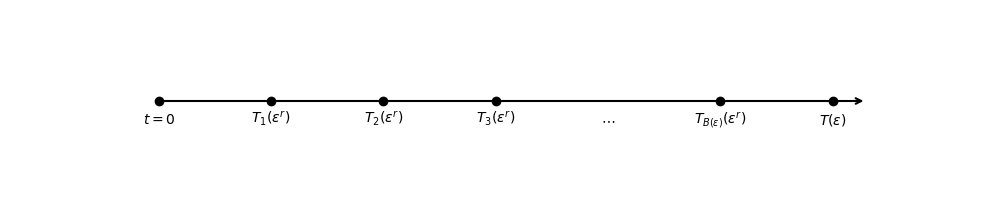}
    \caption{Sequential subsampling. While the global procedure runs at
    accuracy level $\varepsilon$, independent local procedures are run
    sequentially at the coarser accuracy level $\varepsilon^r$, where
    $0<r<1$. Their stopping times have a smaller asymptotic order than the
    global stopping time. A new local procedure starts immediately after the
    preceding one terminates, and $B(\varepsilon)$ denotes the number completed
    by the time the global procedure stops.}
    \label{fig: sub}
\end{figure}
Theorem~\ref{thm: consistency} shows that sequential subsampling consistently
estimates the terminal quantiles and yields an asymptotically valid confidence
interval. The constraint $c_n(\varepsilon)\ge\underline c$ controls unstable
early estimates of the quantile width, while condition~(e) ensures that it is
asymptotically inactive. 

%% file: tex/alt.tex
This section develops several choices of the scaling process $Z_2(\cdot)$ for the general sequential framework established in the last section. We first consider the empirical sum of squares and then introduce sectioning, batch means, and random scaling as alternative constructions.
\subsection{Sum of Squares}
\label{sec: ss}
Consider estimation of the mean $\mu$ of an i.i.d.\ sequence
$\{X_i\}_{i\ge1}$ and define
\begin{align}
    Z_1(t)&:=\sum_{i=1}^{\lfloor t\rfloor}X_i,
    \label{eq: ss1}\\
    Z_2(t)&:=
    \begin{cases}
        \displaystyle\sum_{i=1}^{\lfloor t\rfloor}
        \left(X_i-\frac{Z_1(t)}{\lfloor t\rfloor}\right)^2,
        &t\ge1,\\
        0,&0\le t<1.
    \end{cases}
    \label{eq: ss2}
\end{align}
Thus, $Z_1$ and $Z_2$ are the c\`adl\`ag step embeddings of the partial
sum and centered sum of squares. The following results verify
Assumptions~\ref{asmp: fclt} and~\ref{asmp: early_stopping} in both the
finite- and infinite-variance settings.
\begin{thm}
\label{thm: iid_finite}
    For an i.i.d. sequence $\{X_i\}_{i\ge1}$ with mean
    $\EB[X_1]=\mu$ and variance $\Var(X_1)=\sigma^2\in(0,\infty)$,
    Assumption~\ref{asmp: fclt} holds with
    $h=\frac{1}{2}$, $\ell\equiv1$,
    $Y_1(\cdot)=\sigma B(\cdot)$, and
    $Y_2(\cdot)=\sigma^2\idx(\cdot)$, where $B$ is a standard
    Brownian motion. Moreover,
    Assumption~\ref{asmp: early_stopping} holds for every
    $0<\nu<2$.
\end{thm}
When $X_1$ has infinite variance and its distribution satisfies
Assumption~\ref{asmp: domain}, the limiting processes and normalization
change, but the same conclusion holds.
\begin{thm}
\label{thm: iid_infinite}
    Let $\{X_i\}_{i\ge1}$ be i.i.d.\ with mean $\EB[X_1]=\mu$, and
    suppose that its common distribution satisfies
    Assumption~\ref{asmp: domain} with $\alpha\in(1,2)$. Then
    Assumption~\ref{asmp: fclt} holds with $h=\frac{1}{\alpha}$, a
    suitable slowly varying function $\ell$, $Y_1=L_1$, and $Y_2=L_2$,
    where $L_1$ is an $\alpha$-stable L\'evy process and $L_2$ is an
    $\frac{\alpha}{2}$-stable subordinator. Moreover,
    Assumption~\ref{asmp: early_stopping} holds for every
    $0<\nu<\frac{\alpha}{\alpha-1}$.
\end{thm}
It is worth noting that Algorithm~\ref{alg: general} exhibits no distinction between the finite-variance and infinite-variance cases, since the nuisance parameters $\alpha$ and $\ell(\cdot)$ are eliminated through self-normalization. Furthermore, the choice of $Z_2(\cdot)$ in~\eqref{eq: ss2} extends to short-range dependent sequences when the corresponding joint functional limit and regularity conditions hold. A detailed time-series example is given in Section~\ref{sec: eg}.

\subsection{Sectioning, Batch Means, and Random Scaling}
\label{sec: sec_bm}
Assumption~\ref{asmp: fclt} requires the estimation and scaling coordinates in~\eqref{eq: fclt} to use the same normalizing factor. However, in some settings, the empirical sum of squares in~\eqref{eq: ss2} does not match this normalization. This mismatch can arise when dependence causes the empirical sum of squares and the estimation error to have different stochastic orders, as illustrated by the M/G/1 example in Section~\ref{sec: eg}. We therefore construct alternative scaling processes directly from $Z_1$ by sectioning, batch means, and random scaling. For illustration, we can construct the scaling process as
\begin{align}
    Z_2(\cdot)=\mathcal L(Z_1)(\cdot),
    \label{eq: scaling_functional}
\end{align}
where a functional $\mathcal L$ in~\eqref{eq: scaling_functional} is chosen so that $\sqrt{Z_2}$ has the same scale as $Z_1(\cdot)-\mu \cdot$. Under this construction, $Z_2(\cdot)$ may have negative jumps, so a stopping rule based on $Z_2(\cdot)$ can terminate at a downward jump time. We therefore smooth $Z_2(\cdot)$ by time averaging:
\begin{align}
    Z_2^*(t)
    &:=\frac{1}{t}\int_0^t Z_2(u)\,\mathrm du,
    \qquad t>0,
    \label{eq: time_averaged_scaling}
\end{align}
with $Z_2^*(0):=0$. This smoothing preserves the original scaling: $Z_2^*$ has the same stochastic order as $Z_2$, and a $2h$-self-similar limit of $Z_2$ is mapped to a $2h$-self-similar limit. It can also be verified that the time average in~\eqref{eq: time_averaged_scaling} is continuous on $[0,\infty)$. The superscript $*$ denotes smoothing of the scaling process $Z_2$. To avoid some trivial cases such as the limiting process is zero with positive probability, we impose the following nondegeneracy condition for the limiting scaling processes throughout this subsection. 
\begin{align}
    \PB\left(Y_2^{*,\iota}(1)>0\right)=1, \qquad\iota\in\{\mathrm{sec},\mathrm{bm},\mathrm{rs}\}
    \label{eq: scaling_nondegeneracy}
\end{align}
By self-similarity and nondecreasing property of $tY_2^{*,\iota}(t)$, \eqref{eq: scaling_nondegeneracy} implies positivity at every $t>0$ almost surely for $Y_2^{*,\iota}(\cdot)$. 
\begin{thm}[Sectioning]
\label{thm: replication}
    Let $Z_1(\cdot)$ satisfy the marginal FCLT in~\eqref{eq: fclt} and the corresponding requirements of Assumption~\ref{asmp: fclt}, for some $\mu$, $h$, $\ell$, and $Y_1$.
    Let $Z_{1}^{(1)}(\cdot),\ldots,Z_{1}^{(m)}(\cdot)$ be $m\ge2$ independent copies of $Z_1(\cdot)$, with corresponding independent limits $Y_{1}^{(1)}(\cdot),\ldots,Y_{1}^{(m)}(\cdot)$, and suppose that the limiting vector is quasi-left-continuous with respect to its augmented natural filtration. Define
    \begin{gather}
        Z_1^{\text{sec}}(\cdot)=\bar Z_1(\cdot):=\frac{1}{m}\sum_{i=1}^m Z_1^{(i)}(\cdot),\\
        Z_2^{*,\text{sec}}(\cdot)=\frac{1}{m-1}\sum_{i=1}^m\frac{\int_0^{\cdot}\left(Z_{1}^{(i)}(s)-\bar{Z}_1(s)\right)^2ds}{\cdot},
    \end{gather}
    Then
    \begin{gather}
        \left(\varepsilon^{h}\ell(\varepsilon^{-1})\left(Z_1^{\text{sec}}\left(\frac{\cdot}{\varepsilon}\right)-\idx\left(\frac{\cdot}{\varepsilon}\right)\mu\right), \varepsilon^{h}\ell(\varepsilon^{-1})\sqrt{Z_2^{*,\text{sec}}\left(\frac{\cdot}{\varepsilon}\right)}\right)
        \overset{\mathrm{WM}_2}{\Rightarrow}
        \left(Y_1^{\text{sec}}(\cdot), \sqrt{Y_2^{*,\text{sec}}(\cdot)}\right),\\
        Y_1^{\text{sec}}(\cdot)=\bar Y_1(\cdot):=\frac{1}{m}\sum_{i=1}^mY_1^{(i)}(\cdot),\\
        Y_2^{*,\text{sec}}(\cdot)=\frac{1}{m-1}\sum_{i=1}^m\frac{\int_0^{\cdot}\left(Y_1^{(i)}(s)-\bar{Y}_1(s)\right)^2ds}{\cdot},
    \end{gather}
    where the time-averaged processes are defined to be zero at time zero.
    The process $Y_2^{*,\text{sec}}(\cdot)$ is continuous, and the pair
    $(Z_1^{\text{sec}},Z_2^{*,\text{sec}})$ satisfies
    Assumption~\ref{asmp: fclt}.
\end{thm}
\begin{thm}[Batch Means]
\label{thm: batch_mean}
    Let $Z_1(\cdot)$ satisfy the marginal FCLT in~\eqref{eq: fclt} and the corresponding requirements of Assumption~\ref{asmp: fclt}, for some $\mu$, $h$, $\ell$, and $Y_1$.
    Define
    \begin{gather}
        Z_1^{\text{bm}}(\cdot)=Z_1(\cdot),\\
        Z_2^{*,\text{bm}}(\cdot)=\frac{1}{m-1}\sum_{k=1}^m\frac{\int_0^{\cdot}\left(Z_{1,k}^{\mathrm{bm}}(s)-\frac{1}{m}Z_1(s)\right)^2ds}{\cdot},
    \end{gather}
    where $Z_{1,k}^{\mathrm{bm}}(\cdot)=Z_1\left(\frac{k}{m}\cdot\right)-Z_1\left(\frac{k-1}{m}\cdot\right)$ and $m\ge2$ is fixed. Then
    \begin{gather}
        \left(\varepsilon^{h}\ell(\varepsilon^{-1})\left(Z_1^{\text{bm}}\left(\frac{\cdot}{\varepsilon}\right)-\idx\left(\frac{\cdot}{\varepsilon}\right)\mu\right), \varepsilon^{h}\ell(\varepsilon^{-1})\sqrt{Z_2^{*,\text{bm}}\left(\frac{\cdot}{\varepsilon}\right)}\right)
        \overset{\mathrm{WM}_2}{\Rightarrow}
        \left(Y_1^{\text{bm}}(\cdot), \sqrt{Y_2^{*,\text{bm}}(\cdot)}\right),\\
        Y_1^{\text{bm}}(\cdot)=Y_1(\cdot),\\
        Y_2^{*,\text{bm}}(\cdot)=\frac{1}{m-1}\sum_{k=1}^m\frac{\int_0^{\cdot}\left(Y_{1}\left(\frac{k}{m}s\right)-Y_1\left(\frac{(k-1)}{m}s\right)-\frac{1}{m}Y_1(s)\right)^2ds}{\cdot}.
    \end{gather}
    The time-averaged processes are defined to be zero at time zero. The
    process $Y_2^{*,\text{bm}}(\cdot)$ is continuous, and the pair
    $(Z_1^{\text{bm}},Z_2^{*,\text{bm}})$ satisfies
    Assumption~\ref{asmp: fclt}.
\end{thm}
\begin{thm}[Random Scaling]
\label{thm: random_scaling}
    Let $Z_1(\cdot)$ satisfy the marginal FCLT in~\eqref{eq: fclt} and the corresponding requirements of Assumption~\ref{asmp: fclt}, for some $\mu$, $h$, $\ell$, and $Y_1$.
    Set $Z_1^{\text{rs}}(\cdot)=Z_1(\cdot)$ and define the pointwise random scaling and its time-smoothed version by
    \begin{align}
        Z_2^{\text{rs}}(t)
        &=\frac{1}{t}\int_0^t\left(Z_1(s)-\frac{s}{t}Z_1(t)\right)^2\,\mathrm ds,\\
        Z_2^{*,\text{rs}}(t)
        &=\frac{1}{t}\int_0^t Z_2^{\text{rs}}(u)\,\mathrm du,
        \qquad t>0,
    \end{align}
    with both processes defined to be zero at time zero. Then
    \begin{gather}
        \left(\varepsilon^{h}\ell(\varepsilon^{-1})\left(Z_1^{\text{rs}}\left(\frac{\cdot}{\varepsilon}\right)-\idx\left(\frac{\cdot}{\varepsilon}\right)\mu\right),
        \varepsilon^{h}\ell(\varepsilon^{-1})\sqrt{Z_2^{*,\text{rs}}\left(\frac{\cdot}{\varepsilon}\right)}\right)
        \overset{\mathrm{WM}_2}{\Rightarrow}
        \left(Y_1^{\text{rs}}(\cdot),\sqrt{Y_2^{*,\text{rs}}(\cdot)}\right),\\
        Y_1^{\text{rs}}(\cdot)=Y_1(\cdot),\\
        Y_2^{\text{rs}}(t)
        =\frac{1}{t}\int_0^t\left(Y_1(s)-\frac{s}{t}Y_1(t)\right)^2\,\mathrm ds,\qquad
        Y_2^{*,\text{rs}}(t)
        =\frac{1}{t}\int_0^t Y_2^{\text{rs}}(u)\,\mathrm du.
    \end{gather}
    Both limiting scaling processes are defined to be zero at time zero.
    The process $Y_2^{*,\text{rs}}(\cdot)$ is continuous, and the pair
    $(Z_1^{\text{rs}},Z_2^{*,\text{rs}})$ satisfies
    Assumption~\ref{asmp: fclt}.
\end{thm}
The sectioning construction requires $m$ independent simulation streams, the batch-means construction divides a single trajectory into $m$ sub-trajectories, and random scaling uses the centered bridge of a single trajectory. Fix $0<\nu<\frac{1}{1-h}$. For $c>0$ and $\iota\in\{\text{sec},\text{bm},\text{rs}\}$, define
\begin{gather}
    T_{\iota,c}(\varepsilon)
    :=\inf\left\{n\in\mathbb N:n>\varepsilon^{-\nu},\quad
    c\sqrt{\frac{Z_2^{*,\iota}(n)}{n^2}}<\varepsilon\right\},
    \label{eq: stop2}
\end{gather}
For the corresponding limiting processes, define
\begin{align*}
    \Phi_{\iota,c}(u)
    :=\inf\left\{s>0:
    c\frac{\sqrt{Y_2^{*,\iota}(s)}}{s}<u\right\},
    \qquad u>0.
\end{align*}
Let $\mathbb F^{Y,\iota}$ be the usual augmentation of the natural
filtration generated by $(Y_1^\iota,Y_2^{*,\iota})$.
The same smoothed scaling process is used both to determine termination and to construct the terminal confidence interval.
The following results are the consequences of Section~\ref{sec: frame} for the three scaling constructions.
\begin{thm}
    \label{thm: stop_prop_rep}
    Fix $\iota\in\{\text{sec},\text{bm},\text{rs}\}$. Suppose that the
    hypotheses of the corresponding construction theorem hold and that
    Assumption~\ref{asmp: early_stopping} holds with
    $Z_2=Z_2^{*,\iota}$. Then, for every fixed $c>0$,
    \begin{enumerate}
        \item[(a)] as $\varepsilon\to0^+$, $a_\varepsilon^{-1}T_{\iota,c}\left(\frac{\varepsilon}{\cdot}\right)\overset{M_1}{\Rightarrow}\Phi_{\iota,c}\left(\frac{1}{\cdot}\right)$ locally on $(0,\infty)$;
        \item[(b)] for every $t>0$ and $a>0$, $\Phi_{\iota,c}(at)\overset{d}{=}a^{\frac{1}{h-1}}\Phi_{\iota,c}(t)$, and $\Phi_{\iota,c}(t)<\infty$ almost surely;
        \item[(c)] $\varepsilon^\gamma T_{\iota,c}(\varepsilon)\topb0$ whenever $\gamma>\frac{1}{1-h}$;
        \item[(d)] for every $u>0$ at which $\Phi_{\iota,c}(\cdot)$ is almost surely continuous, $\Phi_{\iota,c}(u)$ is predictable with respect to $\mathbb F^{Y,\iota}$.
    \end{enumerate}
\end{thm}
\begin{thm}
\label{thm: asymp_valid_rep}
    Fix $\iota\in\{\text{sec},\text{bm},\text{rs}\}$ and suppose that the
    assumptions of Theorem~\ref{thm: stop_prop_rep} hold. Fix
    $\delta\in(0,1)$ and suppose that
    $\Phi_{\iota,1}(1)>0$ almost surely. Fix $c>0$ and suppose that
    $\Phi_{\iota,c}(\cdot)$ is almost surely continuous at $1$. Then, as
    $\varepsilon\to0^+$,
    \begin{align}
             \frac{Z_1^{\iota}(T_{\iota,c}(\varepsilon))-T_{\iota,c}(\varepsilon)\mu}{\sqrt{Z_2^{*,\iota}(T_{\iota,c}(\varepsilon))}}\tod
             \frac{Y_1^{\iota}(\Phi_{\iota,c}(1))}{\sqrt{Y_2^{*,\iota}(\Phi_{\iota,c}(1))}}=:W_{\iota,c}.
             \label{eq: limit_dist_rep}
    \end{align}
    The distribution of $W_{\iota,c}$ does not depend on $c$. Let $F_\iota$
    be this common distribution function, set
    $c_{l,\iota}=F_\iota^{\leftarrow}\left(\frac{\delta}{2}\right)$,
    $c_{u,\iota}=F_\iota^{\leftarrow}\left(1-\frac{\delta}{2}\right)$, and
    $c_\iota^*=c_{u,\iota}-c_{l,\iota}$. Suppose that
    $c_{l,\iota}<c_{u,\iota}$, $F_\iota$ is continuous at these quantiles,
    and $\Phi_{\iota,c_\iota^*}(\cdot)$ is almost surely continuous at
    $1$. Then the confidence interval
    \begin{align}
    \label{eq: cover_rep}
        \text{CI}_{\iota}(\varepsilon)
        :=\left[
        \frac{Z^{\iota}_1(T_{\iota,c_\iota^*}(\varepsilon))}{T_{\iota,c_\iota^*}(\varepsilon)}
        -c_{u,\iota}\sqrt{\frac{Z_2^{*,\iota}(T_{\iota,c_\iota^*}(\varepsilon))}{T_{\iota,c_\iota^*}(\varepsilon)^2}},
        \frac{Z_1^{\iota}(T_{\iota,c_\iota^*}(\varepsilon))}{T_{\iota,c_\iota^*}(\varepsilon)}
        -c_{l,\iota}\sqrt{\frac{Z_2^{*,\iota}(T_{\iota,c_\iota^*}(\varepsilon))}{T_{\iota,c_\iota^*}(\varepsilon)^2}}
        \right]
    \end{align}
    satisfies $\PB\bigl(\mu\in\text{CI}_{\iota}(\varepsilon)\bigr)\to1-\delta$.
\end{thm}

%% file: tex/example.tex
In this section, we illustrate the applications of the sequential stopping
procedure introduced in Section~\ref{sec: frame} using a heavy-tailed
moving-average process, stochastic approximation, and an M/G/1 queue
with heavy-tailed service times. For each case, we verify the process-level
convergence in Assumption~\ref{asmp: fclt}, together with
Assumption~\ref{asmp: early_stopping} for the corresponding scaling
processes.

\subsection{Moving Average Process}
We consider the time series model where the time coefficient is fixed to be a constant:
\begin{align}
\label{md: ar1}
    &\text{MA(p): }\hspace{4pt}X_t=\sum_{i=0}^{p}\varphi^i\xi_{t-i}.
\end{align}
where $\{\xi_{i}\}_{i\in\ZB}$ is an i.i.d. sequence with finite mean
$\EB[\xi_0]=\mu_\xi$ and satisfies Assumption~\ref{asmp: domain}. The target
is the marginal mean
$\mu:=\EB[X_t]=\mu_\xi\sum_{i=0}^p\varphi^i$, where the sum is interpreted
as an infinite series when $p=\infty$.
\begin{thm}
\label{thm: time}
    Suppose that $0<\varphi<1$, $p\in\mathbb N_0\cup\{\infty\}$, and
    the distribution of $\xi_t$ satisfies Assumption~\ref{asmp: domain}
    with $\alpha\in(1,2)$. Let $(Z_1,Z_2)$ be as
    in~\eqref{eq: ss2}, where $X_i$ follows~\eqref{md: ar1}. Then
    $(Z_1,Z_2)$ satisfies Assumption~\ref{asmp: fclt} with
    $h=\frac{1}{\alpha}$ and a suitable slowly varying function $\ell$.
    Moreover, Assumption~\ref{asmp: early_stopping} holds for every
    $0<\nu<\frac{\alpha}{\alpha-1}$. When $p=\infty$, the moving-average
    representation is that of an autoregressive process of order one.
\end{thm}

\subsection{Stochastic Approximation}
We consider the one-dimensional stochastic approximation
\begin{align}
    \theta_{n+1}
    =\theta_n-\eta_n g(\theta_n;\xi_{n+1}),
    \qquad
    g(\theta;\xi)=H(\theta-\theta^*)+\xi,
    \label{eq: sgd_linear}
\end{align}
where $\theta^*\in\RB$, $H>0$, and
$\eta_n=c_\eta(n+1)^{-\varrho}$ with $c_\eta>0$,
$\frac{1}{\alpha}<\varrho<1$, and $0<c_\eta H<1$. The initial value $\theta_0$ is deterministic. The variables
$\{\xi_n\}_{n\ge1}$ are i.i.d., have mean zero, and their common
distribution satisfies Assumption~\ref{asmp: domain} for some
$\alpha\in(1,2)$. Define
\begin{align}
    \label{eq: pa}
    &Z_1(t)=\sum_{i=1}^{\lfloor t\rfloor}\theta_i,\\
    \label{eq: pa_var}
    &Z_2(t)=\sum_{i=1}^{\lfloor t\rfloor}
    g(\theta_{i-1};\xi_i)^2.
\end{align}
\begin{thm}
\label{thm: pa_limit}
    Let $(L_1,L_2)$ be the joint L\'evy process characterized by the
    following convergence as $\varepsilon\to0^+$:
    \begin{align}
        \left(
        \varepsilon^{\frac{1}{\alpha}}\ell(\varepsilon^{-1})
        \sum_{i=1}^{\lfloor\frac{\cdot}{\varepsilon}\rfloor}\xi_i,
        \varepsilon^{\frac{2}{\alpha}}\ell(\varepsilon^{-1})^2
        \sum_{i=1}^{\lfloor\frac{\cdot}{\varepsilon}\rfloor}\xi_i^2
        \right)
        \overset{J_1}{\Rightarrow}(L_1,L_2),
        \label{eq: pa_noise_joint_fclt}
    \end{align}
    where $\ell$ is a slowly varying function, $L_1$ is an
    $\alpha$-stable L\'evy process, and $L_2$ is an
    $\frac{\alpha}{2}$-stable subordinator. Then, as
    $\varepsilon\to0^+$,
    \begin{align}
        \left(
        \varepsilon^{\frac{1}{\alpha}}\ell(\varepsilon^{-1})\left(
        Z_1\left(\frac{\cdot}{\varepsilon}\right)
        -\theta^*\left\lfloor\frac{\cdot}{\varepsilon}\right\rfloor\right),
        \varepsilon^{\frac{2}{\alpha}}\ell(\varepsilon^{-1})^2
        Z_2\left(\frac{\cdot}{\varepsilon}\right)
        \right)
        \overset{\mathrm{WM}_1}{\Rightarrow}
        \left(-H^{-1}L_1,L_2\right).
        \label{eq: pa_joint_fclt}
    \end{align}
    Consequently, $(Z_1,Z_2)$ satisfies
    Assumption~\ref{asmp: fclt} with $\mu=\theta^*$,
    $h=\frac{1}{\alpha}$, the above $\ell$, $Y_1=-H^{-1}L_1$, and $Y_2=L_2$.
    Moreover, Assumption~\ref{asmp: early_stopping} holds for
    every $0<\nu<\frac{\alpha}{\alpha-1}$.
\end{thm}
\begin{rem}
    Multidimensional and nonlinear extensions follow the same linearization
    principle, but require suitable approximation and multivariate distribution conditions. 
\end{rem}

\subsection{M/G/1 Queueing System}
\label{sec: queue}
In some stochastic models, the empirical standard deviation need not have the same stochastic order as the estimation error and therefore may fail to eliminate the nuisance parameters. We illustrate this issue with an M/G/1 queue model:
\begin{itemize}
	\item The inter-arrival times $\{A_n\}_{n\ge1}$ are i.i.d.
	exponential with $\EB[A_n]=\lambda_A^{-1}$.
	\item Service time $\{U_n\}_{n\ge0}$ is i.i.d. and $\alpha$-regular varying with mean $\EB[U_n]=\mu_U$ with $\alpha\in(2,4)$.
	\item Single server.
	\item Light traffic: $\mu_U<\lambda_A^{-1}$.
\end{itemize}
Assume that the inter-arrival and service-time sequences are mutually
independent. We denote the delay sequence by
$W_{n+1}=(W_n+U_n-A_{n+1})_+$, with $W_0=0$, and set
$Z_1(t)=\sum_{i=1}^{\lfloor t\rfloor}W_i$ for $t\ge0$.
For simplification, we consider the first busy cycle, from
$\beta_0=0$ to
$\beta_1=\inf\{k\ge1:W_k=0\}$, during which $W_k>0$ for
$1\le k<\beta_1$. A large service time of order $x$ produces a workload
excursion whose height is also of order $x$. Since the workload drains at
mean rate $\lambda_A^{-1}-\mu_U>0$, the duration of this excursion is of
order $x$, and its accumulated waiting time is therefore of order $x^2$.
More generally, if
$Z_p(t):=\sum_{i=1}^{\lfloor t\rfloor}W_i^p$, then
$Z_p(\beta_1)-Z_p(\beta_0)$ is of order $x^{p+1}$. Equivalently, over a
horizon of order $\varepsilon^{-1}$, the largest service time is of order
$\varepsilon^{-\frac{1}{\alpha}}$, up to slowly varying factors, so a
single extreme busy cycle contributes order
$\varepsilon^{-\frac{p+1}{\alpha}}$ to $Z_p$; for $p=1$, this is order
$\varepsilon^{-\frac{2}{\alpha}}$. This one-big-jump mechanism is
illustrated in Figure~\ref{fig: queue}. Consequently, the centered
cumulative waiting-time process and the empirical sum of squares generally
have incompatible stochastic scales, so the sum-of-squares construction in
Section~\ref{sec: ss} does not provide a matching self-normalizer. We
therefore establish a centered functional limit for $Z_1$ and use the
constructions in Section~\ref{sec: sec_bm}.
\begin{figure}[htbp!]
    \centering
    \includegraphics[width=0.7\linewidth]{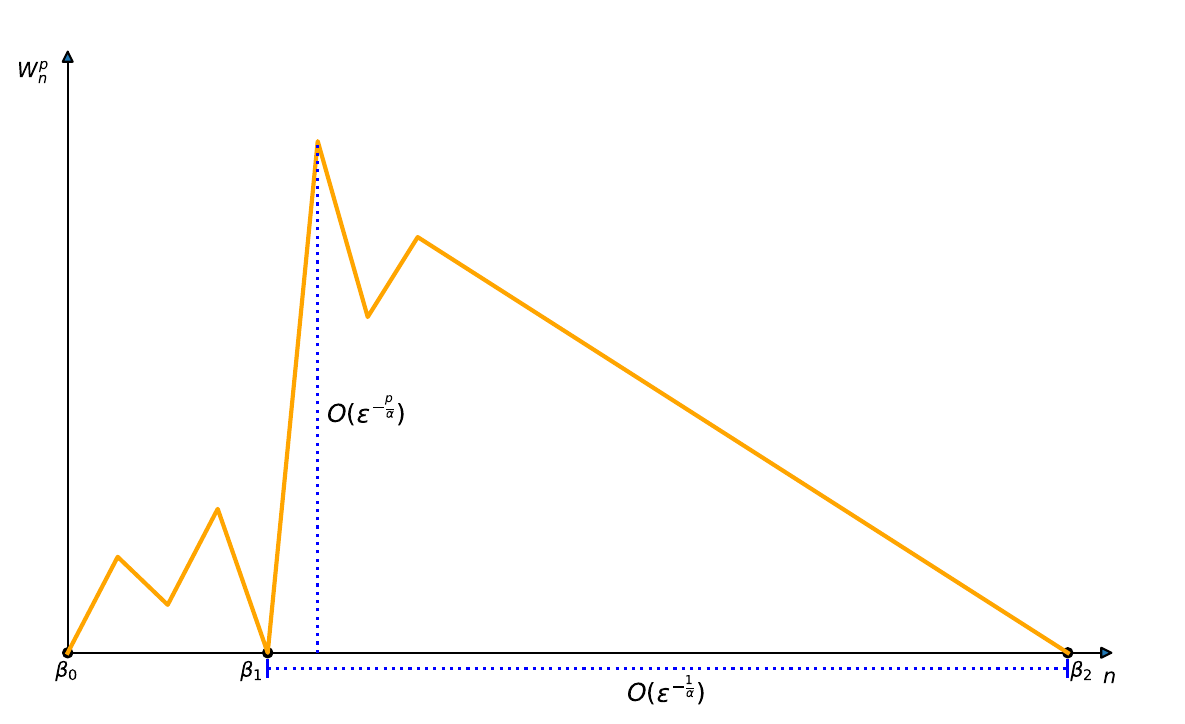}
    \caption{Illustration of $W_n^p$ in M/G/1 queue.}
    \label{fig: queue}
\end{figure}
The following result identifies the tail of the total delay time during one busy cycle with general power $p>0$.
For $p=1$, the corresponding excursion-area asymptotic was established for heavy-tailed random walks by \citet{denisov2021tail}; see also \citet{borovkov2003integral} for closely related workload-integral asymptotics in single-server queues.
\begin{thm}
\label{thm: queue_reward_tail}
    For the M/G/1 light-traffic queue described above, let
    $X_i=U_{i-1}-A_i$ and $\beta_1=\inf\{n\ge1:W_n=0\}$. Put
    $a=-\EB[X_1]=\lambda_A^{-1}-\mu_U>0$. Then, for every $p>0$,
    \begin{align}
        \PB\left(\sum_{i=1}^{\beta_1}W_i^p>x\right)
        \sim
        \EB[\beta_1]\,
        \PB\left(X_1>(a(p+1)x)^{\frac{1}{p+1}}\right).
        \label{eq: queue_reward_tail}
    \end{align}
    Consequently, $\sum_{i=1}^{\beta_1}W_i^p$ has a regularly
    varying tail with index $-\frac{\alpha}{p+1}$. In particular,
    $\sum_{i=1}^{\beta_1}W_i$ has a regularly varying tail with
    index $-\frac{\alpha}{2}$.
\end{thm}

\begin{thm}
\label{thm: queue}
    Let $\tau_1=\beta_1$, $D_1=\sum_{i=1}^{\beta_1}W_i$, and
    $\mu_W=\frac{\EB[D_1]}{\EB[\tau_1]}$.
    There exists a slowly varying function $\ell$
    such that, as $\varepsilon\downarrow0$,
    \begin{align}
        \varepsilon^{\frac{2}{\alpha}}\ell(\varepsilon^{-1})
        \left(
        Z_1\left(\frac{\cdot}{\varepsilon}\right)
        -\frac{\cdot}{\varepsilon}\mu_W
        \right)
        \overset{M_1}{\Rightarrow}
        L_{\frac{\alpha}{2}}\left(\frac{\cdot}{\EB[\tau_1]}\right),
        \label{eq: queue_fclt}
    \end{align}
    where $L_{\frac{\alpha}{2}}$ is a mean-zero, spectrally positive
    $\frac{\alpha}{2}$-stable L\'evy process. Moreover,
    $\frac{Z_1(t)-\mu_Wt}{t}\to0$ almost surely. Hence the marginal conditions
    in Section~\ref{sec: sec_bm} hold with $h=\frac{2}{\alpha}$ and mean $\mu_W$.
\end{thm}

\begin{thm}
\label{thm: queue_early_stopping}
    Under the conditions of Theorem~\ref{thm: queue}, fix an integer
    $m\ge2$ and construct $Z_2^{*,\iota}$ as in
    Section~\ref{sec: sec_bm}, where
    $\iota\in\{\text{sec},\text{bm},\text{rs}\}$, using $m$ independent
    queueing systems for sectioning. Then
    Assumption~\ref{asmp: early_stopping} holds with
    $Z_2=Z_2^{*,\iota}$ for every
    $0<\nu<\frac{\alpha}{\alpha-2}$.
\end{thm}

%% file: tex/appendix.tex
\subsection{Auxiliary Lemmas}
\begin{lem}
\label{lem: phi_cont}
    Let $x_n,x\in\mathbb D_+$ satisfy
    \begin{align*}
        x_n\to x
    \end{align*}
    in the $M_2$ topology on compact time intervals. Fix $0<\eta<M<\infty$
    such that $x$ is continuous at $\eta$ and $M$, and define
    \begin{align*}
        \mathcal I_{\eta,M}(x)(y)
        :=\inf\left\{t\in[\eta,M]:
        \frac{t}{x(t)}>y\right\}\wedge M,
        \qquad y>0.
    \end{align*}
    Here $\frac{t}{0}:=+\infty$ and $\inf\varnothing:=+\infty$. Then
    \begin{align*}
        \mathcal I_{\eta,M}(x_n)\to
        \mathcal I_{\eta,M}(x)
    \end{align*}
    locally in the $M_1$ topology on $(0,\infty)$. Moreover, if
    $\delta_n\downarrow0$ and
    \begin{align*}
        \mathcal I_{\eta,M}^{[\delta]}(z)(y)
        :=\inf\left\{t\in\delta\mathbb N\cap[\eta,M]:
        \frac{t}{z(t)}>y\right\}\wedge M,
        \qquad y>0,
    \end{align*}
    then
    \begin{align*}
        \mathcal I_{\eta,M}^{[\delta_n]}(x_n)
        \to\mathcal I_{\eta,M}(x)
    \end{align*}
    locally in the $M_1$ topology on $(0,\infty)$.
\end{lem}
\begin{proof}
    The continuity of $x$ at $\eta$ and $M$ ensures that restriction to
    $[\eta,M]$ preserves the assumed $M_2$ convergence. For a c\`adl\`ag
    function $f$, let
    $\Gamma_f$ denote its completed graph on this interval, written in
    time--space coordinates, and define
    \begin{align*}
        \rho_n(t)&:=\frac{t}{t+x_n(t)},
        \qquad
        \rho(t):=\frac{t}{t+x(t)},
        \qquad t\in[\eta,M],\\
        \mathcal T(t,z)&:=\left(t,\frac{t}{t+z}\right),
        \qquad t\in[\eta,M],\quad z\ge0.
    \end{align*}
    Since $z\mapsto \frac{t}{t+z}$ is continuous and decreasing, $\mathcal T$ maps each
    vertical segment in the completed graph of $x_n$ or $x$
    onto the corresponding vertical segment of the transformed path.
    Consequently,
    \begin{align*}
        \Gamma_{\rho_n}=\mathcal T(\Gamma_{x_n}),
        \qquad
        \Gamma_\rho=\mathcal T(\Gamma_x).
    \end{align*}
    Moreover, for $t,s\ge\eta$ and $z,w\ge0$,
    \begin{align*}
        \left|\frac{t}{t+z}-\frac{s}{s+w}\right|
        &\le \frac{1}{4\eta}|t-s|+\frac{1}{\eta}|z-w|.
    \end{align*}
    Thus, under the maximum norm, $\mathcal T$ is Lipschitz on
    $[\eta,M]\times[0,\infty)$. Since the univariate $M_2$ metric is the
    Hausdorff metric between completed graphs, it follows that
    \begin{align*}
        d_{M_2}(\rho_n,\rho)
        &\le \left(1\vee\frac{5}{4\eta}\right)
        d_{M_2}(x_n,x)
        \to0
    \end{align*}
    on $[\eta,M]$.

    Let $L=M-\eta$ and define the unbounded extension
    \begin{align*}
        r_x(s):=
        \begin{cases}
            \rho(\eta+s),&0\le s<L,\\
            1+s-L,&s\ge L,
        \end{cases}
    \end{align*}
    with $r_{x_n}$ defined analogously. The endpoint continuity and the
    preceding completed-graph bound imply that $r_{x_n}\to r_x$ locally in
    $M_2$. Moreover,
    \begin{align*}
        \frac{t}{x(t)}>y
        \quad\Longleftrightarrow\quad
        \frac{t}{t+x(t)}>\frac{y}{1+y},
        \qquad t\in[\eta,M].
    \end{align*}
    If $r_x^{-1}(v):=\inf\{s\ge0:r_x(s)>v\}$, then
    \begin{align*}
        \mathcal I_{\eta,M}(x)(y)
        =\eta+r_x^{-1}\left(\frac{y}{1+y}\right).
    \end{align*}
    Corollary 13.6.5 of \cite{whitt2002stochastic}, followed by the continuous
    level change $y\mapsto \frac{y}{1+y}$ and addition of $\eta$, gives the first
    assertion.

    It remains to treat the vanishing grid. Enumerate
    \begin{align*}
        \delta_n\mathbb N\cap[\eta,M)
        =\{g_{n,0}<\cdots<g_{n,N_n}\},
        \qquad q_{n,j}:=g_{n,j}-\eta.
    \end{align*}
    For all sufficiently large $n$, this set is nonempty, and the partition
    formed by $0,q_{n,0},\ldots,q_{n,N_n},L$ has mesh at most
    $\delta_n$. Define the c\`adl\`ag sampled extension
    \begin{align*}
        \widehat r_n(s):=
        \begin{cases}
            r_{x_n}(q_{n,0}),&0\le s<q_{n,0},\\
            r_{x_n}(q_{n,j}),&q_{n,j}\le s<q_{n,j+1},
            \quad 0\le j<N_n,\\
            r_{x_n}(q_{n,N_n}),&q_{n,N_n}\le s<L,\\
            1+s-L,&s\ge L.
        \end{cases}
    \end{align*}
    We claim that $\widehat r_n\to r_x$ locally in $M_2$. To see this,
    fix a compact time interval and use the completed-graph characterization
    of the univariate $M_2$ topology. Every point on a horizontal segment of
    $\Gamma_{\widehat r_n}$ is within $\delta_n$ in time of a point of
    $\Gamma_{r_{x_n}}$. Every point on a vertical segment lies between two
    sampled values whose time coordinates are at most $\delta_n$ apart. After
    passing to a subsequence, the two sampled values converge to points in
    the same vertical fiber of $\Gamma_{r_x}$. This fiber is an interval and
    therefore also contains the limit of every value between the two sampled
    values. Since $\Gamma_{r_{x_n}}\to\Gamma_{r_x}$ in Hausdorff distance,
    every limit point of $\Gamma_{\widehat r_n}$ belongs to
    $\Gamma_{r_x}$.

    Conversely, let $(s,z)\in\Gamma_{r_x}$. If $r_x$ is continuous at
    $s$, evaluation at a grid point converging to $s$ approximates
    $(s,r_x(s))$. If $s$ is a jump time, first choose continuity points
    $a<s<b$. Grid points approaching $a$ and $b$ have sampled values
    converging to $r_x(a)$ and $r_x(b)$, respectively. The completed graph of
    $\widehat r_n$ between those grid points is connected, so its spatial
    projection contains all values between the two endpoint values. Letting
    first $n\to\infty$ and then $a\uparrow s$ and $b\downarrow s$
    approximates every point of $[r_x(s-),r_x(s)]$. The common linear
    extension handles times after $L$. These two inclusions prove
    \begin{align*}
        \widehat r_n\to r_x
        \qquad\text{locally in }M_2.
    \end{align*}
    Applying Corollary 13.6.5 of \cite{whitt2002stochastic} once more gives
    $\widehat r_n^{-1}\to r_x^{-1}$ locally in $M_1$. For $0<v<1$,
    the two grid constructions can differ only on the initial grid cell, and
    hence
    \begin{align*}
        \sup_{0<v<1}\left|
        \eta+\widehat r_n^{-1}(v)
        -\mathcal I_{\eta,M}^{[\delta_n]}(x_n)
        \left(\frac{v}{1-v}\right)\right|\le\delta_n.
    \end{align*}
    The level change $v=\frac{y}{1+y}$ now proves the second assertion.
\end{proof}

\begin{lem}
    \label{lem: continuity_stop}
    Let $Z(\cdot)$ be a positive c\`adl\`ag process with no negative jumps,
    and, for $u>0$, define
    \begin{align*}
        \tau(u):=\inf\{s>0:Z(s)<u\}.
    \end{align*}
    On $\{0<\tau(u)<\infty\}$, the process $Z(\cdot)$ is continuous at
    $\tau(u)$.
\end{lem}
\begin{proof}
    Fix $u>0$ and write $\tau=\tau(u)$. On
    $\{0<\tau<\infty\}$, the definition of $\tau$ gives
    $Z(s)\ge u$ for every $s<\tau$, and hence $Z(\tau-)\ge u$.
    There are times $s_k\downarrow\tau$ such that $Z(s_k)<u$, so right
    continuity gives $Z(\tau)\le u$. Since $Z$ has no negative jumps,
    $Z(\tau)\ge Z(\tau-)\ge u$. Thus
    $Z(\tau)=Z(\tau-)=u$, proving continuity at $\tau$.
\end{proof}
\begin{lem}
\label{lem:subquadratic_squares}
Let $(X_i)_{i\ge1}$ be a stationary sequence with
$\EB|X_1|<\infty$. Then
\begin{align*}
    \frac{1}{n^2}\sum_{i=1}^nX_i^2\to0
    \qquad\text{almost surely}.
\end{align*}
\end{lem}
\begin{proof}
For every $\delta>0$, stationarity and integrability give
\begin{align*}
    \sum_{i=1}^{\infty}\PB(|X_i|>\delta i)
    =\sum_{i=1}^{\infty}\PB(|X_1|>\delta i)<\infty.
\end{align*}
The Borel--Cantelli lemma therefore gives
$\max_{1\le i\le n}\frac{|X_i|}{n}\to0$ almost surely. By Birkhoff's ergodic
theorem applied on each ergodic component,
$n^{-1}\sum_{i=1}^n|X_i|$ converges almost surely to a finite random
variable. Consequently,
\begin{align*}
    \frac{1}{n^2}\sum_{i=1}^nX_i^2
    \le
    \frac{\max_{1\le i\le n}|X_i|}{n}
    \frac{1}{n}\sum_{i=1}^n|X_i|\to0
\end{align*}
almost surely.
\end{proof}

\begin{lem}
\label{lem: kolmogorov_rogozin}
Let $X_1,\ldots,X_n$ be independent real-valued random variables. There
is a universal constant $C>0$ such that, for every $z>0$,
\begin{align*}
    \sup_{x\in\RB}\PB\left(\left|\sum_{i=1}^nX_i-x\right|\le z\right)
    \le\frac{C}{\sqrt{\sum_{i=1}^n\left(1-\sup_{x\in\RB}\PB(|X_i-x|\le z)\right)}}.
\end{align*}
\end{lem}
\begin{proof}
Apply Theorem~1 of \citet{rogozin1961estimate} with $L=l=2z$.
\end{proof}

\begin{lem}
    \label{lem: q_est}
    Let $H,c_\eta>0$, $0<\varrho<1$, $0<c_\eta H<1$, and
    $\eta_k=c_\eta(k+1)^{-\varrho}$. For $0\le k<j$, define
    \begin{align*}
        P_{k,j}:=\prod_{l=k+1}^j(1-H\eta_l),
        \qquad
        q_{k,j}:=H\eta_k\sum_{i=k+1}^jP_{k,i-1},
    \end{align*}
    with $P_{k,k}=1$ and $q_{k,k}=0$. Then, uniformly over $j>k$,
    \begin{align*}
        0\le q_{k,j}-1+P_{k,j}\lesssim(k+1)^{\varrho-1}.
    \end{align*}
\end{lem}
\begin{proof}
    The identity
    $P_{k,i-1}-P_{k,i}=H\eta_iP_{k,i-1}$ gives
    \begin{align*}
        q_{k,j}-1+P_{k,j}
        =H\sum_{i=k+1}^j(\eta_k-\eta_i)P_{k,i-1}\ge0.
    \end{align*}
    Concavity of $x\mapsto x^{\varrho}$ gives
    \begin{align*}
        \eta_k-\eta_i
        =\eta_i\left(\left(\frac{i+1}{k+1}\right)^{\varrho}-1\right)
        \le\eta_i\frac{i-k}{k+1}.
    \end{align*}
    Summation by parts therefore yields
    \begin{align*}
        q_{k,j}-1+P_{k,j}
        &\le\frac{1}{k+1}\sum_{i=k+1}^j
        (i-k)(P_{k,i-1}-P_{k,i})\\
        &\le\frac{1}{k+1}\sum_{i=k+1}^{\infty}P_{k,i-1}
        \lesssim(k+1)^{\varrho-1}.
    \end{align*}
    For the last bound, split the sum into blocks of length
    $(k+1)^{\varrho}$ and use the exponential bound for $P_{k,i-1}$.
\end{proof}

\begin{lem}
\label{lem: q_square}
    Under the conditions and notation of Lemma~\ref{lem: q_est},
    \begin{align*}
        \sum_{k=1}^i(\eta_{k-1}-\eta_i)^2P_{k-1,i-1}^2
        \lesssim i^{\varrho-2}.
    \end{align*}
\end{lem}
\begin{proof}
    For $\frac{i}{2}\le k\le i$, concavity and monotonicity give
    \begin{align*}
        \eta_{k-1}-\eta_i
        \le\frac{2(i-k+1)}{i}\eta_i,
        \qquad
        P_{k-1,i-1}^2
        \le\exp(-2H(i-k)\eta_i).
    \end{align*}
    Hence
    \begin{align*}
        \sum_{\frac{i}{2}\le k\le i}(\eta_{k-1}-\eta_i)^2P_{k-1,i-1}^2
        &\lesssim\frac{\eta_i^2}{i^2}
        \sum_{m=0}^{\infty}(m+1)^2e^{-2Hm\eta_i}\\
        &\lesssim\frac{1}{i^2\eta_i}\asymp i^{\varrho-2}.
    \end{align*}
    For $k<\frac{i}{2}$, the product is at most
    $\exp(-c_0i^{1-\varrho})$, uniformly in $k$, and the resulting sum is
    $o(i^{\varrho-2})$.
\end{proof}

\subsection{Proofs for Section~\ref{sec: frame}}

\subsubsection{Proof of Theorem~\ref{thm: stop_prop}}
\paragraph*{Proof of (a).} Let
\begin{align*}
    g(\varepsilon)
    :=\frac{1}{\varepsilon^{h-1}\ell(\varepsilon^{-1})},
    \qquad
    X_\varepsilon(t)
    :=\varepsilon^{h}\ell(\varepsilon^{-1})
    \sqrt{Z_2\left(\frac{t}{\varepsilon}\right)}.
\end{align*}
For every $u>0$, the stopping time $T_c(\cdot)$ satisfies
\begin{align}
    \varepsilon T_c\left(\frac{g(\varepsilon)}{u}\right)
    &=\varepsilon\inf\left\{n\in\mathbb N:
    n>\left(\frac{g(\varepsilon)}{u}\right)^{-\nu},\quad
    c\sqrt{\frac{Z_2(n)}{n^2}}<\frac{g(\varepsilon)}{u}\right\}\\
    &=\inf\left\{t\in\varepsilon\mathbb N:
    t>\varepsilon g(\varepsilon)^{-\nu}u^\nu,\quad
    \frac{t}{X_\varepsilon(t)}>cu\right\}.
    \label{eq: stop_1}
\end{align}
Since $\nu<\frac{1}{1-h}$,
$\varepsilon g(\varepsilon)^{-\nu}\to0$. Moreover,
Assumption~\ref{asmp: fclt} gives
\begin{align}
\label{eq: stop_2}
    X_\varepsilon(\cdot)
    \overset{M_2}{\Rightarrow}\sqrt{Y_2(\cdot)}.
\end{align}
By self-similarity and the c\`adl\`ag property, it can be simply verified
that every deterministic positive time is almost surely a continuity point
of $Y_2$. For fixed $0<\eta<M<\infty$, the extended continuous mapping theorem and
Lemma~\ref{lem: phi_cont} give
\begin{align}
    \mathcal I^{[\varepsilon]}_{\eta,M}(X_\varepsilon)(c\,\cdot)
    \overset{M_1}{\Rightarrow}
    \mathcal I_{\eta,M}\left(\sqrt{Y_2}\right)(c\,\cdot)
    \label{eq: truncated_inverse_convergence}
\end{align}
locally on $(0,\infty)$. Fix a compact interval $K\subset(0,\infty)$.
It remains to remove the truncation at $\eta$. Since
$\varepsilon g(\varepsilon)^{-\nu}\to0$, for fixed $\eta>0$ and all
sufficiently small $\varepsilon$, the lower bound in~\eqref{eq: stop_1}
is smaller than $\eta$ uniformly over $u\in K$:
\begin{align*}
    \varepsilon g(\varepsilon)^{-\nu}
    \max_{u\in K}u^\nu<\eta.
\end{align*}
Define the event
\begin{align*}
    E_{\varepsilon,\eta}
    :=\left\{
    \inf_{u\in K}
    \varepsilon T_c\left(\frac{g(\varepsilon)}{u}\right)>\eta
    \right\}.
\end{align*}
On $E_{\varepsilon,\eta}$, the definition in~\eqref{eq: stop_1} gives
\begin{align}
    \left[
    \varepsilon T_c\left(\frac{g(\varepsilon)}{\cdot}\right)
    \right]\wedge M
    =\mathcal I^{[\varepsilon]}_{\eta,M}(X_\varepsilon)(c\,\cdot)
    \qquad\text{on }K.
    \label{eq: remove_eta}
\end{align}
To control $\PB(E_{\varepsilon,\eta}^c)$, let $\underline u:=\min K$. Since
$u\mapsto T_c\left(\frac{g(\varepsilon)}{u}\right)$ is nondecreasing and
$a_{\frac{g(\varepsilon)}{\underline u}}\asymp\varepsilon^{-1}$, there is a
constant $C<\infty$ such that, for all sufficiently small $\varepsilon$,
\begin{align*}
\PB(E_{\varepsilon,\eta}^c)
&=\PB\left(
    T_c\left(\frac{g(\varepsilon)}{\underline u}\right)
    \le\frac{\eta}{\varepsilon}
    \right)\\
&\le\PB\left(
    \inf_{\substack{n\in\mathbb N:\,
    \left(\frac{g(\varepsilon)}{\underline u}\right)^{-\nu}\le n
    \le\frac{\eta}{\varepsilon}}}
    \frac{\sqrt{Z_2(n)}}{n}
    \le\frac{g(\varepsilon)}{c\underline u}
    \right)\\
&\le\PB\left(
    \inf_{\substack{n\in\mathbb N:\,
    \left(\frac{g(\varepsilon)}{\underline u}\right)^{-\nu}\le n
    \le C\eta a_{\frac{g(\varepsilon)}{\underline u}}}}
    \frac{\sqrt{Z_2(n)}}{n}
    \le\frac{g(\varepsilon)}{c\underline u}
    \right).
\end{align*}
Applying Assumption~\ref{asmp: early_stopping} gives
\begin{align*}
    \lim_{\eta\downarrow0}\limsup_{\varepsilon\downarrow0}
    \PB(E_{\varepsilon,\eta}^c)=0.
\end{align*}
Together with~\eqref{eq: remove_eta}, this shows that
\begin{align}
&\lim_{\eta\downarrow0}\limsup_{\varepsilon\downarrow0}
\PB\left(
\sup_{u\in K}\left|
\left[
\varepsilon T_c\left(\frac{g(\varepsilon)}{u}\right)
\right]\wedge M
-\mathcal I^{[\varepsilon]}_{\eta,M}(X_\varepsilon)(cu)
\right|>0
\right)=0.
\label{eq: lower_truncation_error}
\end{align}
We next study the limit of
$\mathcal I_{\eta,M}\left(\sqrt{Y_2}\right)(c\,\cdot)$ as $\eta\downarrow0$, with
$M$ fixed. For every $u>0$, the crossing sets
\begin{align*}
    \left\{t\in[\eta,M]:
    \frac{t}{\sqrt{Y_2(t)}}>cu\right\}
\end{align*}
increase to the corresponding set with $t\in(0,M]$ as
$\eta\downarrow0$. Hence,
\begin{align*}
    \mathcal I_{\eta,M}\left(\sqrt{Y_2}\right)(cu)
    &\to
    \inf\left\{t\in(0,M]:
    \frac{t}{\sqrt{Y_2(t)}}>cu\right\}\wedge M\\
    &=\Phi_c\left(\frac{1}{u}\right)\wedge M,
\end{align*}
almost surely. Since the inverse paths are nondecreasing, pointwise
convergence at continuity points yields
\begin{align}
    \mathcal I_{\eta,M}\left(\sqrt{Y_2}\right)(c\,\cdot)
    \overset{M_1}{\to}
    \Phi_c\left(\frac{1}{\cdot}\right)\wedge M
    \label{eq: limiting_truncated_inverse}
\end{align}
almost surely locally on $(0,\infty)$. Combining
\eqref{eq: truncated_inverse_convergence},
\eqref{eq: lower_truncation_error}, and
\eqref{eq: limiting_truncated_inverse}, the converging-together theorem
gives, for every fixed $M$,
\begin{align*}
    \left[
    \varepsilon T_c\left(\frac{g(\varepsilon)}{\cdot}\right)
    \right]\wedge M
    \overset{M_1}{\Rightarrow}
    \Phi_c\left(\frac{1}{\cdot}\right)\wedge M.
\end{align*}
It remains to remove the upper truncation. Let $\overline u:=\max K$ and
$M_\varepsilon:=\varepsilon\lfloor \frac{M}{\varepsilon}\rfloor$. Then
$M_\varepsilon\to M$, and $M_\varepsilon\in\varepsilon\mathbb N$.
For all sufficiently small $\varepsilon$, the lower bound
in~\eqref{eq: stop_1} is less than $M_\varepsilon$ uniformly over
$u\in K$. Hence, monotonicity and
\eqref{eq: stop_1} give
\begin{align*}
\PB\left(\sup_{u\in K}\varepsilon T_c\left(\frac{g(\varepsilon)}{u}\right)>M\right)
&=\PB\left(\varepsilon T_c\left(\frac{g(\varepsilon)}{\overline u}\right)>M\right)\\
&\le\PB\left(X_\varepsilon(M_\varepsilon)\ge\frac{M_\varepsilon}{c\overline u}\right).
\end{align*}
By~\eqref{eq: stop_2}, continuity of $Y_2$ at $M$, and
$M_\varepsilon\to M$, we have
$X_\varepsilon(M_\varepsilon)\tod\sqrt{Y_2(M)}$. Self-similarity then gives
\begin{align*}
\limsup_{\varepsilon\downarrow0}\PB\left(\varepsilon T_c\left(\frac{g(\varepsilon)}{\overline u}\right)>M\right)
&\le\PB\left(\sqrt{Y_2(M)}\ge\frac{M}{c\overline u}\right)\\
&=\PB\left(\sqrt{Y_2(1)}\ge\frac{M^{1-h}}{c\overline u}\right)\to0
\end{align*}
as $M\to\infty$. The same bound shows that the limiting inverse is almost
surely finite on $K$. Letting $M\to\infty$ gives
\begin{align}
    \varepsilon T_c\left(\frac{g(\varepsilon)}{\cdot}\right)
    \overset{M_1}{\Rightarrow}
    \Phi_c\left(\frac{1}{\cdot}\right).
\end{align}
Since $A(a_\varepsilon)\sim\varepsilon^{-1}$,
\begin{align*}
    \frac{g(a_\varepsilon^{-1})}{\varepsilon}\to1.
\end{align*}
Applying the preceding convergence with $a_\varepsilon^{-1}$ in place of
$\varepsilon$ gives
\begin{align*}
    \left[a_\varepsilon^{-1}T_c\left(
    \frac{g(a_\varepsilon^{-1})}{\cdot}\right)\right]
    \left(\frac{g(a_\varepsilon^{-1})}{\varepsilon}u\right)
    =a_\varepsilon^{-1}T_c\left(\frac{\varepsilon}{u}\right).
\end{align*}
The increasing homeomorphisms
$u\mapsto \frac{g(a_\varepsilon^{-1})u}{\varepsilon}$ converge locally uniformly
to the identity, so the $M_1$ time-change theorem gives
\begin{align*}
    a_\varepsilon^{-1}T_c\left(\frac{\varepsilon}{\cdot}\right)
    \overset{M_1}{\Rightarrow}
    \Phi_c\left(\frac{1}{\cdot}\right)
\end{align*}
locally on $(0,\infty)$.

\paragraph*{Proof of (b).} For any $t>0$ and $a>0$,
\begin{align}
    a\Phi_c(t) &= \inf\left\{as>0\left|c\sqrt{\frac{Y_2(s)}{s^2}}<t\right.\right\}\\
    &=\inf\left\{s>0\left|c\sqrt{\frac{Y_2(a^{-1}s)}{a^{-2}s^2}}<t\right.\right\}\\
    &\overset{d}{=}\inf\left\{s>0\left|c\sqrt{\frac{a^{-2h}Y_2(s)}{a^{-2}s^2}}<t\right.\right\}\\
    &=\Phi_c\left(a^{h-1}t\right),
\end{align}
where the third step is due to the self-similarity of $Y_2(\cdot)$. Moreover, as $x\to\infty$,
\begin{align*}
    \PB\left(\Phi_c(t)>x\right)
    &\le\PB\left(c\frac{\sqrt{Y_2(x)}}{x}\ge t\right)\\
    &=\PB\left(cx^{h-1}\sqrt{Y_2(1)}\ge t\right)\to0,
\end{align*}
because $h<1$. Thus, $\Phi_c(t)<\infty$ almost surely.

\paragraph*{Proof of (c).}
Part~(a) implies that $\frac{T_c(\varepsilon)}{a_\varepsilon}$ is tight. Since
$A^{\leftarrow}$ is regularly varying with index $\frac{1}{1-h}$,
\begin{align*}
    \varepsilon^\gamma a_\varepsilon\to0
    \qquad\text{whenever }\gamma>\frac{1}{1-h}.
\end{align*}
Therefore,
\begin{align*}
    \varepsilon^\gamma T_c(\varepsilon)
    =\left(\varepsilon^\gamma a_\varepsilon\right)
    \frac{T_c(\varepsilon)}{a_\varepsilon}\topb0.
\end{align*}

\paragraph*{Proof of (d).}
Define
\begin{align*}
    R(s)=c\frac{\sqrt{Y_2(s)}}{s},\qquad s>0.
\end{align*}
Then $R(\cdot)$ is $\mathbb F^Y$-adapted and c\`adl\`ag on
$(0,\infty)$. Since $Y_2(\cdot)$ has no negative jumps and the
denominator is continuous and strictly positive on $(0,\infty)$,
$R(\cdot)$ also has no negative jumps. For every $v,t>0$, right
continuity gives
\begin{align*}
    \{\Phi_c(v)<t\}
    =\bigcup_{q\in\mathbb Q\cap(0,t)}\{R(q)<v\}\in\mathcal F_t^Y.
\end{align*}
Thus, $\Phi_c(v)$ is an $\mathbb F^Y$-stopping time without requiring
$R$ to have a finite extension at zero. Fix $u>0$ at which
$\Phi_c(\cdot)$ is almost surely continuous, and set
\begin{align*}
    \tau=\Phi_c(u),
    \qquad
    \tau_n=\Phi_c\left(u+\frac{1}{n}\right),
    \qquad n\ge1.
\end{align*}
The stopping times $\tau_n$ are nondecreasing in $n$ and satisfy
$\tau_n\le\tau$. On $\{\tau=0\}$, this implies $\tau_n=0$ for every
$n$, so the announcing property at zero is immediate. By
Theorem~\ref{thm: stop_prop}(b), $\tau<\infty$ almost surely. On
$\{0<\tau<\infty\}$, Lemma~\ref{lem: continuity_stop} implies that
$R(\cdot)$ is continuous at $\tau$ and $R(\tau)=u$. For every $n$,
continuity at $\tau$ and $u<u+\frac{1}{n}$ imply that
$\{0<s<\tau:R(s)<u+\frac{1}{n}\}$ is nonempty. Hence,
$\tau_n<\tau$. Finally, monotonicity gives
\begin{align*}
    \tau_n\uparrow\Phi_c(u+)=\Phi_c(u)=\tau
\end{align*}
almost surely, where the second equality follows from the assumed continuity of $\Phi_c(\cdot)$ at $u$. Thus, $(\tau_n)_{n\ge1}$ announces $\tau$, proving that $\Phi_c(u)$ is predictable.

\subsubsection{Proof of Theorem~\ref{thm: asymp_valid}}
\paragraph*{Proof of the random-time limit.}
The assumed positivity of $\Phi_1(1)$, the identity
$\Phi_c(1)=\Phi_1(c^{-1})$, and Theorem~\ref{thm: stop_prop}(b)
show that $\Phi_c(1)$ is almost surely positive and finite. By
Assumption~\ref{asmp: fclt},
\begin{align}
\label{eq: valid_1}
    \left(
    a_\varepsilon^{-h}\ell(a_\varepsilon)
    \left(Z_1(a_\varepsilon\,\cdot)-\idx(a_\varepsilon\,\cdot)\mu\right),
    a_\varepsilon^{-h}\ell(a_\varepsilon)
    \sqrt{Z_2(a_\varepsilon\,\cdot)}
    \right)
    \overset{\mathrm{WM}_2}{\Rightarrow}
    \left(Y_1(\cdot),\sqrt{Y_2(\cdot)}\right).
\end{align}
Let $V_\varepsilon$ denote the second coordinate on the left-hand side
of~\eqref{eq: valid_1}.
Since $\varepsilon A(a_\varepsilon)\to1$,
for the stopping rule $T_c(\varepsilon)$ in~\eqref{eq: stop1}, define its
normalized version by
\begin{align*}
    \tau_{\varepsilon,c}
    :=\frac{T_c(\varepsilon)}{a_\varepsilon}
    =\inf\left\{s\in a_\varepsilon^{-1}\mathbb N:
    s>\frac{\varepsilon^{-\nu}}{a_\varepsilon},\quad
    \frac{s}{V_\varepsilon(s)}>
    \frac{c}{\varepsilon A(a_\varepsilon)}\right\}.
\end{align*}
Here $\frac{s}{0}:=+\infty$.
The lower bound $\frac{\varepsilon^{-\nu}}{a_\varepsilon}$ converges to zero.
As in the proof of Theorem~\ref{thm: stop_prop}(a), every deterministic
positive time is almost surely a continuity point of $Y_2$. Fix
$0<\eta<M<\infty$. Since $\varepsilon A(a_\varepsilon)\to1$,
Slutsky's theorem and the
grid version of Lemma~\ref{lem: phi_cont}, applied jointly
to~\eqref{eq: valid_1}, give
\begin{align*}
    \left(
    a_\varepsilon^{-h}\ell(a_\varepsilon)
    \left(Z_1(a_\varepsilon\,\cdot)
    -\idx(a_\varepsilon\,\cdot)\mu\right),
    V_\varepsilon,
    \mathcal I_{\eta,M}^{[a_\varepsilon^{-1}]}
    \left(\frac{V_\varepsilon}{\varepsilon A(a_\varepsilon)}\right)
    \right)
    \Rightarrow
    \left(
    Y_1,\sqrt{Y_2},
    \mathcal I_{\eta,M}\left(\sqrt{Y_2}\right)
    \right)
\end{align*}
in the product of the two $M_2$ topologies and the $M_1$ topology for
the inverse process. The assumed continuity of $\Phi_c(\cdot)$ at $1$
implies that $y\mapsto\Phi_c\left(\frac{c}{y}\right)$ is almost surely continuous at $c$.
Define
\begin{align*}
    \tau_{\varepsilon,c}^{\eta,M}
    :=\mathcal I_{\eta,M}^{[a_\varepsilon^{-1}]}
    \left(\frac{V_\varepsilon}{\varepsilon A(a_\varepsilon)}\right)(c).
\end{align*}
On the event
\begin{align*}
    G_{\eta,M}:=\{\eta<\Phi_c(1)<M\},
\end{align*}
the continuity of $\Phi_c\left(\frac{c}{\cdot}\right)$ at $c$ implies that
$\mathcal I_{\eta,M}\left(\sqrt{Y_2}\right)$ is continuous at $c$, with
\begin{align*}
    \mathcal I_{\eta,M}\left(\sqrt{Y_2}\right)(c)=\Phi_c(1).
\end{align*}
For every bounded Lipschitz function $f$ on the product space,
\begin{align}
&\limsup_{\varepsilon\downarrow0}
\left|\EB f\left(
a_\varepsilon^{-h}\ell(a_\varepsilon)
\left(Z_1(a_\varepsilon\,\cdot)
-\idx(a_\varepsilon\,\cdot)\mu\right),V_\varepsilon,
\tau_{\varepsilon,c}^{\eta,M}\right)
-\EB f\left(Y_1,\sqrt{Y_2},\Phi_c(1)\right)\right|\notag\\
&\quad\le2\lVert f\rVert_\infty\PB(G_{\eta,M}^c)\notag\\
&\quad\le2\lVert f\rVert_\infty
\left(\PB\bigl(\Phi_c(1)\le\eta\bigr)
+\PB\bigl(\Phi_c(1)\ge M\bigr)\right).
\label{eq: valid_localized_bl}
\end{align}
The two terms on the right converge to zero as $\eta\downarrow0$ and
$M\to\infty$, respectively, because $\Phi_c(1)$ is almost surely positive
and finite.
We now control the localization error. Since the discrete infimum is
attained, for every fixed $\eta>0$ and all sufficiently small
$\varepsilon$,
\begin{align*}
    \PB\left(\tau_{\varepsilon,c}\le\eta\right)
    &\le\PB\left(
    \inf_{\substack{n\in\mathbb N:\,
    \varepsilon^{-\nu}\le n\le\eta a_\varepsilon}}
    \frac{\sqrt{Z_2(n)}}{n}
    \le\frac{\varepsilon}{c}
    \right).
\end{align*}
Applying Assumption~\ref{asmp: early_stopping} gives
\begin{align}
    \lim_{\eta\downarrow0}\limsup_{\varepsilon\downarrow0}
    \PB\left(\tau_{\varepsilon,c}\le\eta\right)=0.
    \label{eq: valid_lower_localization}
\end{align}
For the upper tail, let
$M_\varepsilon:=a_\varepsilon^{-1}\lfloor Ma_\varepsilon\rfloor$.
Then $M_\varepsilon\to M$ and
$M_\varepsilon\in a_\varepsilon^{-1}\mathbb N$. For all
sufficiently small $\varepsilon$, the lower bound
$\frac{\varepsilon^{-\nu}}{a_\varepsilon}$ is smaller than $M_\varepsilon$,
and
\begin{align*}
    \{\tau_{\varepsilon,c}>M\}
    \subseteq
    \left\{V_\varepsilon(M_\varepsilon)
    \ge\frac{\varepsilon A(a_\varepsilon)M_\varepsilon}{c}\right\}.
\end{align*}
Because $\varepsilon A(a_\varepsilon)\to1$ and
$V_\varepsilon(M_\varepsilon)\tod\sqrt{Y_2(M)}$,
\begin{align}
    \limsup_{\varepsilon\downarrow0}
    \PB\left(\tau_{\varepsilon,c}>M\right)
    &\le\PB\left(\sqrt{Y_2(M)}\ge\frac{M}{2c}\right)\notag\\
    &=\PB\left(\sqrt{Y_2(1)}\ge
    \frac{M^{1-h}}{2c}\right)\to0
    \qquad\text{as }M\to\infty.
    \label{eq: valid_upper_localization}
\end{align}
For small $\varepsilon$, on
$\{\eta<\tau_{\varepsilon,c}\le M\}$,
\begin{align*}
    \tau_{\varepsilon,c}
    =\tau_{\varepsilon,c}^{\eta,M}.
\end{align*}
Hence, for the same bounded Lipschitz function $f$,
the absolute difference between the corresponding expectations is bounded by
\begin{align*}
    2\lVert f\rVert_\infty\left(
    \PB(\tau_{\varepsilon,c}\le\eta)
    +\PB(\tau_{\varepsilon,c}>M)\right).
\end{align*}
Combining this bound with~\eqref{eq: valid_localized_bl}, then using
\eqref{eq: valid_lower_localization},
\eqref{eq: valid_upper_localization}, and the preceding bound on
$\PB(G_{\eta,M}^c)$, and finally letting $\eta\downarrow0$ and
$M\to\infty$, proves
\begin{align}
\label{eq: valid_2}
    \left(
    a_\varepsilon^{-h}\ell(a_\varepsilon)
    \left(Z_1(a_\varepsilon\,\cdot)
    -\idx(a_\varepsilon\,\cdot)\mu\right),
    V_\varepsilon(\cdot),\tau_{\varepsilon,c}\right)
    \Rightarrow
    \left(Y_1(\cdot),\sqrt{Y_2(\cdot)},\Phi_c(1)\right)
\end{align}
in the product topology. To justify evaluation at the limiting stopping
time, Theorem~\ref{thm: stop_prop}(d) shows that $\Phi_c(1)$ is an
$\mathbb F^Y$-predictable stopping time, so
Assumption~\ref{asmp: fclt}(d) gives
$\Delta Y_1(\Phi_c(1))=0$ almost surely. Since $\Phi_c(1)>0$ almost
surely, Assumption~\ref{asmp: fclt}(e) and
Lemma~\ref{lem: continuity_stop} also show that $Y_2$ is almost surely
continuous at $\Phi_c(1)$. Thus, random-time evaluation is continuous
under $M_2$, and~\eqref{eq: valid_2} gives
\begin{align*}
    \left(
    a_\varepsilon^{-h}\ell(a_\varepsilon)
    \left(Z_1(a_\varepsilon\tau_{\varepsilon,c})
    -a_\varepsilon\tau_{\varepsilon,c}\mu\right),
    V_\varepsilon(\tau_{\varepsilon,c})
    \right)
    \Rightarrow
    \left(
    Y_1(\Phi_c(1)),
    \sqrt{Y_2(\Phi_c(1))}
    \right).
\end{align*}
Since the limiting denominator is positive almost surely,
\begin{align*}
    \PB\left(Z_2(T_c(\varepsilon))=0\right)\to0.
\end{align*}
Since
$a_\varepsilon\tau_{\varepsilon,c}=T_c(\varepsilon)$ and
$\idx$ is the identity map, self-normalization, with the zero-denominator
convention in Theorem~\ref{thm: asymp_valid}, yields
\begin{align}
\label{eq: valid_3}
    \frac{Z_1(T_c(\varepsilon))-T_c(\varepsilon)\mu}
    {\sqrt{Z_2(T_c(\varepsilon))}}
    \tod
    \frac{Y_1(\Phi_c(1))}{\sqrt{Y_2(\Phi_c(1))}}.
\end{align}

\paragraph*{Proof of invariance in $c$.} Let
\begin{align}
    W^\star(u)
    :=\frac{Y_1(\Phi_1(u))}{\sqrt{Y_2(\Phi_1(u))}},
    \qquad u>0,
\end{align}
where $\Phi_1$ denotes $\Phi_c$ with $c=1$. By the joint self-similarity in Assumption~\ref{asmp: fclt}(c), for any $a>0$,
\begin{align}
    \left(Y_1(\cdot),\sqrt{Y_2(\cdot)}\right)
    \overset{d}{=}
    \left(a^{-h}Y_1(a\,\cdot),a^{-h}\sqrt{Y_2(a\,\cdot)}\right).
\end{align}
Since the stopping functional is measurable, the preceding equality in
distribution gives
\begin{align}
    \left(Y_1(\cdot),\sqrt{Y_2(\cdot)},\Phi_1(1)\right)
    &\overset{d}{=}
    \left(a^{-h}Y_1(a\,\cdot),a^{-h}\sqrt{Y_2(a\,\cdot)},
    \inf\left\{s>0:a^{-h}\frac{\sqrt{Y_2(as)}}{s}<1\right\}\right)\\
    &=
    \left(a^{-h}Y_1(a\,\cdot),a^{-h}\sqrt{Y_2(a\,\cdot)},
    \frac{1}{a}\Phi_1(a^{h-1})\right).
\end{align}
Applying the random time change and self-normalization gives
\begin{align}
    W^\star(1)\overset{d}{=}W^\star(a^{h-1}).
\end{align}
Since $a>0$ is arbitrary and $h\ne1$, $W^\star(u)$ has the same distribution for every $u>0$. Finally,
\begin{align*}
    \Phi_c(1)=\Phi_1(c^{-1})
    \quad\text{and}\quad
    W_c=W^\star(c^{-1}),
\end{align*}
so the distribution of $W_c$ does not depend on $c$.

\paragraph*{Proof of asymptotic validity.} Set
\begin{align*}
    S_\varepsilon
    =\frac{Z_1(T_{c^*}(\varepsilon))-T_{c^*}(\varepsilon)\mu}
    {\sqrt{Z_2(T_{c^*}(\varepsilon))}}.
\end{align*}
By~\eqref{eq: valid_3}, $S_\varepsilon$ converges weakly to a random variable with distribution function $F$. Moreover,
on $\{Z_2(T_{c^*}(\varepsilon))>0\}$,
\begin{align*}
    \left\{\mu\in\text{CI}(\varepsilon)\right\}
    =\left\{c_l\le S_\varepsilon\le c_u\right\}.
\end{align*}
Consequently,
\begin{align*}
\left|\PB\left(\mu\in\text{CI}(\varepsilon)\right)
-\PB\left(c_l\le S_\varepsilon\le c_u\right)\right|
&\le\PB\left(Z_2(T_{c^*}(\varepsilon))=0\right)\to0.
\end{align*}
Continuity of $F$ at $c_l$ and $c_u$ therefore gives
\begin{align*}
    \PB\left(\mu\in\text{CI}(\varepsilon)\right)
    \to F(c_u)-F(c_l)=1-\delta.
\end{align*}

\subsubsection{Proof of Theorem~\ref{thm: consistency}}
We continue the global and local simulation streams after
Algorithm~\ref{alg: general} terminates so that the calibration statistics
below are defined for every $n$. This proof-only continuation does not alter
the algorithm. All suprema indexed by $n$ below are taken over
$n\in\mathbb N$. According to Algorithm~\ref{alg: general}, for $j\ge1$,
the stopping time of the $j$th local procedure can be written as
\begin{align*}
    T^{(j)}(\varepsilon^r)
    =\inf\left\{m\in\mathbb N:
    m>\varepsilon^{-r\nu},\quad
    \frac{\sqrt{Z_2^{(j)}(m)}}{m}<\varepsilon^r
    \right\}.
\end{align*}
Define
\begin{align*}
    B_n(\varepsilon)
    :=\max\left\{b\in\mathbb N_0:
    \sum_{j=1}^bT^{(j)}(\varepsilon^r)\le n\right\},
\end{align*}
where the empty sum is zero. Thus, $B_n(\varepsilon)$ is the number of
local procedures completed by global time $n$.
When $B_n(\varepsilon)\ge1$, define the empirical distribution, its
quantiles, and $c_n(\varepsilon)$ by the same formulas as in
Algorithm~\ref{alg: general}.
As a proof convention, when $B_n(\varepsilon)=0$, we set
\begin{align*}
    \widehat F_n(x;\varepsilon)&:=\mathbbm{1}\{x\ge0\},
    &c_{l,n}(\varepsilon)=c_{u,n}(\varepsilon)&:=0,
    &c_n(\varepsilon)&:=\underline c.
\end{align*}
The global stopping test is inactive until $B_n(\varepsilon)\ge1$.
Consequently, the rule implemented by Algorithm~\ref{alg: general} is
\begin{align*}
    T(\varepsilon)
    :=\inf\left\{n\in\mathbb N:
    n>\varepsilon^{-\nu},\quad
    B_n(\varepsilon)\ge1,\quad
    c_n(\varepsilon)\frac{\sqrt{Z_2(n)}}{n}<\varepsilon
    \right\}.
\end{align*}
We first show that the online calibration has stabilized before the main procedure can stop.

\paragraph*{Number of completed local procedures.}
The natural scale of a local stopping time is $a_{\varepsilon^r}$.
Since $A^{\leftarrow}$ is regularly varying with index $\frac{1}{1-h}$ and
$r<1$,
\begin{align}
    \frac{a_{\varepsilon^r}}{a_\varepsilon}\to0
    \qquad\text{as }\varepsilon\downarrow0.
    \label{eq: local_global_scale}
\end{align}
For every fixed $N\ge1$, Theorem~\ref{thm: stop_prop}(a), the assumed
almost-sure continuity of $\Phi_1(\cdot)$ at $1$, and
\eqref{eq: local_global_scale} give
\begin{align*}
    \frac{1}{a_\varepsilon}
    \sum_{j=1}^N T^{(j)}(\varepsilon^r)
    \topb0
    \qquad\text{as }\varepsilon\downarrow0.
\end{align*}
Consequently, for every $\eta>0$,
\begin{align*}
    \PB\left(B_{\lfloor\eta a_\varepsilon\rfloor}(\varepsilon)<N\right)
    &=
    \PB\left(\sum_{j=1}^NT^{(j)}(\varepsilon^r)
    >\lfloor\eta a_\varepsilon\rfloor\right)
    \to0
    \qquad\text{as }\varepsilon\downarrow0.
\end{align*}
Since $N$ is arbitrary, for every $\eta>0$,
\begin{align}
    B_{\lfloor\eta a_\varepsilon\rfloor}(\varepsilon)
    \topb+\infty
    \qquad\text{as }\varepsilon\downarrow0.
    \label{eq: local_count_uniform}
\end{align}

\paragraph*{Oracle empirical distribution.}
Define the oracle terminal statistics
\begin{align*}
    \overline\omega_{j,\varepsilon}
    :=\begin{cases}
    \displaystyle
    \frac{Z_1^{(j)}(T^{(j)}(\varepsilon^r))
    -T^{(j)}(\varepsilon^r)\mu}
    {\sqrt{Z_2^{(j)}(T^{(j)}(\varepsilon^r))}},
    &Z_2^{(j)}(T^{(j)}(\varepsilon^r))>0,\\[8pt]
    0,&Z_2^{(j)}(T^{(j)}(\varepsilon^r))=0,
    \end{cases}
\end{align*}
and let $F_\varepsilon$ be their common distribution function. The local streams are independent, so the variables $\overline\omega_{j,\varepsilon}$ are i.i.d. for each $\varepsilon$. Theorem~\ref{thm: asymp_valid} gives
\begin{align*}
    \overline\omega_{1,\varepsilon}\tod W_1.
\end{align*}
Because $F$ is continuous, P\'olya's theorem yields
\begin{align}
    \sup_{x\in\RB}|F_\varepsilon(x)-F(x)|\to0.
    \label{eq: local_polya}
\end{align}
For $m\ge1$, define the oracle empirical distribution by
\begin{align*}
    G_{m,\varepsilon}(x)
    :=\frac{1}{m}\sum_{j=1}^m
    \mathbbm{1}\{\overline\omega_{j,\varepsilon}\le x\}.
\end{align*}
The Dvoretzky--Kiefer--Wolfowitz inequality and a union bound
imply that, for every $z>0$ and $N\ge1$,
\begin{align}
    \PB\left(\sup_{m\ge N}\sup_{x\in\RB}
    |G_{m,\varepsilon}(x)-F_\varepsilon(x)|>z\right)
    &\le\sum_{m=N}^{\infty}2\exp(-2mz^2).
    \label{eq: sequential_dkw}
\end{align}
The bound is uniform in $\varepsilon$ and converges to zero as $N\to\infty$.

\paragraph*{Replacement of $\mu$ by the global estimator.}
For $n\ge1$ and every $j\le B_n(\varepsilon)$,
Algorithm~\ref{alg: general} uses
\begin{align}
    \omega_{j,n}(\varepsilon)
    &=\overline\omega_{j,\varepsilon}
    -R_{j,\varepsilon}
    \frac{\frac{Z_1(n)}{n}-\mu}{\varepsilon^r},
    \label{eq: online_oracle_difference}\\
    R_{j,\varepsilon}
    &:=\begin{cases}
    \displaystyle
    \frac{\varepsilon^rT^{(j)}(\varepsilon^r)}
    {\sqrt{Z_2^{(j)}(T^{(j)}(\varepsilon^r))}},
    &Z_2^{(j)}(T^{(j)}(\varepsilon^r))>0,\\[8pt]
    0,&Z_2^{(j)}(T^{(j)}(\varepsilon^r))=0.
    \end{cases}
\end{align}
Fix $0<\eta<M<\infty$ and set
\begin{align*}
    D_\varepsilon(\eta,M)
    :=\frac{1}{\varepsilon^r}
    \sup_{\eta a_\varepsilon\le n\le Ma_\varepsilon}
    \left|\frac{Z_1(n)}{n}-\mu\right|.
\end{align*}
For $z>0$ and $n\in[\eta a_\varepsilon,Ma_\varepsilon]$ such that
$B_n(\varepsilon)\ge1$, let
\begin{align*}
    \Gamma_{n,\varepsilon}(z)
    :=\frac{1}{B_n(\varepsilon)}
    \sum_{j=1}^{B_n(\varepsilon)}
    \mathbbm{1}\left\{
    |\omega_{j,n}(\varepsilon)-\overline\omega_{j,\varepsilon}|>z
    \right\}.
\end{align*}
On the event
$\{B_{\lfloor\eta a_\varepsilon\rfloor}(\varepsilon)\ge1\}$,
the indicator inequalities give, for every
$n\in[\eta a_\varepsilon,Ma_\varepsilon]$ and $x\in\RB$,
\begin{align*}
    G_{B_n(\varepsilon),\varepsilon}(x-z)
    -\Gamma_{n,\varepsilon}(z)
    \le \widehat F_n(x;\varepsilon)
    \le G_{B_n(\varepsilon),\varepsilon}(x+z)
    +\Gamma_{n,\varepsilon}(z).
\end{align*}
On the same event, it follows that
\begin{align}
    &\sup_{\eta a_\varepsilon\le n\le Ma_\varepsilon}
    \sup_{x\in\RB}
    |\widehat F_n(x;\varepsilon)-F(x)|\notag\\
    &\quad\le
    \underbrace{
    \sup_{\eta a_\varepsilon\le n\le Ma_\varepsilon}
    \sup_{x\in\RB}
    \left|G_{B_n(\varepsilon),\varepsilon}(x)-F_\varepsilon(x)\right|
    }_{\text{oracle empirical error}}
    +\underbrace{
    \sup_{x\in\RB}|F_\varepsilon(x)-F(x)|
    }_{\text{local limiting error}}\notag\\
    &\qquad+
    \underbrace{
    \sup_{\eta a_\varepsilon\le n\le Ma_\varepsilon}
    \Gamma_{n,\varepsilon}(z)
    }_{\text{replacement error}}
    +\underbrace{
    \sup_{x\in\RB}\bigl(F(x+z)-F(x-z)\bigr)
    }_{\text{continuity error}}.
    \label{eq: replacement_decomposition}
\end{align}
We now control these four terms separately.
First, for every $\rho>0$ and $N\ge1$,~\eqref{eq: sequential_dkw} and the monotonicity of $B_n(\varepsilon)$ give
\begin{align*}
    &\PB\left(
    B_{\lfloor\eta a_\varepsilon\rfloor}(\varepsilon)\ge1,
    \sup_{\eta a_\varepsilon\le n\le Ma_\varepsilon}
    \sup_{x\in\RB}
    \left|G_{B_n(\varepsilon),\varepsilon}(x)-F_\varepsilon(x)\right|>\rho
    \right)\\
    &\le\PB\left(B_{\lfloor\eta a_\varepsilon\rfloor}(\varepsilon)<N\right)
    +\sum_{m=N}^{\infty}2\exp(-2m\rho^2).
\end{align*}
The first probability converges to zero by~\eqref{eq: local_count_uniform}. Taking first $\varepsilon\downarrow0$ and then $N\to\infty$ gives
\begin{align}
    \PB\left(
    B_{\lfloor\eta a_\varepsilon\rfloor}(\varepsilon)\ge1,
    \sup_{\eta a_\varepsilon\le n\le Ma_\varepsilon}
    \sup_{x\in\RB}
    \left|G_{B_n(\varepsilon),\varepsilon}(x)-F_\varepsilon(x)\right|>\rho
    \right)&\to0.
    \label{eq: oracle_empirical_uniform}
\end{align}
Second,~\eqref{eq: local_polya} gives
\begin{align}
    \sup_{x\in\RB}|F_\varepsilon(x)-F(x)|
    &\to0.
    \label{eq: local_limiting_uniform}
\end{align}
Thus, for every $\rho>0$, the local limiting error is smaller than $\rho$ for all sufficiently small $\varepsilon$.
Third, fix $z>0$. We show that the replacement error in~\eqref{eq: replacement_decomposition} is $o_\PB(1)$. Applying the joint convergence in~\eqref{eq: valid_2} and the random-time change to $Z_2$, with $\varepsilon$ replaced by $\varepsilon^r$ and $c=1$, gives
\begin{align*}
    R_{1,\varepsilon}\tod
    \frac{\Phi_1(1)}{\sqrt{Y_2(\Phi_1(1))}},
\end{align*}
which implies that $R_{1,\varepsilon}$ is asymptotically tight. For
$m\ge1$, let
\begin{align*}
    H_{m,\varepsilon}(K)
    :=\frac{1}{m}\sum_{j=1}^m
    \mathbbm{1}\{R_{j,\varepsilon}>K\}.
\end{align*}
Assumption~\ref{asmp: fclt} also gives
\begin{align*}
    \sup_{\eta a_\varepsilon\le n\le Ma_\varepsilon}
    |Z_1(n)-n\mu|
    =\OM_{\PB}\left(\frac{a_\varepsilon^{h}}{\ell(a_\varepsilon)}\right).
\end{align*}
Since $\varepsilon a_\varepsilon^{1-h}\ell(a_\varepsilon)\to1$, it follows that
\begin{align}
    D_\varepsilon(\eta,M)
    &=\OM_{\PB}\left(
    \frac{\varepsilon^{-r}}
    {a_\varepsilon^{1-h}\ell(a_\varepsilon)}
    \right)
    =\OM_{\PB}(\varepsilon^{1-r})
    =o_\PB(1).
    \label{eq: global_centering_uniform}
\end{align}
By~\eqref{eq: online_oracle_difference}, on the event
$\{B_{\lfloor\eta a_\varepsilon\rfloor}(\varepsilon)\ge1,
K D_\varepsilon(\eta,M)\le z\}$, we have
\begin{align*}
    \sup_{\eta a_\varepsilon\le n\le Ma_\varepsilon}
    \Gamma_{n,\varepsilon}(z)
    \le
    \sup_{\eta a_\varepsilon\le n\le Ma_\varepsilon}
    H_{B_n(\varepsilon),\varepsilon}(K).
\end{align*}
Fix $\rho>0$. By the tightness of $R_{1,\varepsilon}$, choose $K<\infty$ such that $\PB(R_{1,\varepsilon}>K)\le\frac{\rho}{2}$ for all sufficiently small $\varepsilon$. For every $N\ge1$ and all sufficiently small $\varepsilon$,
\begin{align*}
&\PB\left(
    B_{\lfloor\eta a_\varepsilon\rfloor}(\varepsilon)\ge1,
    \sup_{\eta a_\varepsilon\le n\le Ma_\varepsilon}
    \Gamma_{n,\varepsilon}(z)>\rho
    \right)\\
&\le\PB\bigl(KD_\varepsilon(\eta,M)>z\bigr)
    +\PB\left(
    B_{\lfloor\eta a_\varepsilon\rfloor}(\varepsilon)\ge1,
    \sup_{\eta a_\varepsilon\le n\le Ma_\varepsilon}
    H_{B_n(\varepsilon),\varepsilon}(K)>\rho
    \right)\\
&\le\PB\bigl(KD_\varepsilon(\eta,M)>z\bigr)
    +\PB\left(B_{\lfloor\eta a_\varepsilon\rfloor}(\varepsilon)<N\right)
    +\PB\left(\sup_{m\ge N}H_{m,\varepsilon}(K)>\rho\right)\\
&\le\PB\bigl(KD_\varepsilon(\eta,M)>z\bigr)
    +\PB\left(B_{\lfloor\eta a_\varepsilon\rfloor}(\varepsilon)<N\right)
    +\sum_{m=N}^{\infty}2\exp\left(-\frac{m\rho^2}{2}\right).
\end{align*}
The last inequality follows from Hoeffding's inequality and a union bound.
The first two probabilities converge to zero by~\eqref{eq: global_centering_uniform} and~\eqref{eq: local_count_uniform}, respectively. Taking first $\varepsilon\downarrow0$ and then $N\to\infty$ proves that, for every fixed $z>0$,
\begin{align}
    \PB\left(
    B_{\lfloor\eta a_\varepsilon\rfloor}(\varepsilon)\ge1,
    \sup_{\eta a_\varepsilon\le n\le Ma_\varepsilon}
    \Gamma_{n,\varepsilon}(z)>\rho
    \right)&\to0.
    \label{eq: replacement_error_uniform}
\end{align}
Fourth, because a continuous distribution function is uniformly continuous, for every $\rho>0$ we can choose
$z=z_\rho>0$ such that
\begin{align*}
    \sup_{x\in\RB}\bigl(F(x+z)-F(x-z)\bigr)<\rho.
\end{align*}
Fix this choice of $z$ and take $\varepsilon$ sufficiently small that the local limiting error in~\eqref{eq: local_limiting_uniform} is less than $\rho$. Then~\eqref{eq: replacement_decomposition} gives
\begin{align*}
&\PB\left(
    \sup_{\eta a_\varepsilon\le n\le Ma_\varepsilon}
    \sup_{x\in\RB}
    |\widehat F_n(x;\varepsilon)-F(x)|>4\rho
    \right)\\
\le&\PB\left(B_{\lfloor\eta a_\varepsilon\rfloor}(\varepsilon)=0\right)\\
+&\PB\left(
    B_{\lfloor\eta a_\varepsilon\rfloor}(\varepsilon)\ge1,
    \sup_{\eta a_\varepsilon\le n\le Ma_\varepsilon}
    \sup_{x\in\RB}
    \left|G_{B_n(\varepsilon),\varepsilon}(x)-F_\varepsilon(x)\right|>\rho
    \right)\\
+&\PB\left(
    B_{\lfloor\eta a_\varepsilon\rfloor}(\varepsilon)\ge1,
    \sup_{\eta a_\varepsilon\le n\le Ma_\varepsilon}
    \Gamma_{n,\varepsilon}(z)>\rho
    \right).
\end{align*}
The three probabilities on the right converge to zero by
\eqref{eq: local_count_uniform}, \eqref{eq: oracle_empirical_uniform},
and~\eqref{eq: replacement_error_uniform}, respectively. Since $\rho>0$
is arbitrary, this proves
\begin{align}
    \sup_{\eta a_\varepsilon\le n\le Ma_\varepsilon}
    \sup_{x\in\RB}
    |\widehat F_n(x;\varepsilon)-F(x)|
    &\topb0.
    \label{eq: online_cdf_uniform}
\end{align}
By~\eqref{eq: online_cdf_uniform} and conditions~(c) and~(e), for every
$0<\eta<M<\infty$,
\begin{align}
    \sup_{\eta a_\varepsilon\le n\le Ma_\varepsilon}
    \left(
    |c_{l,n}(\varepsilon)-c_l|
    +|c_{u,n}(\varepsilon)-c_u|
    +|c_n(\varepsilon)-c^*|
    \right)\topb0.
    \label{eq: calibration_uniform}
\end{align}

\paragraph*{Online stopping time.}
Fix $0<\zeta<c^*-\underline c$. Since
$c_n(\varepsilon)\ge\underline c$, we have
\begin{align*}
    T_{\underline c}(\varepsilon)\le T(\varepsilon).
\end{align*}
Applying Assumption~\ref{asmp: early_stopping} with
$K=\underline c^{-1}$ and Theorem~\ref{thm: stop_prop}(a) with
coefficient $c^*+\zeta$, we have
\begin{align}
    \lim_{\eta\downarrow0}\limsup_{\varepsilon\downarrow0}
    \PB\left(T_{\underline c}(\varepsilon)<\eta a_\varepsilon\right)
    &=0,
    &
    \lim_{M\uparrow\infty}\limsup_{\varepsilon\downarrow0}
    \PB\left(T_{c^*+\zeta}(\varepsilon)>Ma_\varepsilon\right)
    &=0.
    \label{eq: fixed_stop_localization}
\end{align}
Define
\begin{align*}
    E_{\varepsilon,\eta,M}:=\left\{T_{\underline c}(\varepsilon)\ge\eta a_\varepsilon,
    \quad T_{c^*+\zeta}(\varepsilon)\le Ma_\varepsilon,\quad
    \sup_{\eta a_\varepsilon\le n\le Ma_\varepsilon}
    |c_n(\varepsilon)-c^*|\le\zeta\right\}.
\end{align*}
On $E_{\varepsilon,\eta,M}$, monotonicity gives
$T_{\underline c}(\varepsilon)\le T_{c^*+\zeta}(\varepsilon)$,
so
$T_{c^*+\zeta}(\varepsilon)\in[\eta a_\varepsilon,Ma_\varepsilon]$.
For every $n\in[\eta a_\varepsilon,Ma_\varepsilon]$, we also have
$B_n(\varepsilon)\ge1$ and $c_n(\varepsilon)\le c^*+\zeta$, which gives
$T(\varepsilon)\le T_{c^*+\zeta}(\varepsilon)$. Together with the
pathwise lower bound, this gives
\begin{align*}
    \eta a_\varepsilon
    \le T_{\underline c}(\varepsilon)
    \le T(\varepsilon)
    \le T_{c^*+\zeta}(\varepsilon)
    \le Ma_\varepsilon.
\end{align*}
Since $c_{T(\varepsilon)}(\varepsilon)\ge c^*-\zeta$, we have
\begin{align}
    T_{c^*-\zeta}(\varepsilon)
    \le T(\varepsilon)
    \le T_{c^*+\zeta}(\varepsilon).
    \label{eq: online_sandwich}
\end{align}
By~\eqref{eq: fixed_stop_localization} and
\eqref{eq: calibration_uniform}, for every $\kappa>0$, we can choose
$0<\eta<M<\infty$ such that
\begin{align*}
    \liminf_{\varepsilon\downarrow0}
    \PB(E_{\varepsilon,\eta,M})\ge1-\kappa.
\end{align*}
Hence,~\eqref{eq: online_sandwich} holds with probability tending to one.
In particular, $\PB(T(\varepsilon)<\infty)\to1$, so the fixed-value
convention on nontermination is asymptotically irrelevant.
Monotonicity also gives
\begin{align*}
    T_{c^*-\zeta}(\varepsilon)
    \le T_{c^*}(\varepsilon)
    \le T_{c^*+\zeta}(\varepsilon).
\end{align*}
Consequently, on the event that~\eqref{eq: online_sandwich} holds,
\begin{align*}
    \frac{|T(\varepsilon)-T_{c^*}(\varepsilon)|}{a_\varepsilon}
    \le
    \frac{T_{c^*+\zeta}(\varepsilon)-T_{c^*-\zeta}(\varepsilon)}
    {a_\varepsilon}.
\end{align*}
Because each sample path of $\Phi_1(\cdot)$ is monotone, Fubini's theorem shows that, for Lebesgue-almost every $u>0$, it is almost surely continuous at $u$. We may therefore let $\zeta\downarrow0$ through values for which $\Phi_1(\cdot)$ is almost surely continuous at both $(c^*-\zeta)^{-1}$ and $(c^*+\zeta)^{-1}$. At each such $\zeta$, the continuous-mapping argument in the proof of Theorem~\ref{thm: stop_prop}(a) yields the joint convergence
\begin{align*}
    \left(
    \frac{T_{c^*-\zeta}(\varepsilon)}{a_\varepsilon},
    \frac{T_{c^*+\zeta}(\varepsilon)}{a_\varepsilon}
    \right)
    \Rightarrow
    \left(
    \Phi_1((c^*-\zeta)^{-1}),
    \Phi_1((c^*+\zeta)^{-1})
    \right).
\end{align*}
Since $\Phi_1(\cdot)$ is almost surely continuous at $(c^*)^{-1}$, both coordinates converge almost surely to $\Phi_1((c^*)^{-1})$ as $\zeta\downarrow0$. Letting first $\varepsilon\downarrow0$ and then $\zeta\downarrow0$, the converging-together theorem gives
\begin{align*}
    \frac{T(\varepsilon)-T_{c^*}(\varepsilon)}{a_\varepsilon}
    \topb0.
\end{align*}
Combining this equivalence with Theorem~\ref{thm: stop_prop}(a) and the assumed continuity of $\Phi_1(\cdot)$ at $(c^*)^{-1}$ gives
\begin{align}
    \frac{T(\varepsilon)}{a_\varepsilon}
    \tod\Phi_{c^*}(1).
    \label{eq: online_stop_limit}
\end{align}
Moreover,~\eqref{eq: valid_2} with $c=c^*$ and Slutsky's theorem give the joint convergence
\begin{align*}
    \left(
    a_\varepsilon^{-h}\ell(a_\varepsilon)
    \left(Z_1(a_\varepsilon\,\cdot)
    -\idx(a_\varepsilon\,\cdot)\mu\right),
    V_\varepsilon(\cdot),
    \frac{T(\varepsilon)}{a_\varepsilon}
    \right)
    \Rightarrow
    \left(
    Y_1(\cdot),\sqrt{Y_2(\cdot)},\Phi_{c^*}(1)
    \right).
\end{align*}
As shown in the proof of Theorem~\ref{thm: asymp_valid}, $\Phi_{c^*}(1)$ is almost surely a continuity point of both limiting processes. Random-time evaluation therefore yields
\begin{align}
    \frac{Z_1(T(\varepsilon))-T(\varepsilon)\mu}
    {\sqrt{Z_2(T(\varepsilon))}}
    \tod
    \frac{Y_1(\Phi_{c^*}(1))}
    {\sqrt{Y_2(\Phi_{c^*}(1))}}.
    \label{eq: online_terminal_limit}
\end{align}
The same joint convergence and the positivity of the limiting denominator
give
\begin{align*}
    \PB\left(Z_2(T(\varepsilon))=0\right)\to0.
\end{align*}
By the invariance established in Theorem~\ref{thm: asymp_valid}, the limit in~\eqref{eq: online_terminal_limit} has distribution function $F$. In addition,~\eqref{eq: online_stop_limit} implies that $\frac{T(\varepsilon)}{a_\varepsilon}$ is tight and bounded away from zero in probability. Thus, for every $\rho>0$ and $0<\eta<M<\infty$,
\begin{align*}
    &\PB\left(
    |c_{l,T(\varepsilon)}(\varepsilon)-c_l|
    +|c_{u,T(\varepsilon)}(\varepsilon)-c_u|>\rho
    \right)\\
    \le&\PB\left(\frac{T(\varepsilon)}{a_\varepsilon}\notin[\eta,M]\right)\\
    +&\PB\left(
    \sup_{\eta a_\varepsilon\le n\le Ma_\varepsilon}
    \left(
    |c_{l,n}(\varepsilon)-c_l|
    +|c_{u,n}(\varepsilon)-c_u|
    \right)>\rho
    \right).
\end{align*}
The second probability converges to zero by~\eqref{eq: calibration_uniform}; letting $\eta\downarrow0$ and $M\uparrow\infty$ then makes the first probability arbitrarily small. This proves~\eqref{eq: terminal_quantile_consistency}. On $\{Z_2(T(\varepsilon))>0\}$, the coverage event equals
\begin{align*}
    \left\{
    c_{l,T(\varepsilon)}(\varepsilon)
    \le
    \frac{Z_1(T(\varepsilon))-T(\varepsilon)\mu}
    {\sqrt{Z_2(T(\varepsilon))}}
    \le c_{u,T(\varepsilon)}(\varepsilon)
    \right\}.
\end{align*}
Therefore,~\eqref{eq: online_terminal_limit},
\eqref{eq: terminal_quantile_consistency}, the preceding zero-denominator
bound, and continuity of $F$ at $c_l$ and $c_u$ imply
\eqref{eq: feasible_coverage}.


\subsection{Proofs for Section~\ref{sec: alt}}
\subsubsection{Proof of Theorem~\ref{thm: iid_finite}}
\paragraph{Verification of Assumption~\ref{asmp: fclt}.}
Donsker's theorem~\cite{donsker1951invariance} and the functional law of large numbers for the sample
variance give the joint functional convergence.
The strong law gives
\begin{align*}
    \frac{Z_1(t)-\mu t}{t}\to0,
    \qquad
    \frac{Z_2(t)}{t^2}\to0
    \qquad\text{almost surely}.
\end{align*}
Thus Assumption~\ref{asmp: fclt}(b) holds.
The Brownian limit is continuous and self-similar with index $\frac{1}{2}$, while
$Y_2(t)=\sigma^2t$ is continuous, positive for $t>0$, and self-similar with
index one. This verifies all remaining parts of
Assumption~\ref{asmp: fclt}.

\paragraph{Verification of Assumption~\ref{asmp: early_stopping}.}
Fix $K>0$ and take $0<\eta<\frac{\sigma^2}{2K^2}$. Since
we may take $A^{\leftarrow}(x)=x^2$ and hence
$a_\varepsilon=\varepsilon^{-2}$, any integer
$\varepsilon^{-\nu}\le n\le\eta a_\varepsilon$ satisfying
$\frac{\sqrt{Z_2(n)}}{n}\le K\varepsilon$ also satisfies
$\frac{Z_2(n)}{n}\le K^2\eta<\frac{\sigma^2}{2}$. Therefore,
\begin{align*}
    \PB\left(\inf_{\varepsilon^{-\nu}\le n\le\eta a_\varepsilon}
    \frac{\sqrt{Z_2(n)}}{n}\le K\varepsilon\right)
    \le\PB\left(\inf_{n\ge\lceil\varepsilon^{-\nu}\rceil}
    \frac{Z_2(n)}{n}<\frac{\sigma^2}{2}\right)\to0.
\end{align*}
The convergence follows from $\frac{Z_2(n)}{n}\to\sigma^2$ almost surely.
Letting $\eta\downarrow0$ proves
Assumption~\ref{asmp: early_stopping}.

\subsubsection{Proof of Theorem~\ref{thm: iid_infinite}}
\paragraph{Verification of Assumption~\ref{asmp: fclt}.}
Let $b$ be a regularly varying norming function of index $\frac{1}{\alpha}$
associated with Assumption~\ref{asmp: domain}, and take
$\ell(t)=\frac{t^{\frac{1}{\alpha}}}{b(t)}$. Point-process convergence
\citep{resnick1986point} gives the joint functional convergence of the
centered partial-sum process and the raw squared-sum process. For every
$n\ge1$, the least-squares identity gives
\begin{align*}
    Z_2(n)=\sum_{i=1}^n(X_i-\mu)^2
    -\frac{(Z_1(n)-n\mu)^2}{n}.
\end{align*}
On compact time intervals bounded away from zero, the second term is
negligible uniformly on the $b(t)^2$ scale by tightness of the first
coordinate. Near zero,
\begin{align*}
    \max_{1\le n\le\lfloor t\eta\rfloor}
    \frac{(Z_1(n)-n\mu)^2}{n b(t)^2}
    \le\frac{1}{b(t)^2}\sum_{i=1}^{\lfloor t\eta\rfloor}(X_i-\mu)^2,
\end{align*}
and the right-hand side converges to $L_2(\eta)$, which tends to zero almost
surely as $\eta\downarrow0$. Thus sample centering does not change the joint
functional limit. The difference between centering by $\mu\lfloor
t\,\cdot\rfloor$ and by $\mu t\,\cdot$ is also negligible. Since
$\alpha>1$, the strong law gives $\frac{Z_1(t)-\mu t}{t}\to0$ almost surely. The
least-squares property and Lemma~\ref{lem:subquadratic_squares} give
\begin{align*}
    0\le\frac{Z_2(n)}{n^2}
    \le\frac{1}{n^2}\sum_{i=1}^n(X_i-\mu)^2\to0
    \qquad\text{almost surely}.
\end{align*}
Thus Assumption~\ref{asmp: fclt}(b) holds.
Finally, $(L_1,L_2)$ is a jointly self-similar L\'evy process and is
therefore quasi-left-continuous with respect to its augmented natural
filtration. The subordinator $L_2$ has no negative jumps and is almost
surely positive at every positive time. This verifies the remaining parts of
Assumption~\ref{asmp: fclt}.

\paragraph{Verification of Assumption~\ref{asmp: early_stopping}.}
Let
\begin{align*}
    \Delta_n:=\max_{1\le i\le n}X_i-\min_{1\le i\le n}X_i.
\end{align*}
By $a^2+b^2\ge\frac{(a-b)^2}{2}$,
\begin{align}
    Z_2(n)
    &\ge\left(\max_{1\le i\le n}X_i-\frac{Z_1(n)}{n}\right)^2
    +\left(\min_{1\le i\le n}X_i-\frac{Z_1(n)}{n}\right)^2
    \ge\frac{\Delta_n^2}{2}.
    \label{eq: iid_range_bound}
\end{align}
By the tail characterization of stable domains of attraction and the
convolution closure of regularly varying tails
\citep{geluk2000stable,foss2009convolutions}, we have:
\begin{align*}
    \PB(|X_1-X_2|>x)\sim2\PB(|X_1|>x),
    \qquad
    t\PB(|X_1-X_2|>b(t))\asymp1.
\end{align*}
Here $b$ satisfies $b(a_\varepsilon)\sim\varepsilon a_\varepsilon$ since
$A(t)=t^{1-\frac{1}{\alpha}}\ell(t)=\frac{t}{b(t)}$. Fix $K>0$ and
$0<\delta<\alpha-1$, and let
$m_\varepsilon=\lceil\varepsilon^{-\nu}\rceil$. Regular variation of
$A$ and $\nu<\frac{\alpha}{\alpha-1}$ give
$\frac{m_\varepsilon}{a_\varepsilon}\to0$. Fix $\eta>0$ and set
$r_j=2^jm_\varepsilon$. For all sufficiently small $\varepsilon$, let
$J_\varepsilon$ be the largest integer
such that $r_{J_\varepsilon}\le\eta a_\varepsilon$. For
$n\in[r_j,2r_j)$, if $\frac{\sqrt{Z_2(n)}}{n}\le K\varepsilon$,
\eqref{eq: iid_range_bound} and monotonicity of the sample range give
\begin{align*}
    \Delta_{r_j}\le\Delta_n\le\sqrt{2Z_2(n)}
    \le\sqrt{2}K\varepsilon n<2\sqrt{2}K\varepsilon r_j.
\end{align*}
Since the dyadic blocks cover $[m_\varepsilon,\eta a_\varepsilon]$,
the union bound gives
\begin{align}
\PB\left(\inf_{\varepsilon^{-\nu}\le n\le\eta a_\varepsilon}
    \frac{\sqrt{Z_2(n)}}{n}\le K\varepsilon\right)
&\le\sum_{j=0}^{J_\varepsilon}
    \PB\left(\inf_{r_j\le n<2r_j}
    \frac{\sqrt{Z_2(n)}}{n}\le K\varepsilon\right)
    \notag\\
&\le\sum_{j=0}^{J_\varepsilon}
    \PB\left(\Delta_{r_j}\le2\sqrt{2}K\varepsilon r_j\right)
    \notag\\
&\overset{(a)}{\le}\sum_{j=0}^{J_\varepsilon}
    \PB\left(\max_{1\le i\le\lfloor \frac{r_j}{2}\rfloor}
    |X_{2i-1}-X_{2i}|\le2\sqrt{2}K\varepsilon r_j\right)
    \notag\\
&=\sum_{j=0}^{J_\varepsilon}
    \left(1-\PB\left(|X_1-X_2|>2\sqrt{2}K\varepsilon r_j\right)\right)^{\lfloor \frac{r_j}{2}\rfloor}
    \notag\\
&\le\sum_{j=0}^{J_\varepsilon}\exp\left(-\left\lfloor\frac{r_j}{2}\right\rfloor
    \PB\left(|X_1-X_2|>2\sqrt{2}K\varepsilon r_j\right)\right).
    \label{eq: iid_early_dyadic}
\end{align}
Inequality~(a) follows since
$|X_{2i-1}-X_{2i}|\le\Delta_{r_j}$ for every
$1\le i\le\lfloor \frac{r_j}{2}\rfloor$. We next bound the sum. Choose a
sufficiently large constant $M$. By~\eqref{eq: iid_early_dyadic},
\begin{align*}
\PB\left(\inf_{\varepsilon^{-\nu}\le n\le\eta a_\varepsilon}
\frac{\sqrt{Z_2(n)}}{n}\le K\varepsilon\right)
&\le\sum_{\substack{0\le j\le J_\varepsilon:\,
\varepsilon r_j\le M}}
\exp\left(-\left\lfloor\frac{r_j}{2}\right\rfloor
\PB\left(|X_1-X_2|>2\sqrt{2}K\varepsilon r_j\right)\right)\\
&\quad+\sum_{\substack{0\le j\le J_\varepsilon:\,
\varepsilon r_j>M}}
\exp\left(-\left\lfloor\frac{r_j}{2}\right\rfloor
\PB\left(|X_1-X_2|>2\sqrt{2}K\varepsilon r_j\right)\right).
\end{align*}
For the first sum,
$\PB\left(|X_1-X_2|>2\sqrt{2}KM\right)>0$, and hence, for some $c>0$,
\begin{align*}
\sum_{\substack{0\le j\le J_\varepsilon:\,
\varepsilon r_j\le M}}
\exp\left(-\left\lfloor\frac{r_j}{2}\right\rfloor
\PB\left(|X_1-X_2|>2\sqrt{2}K\varepsilon r_j\right)\right)
&\le\sum_{j=0}^{\infty}\exp(-c2^jm_\varepsilon).
\end{align*}
The last sum converges to zero. For the second sum, Potter's bound and
$b(a_\varepsilon)\sim\varepsilon a_\varepsilon$ give constants
$\eta_0,c,C,\varepsilon_0>0$. Since only $\eta\downarrow0$ is relevant,
assume $0<\eta\le\eta_0$. Then, for every
$0<\varepsilon\le\varepsilon_0$ and
$\frac{M}{\varepsilon}<r\le\eta a_\varepsilon$,
\begin{align*}
    \left\lfloor\frac{r}{2}\right\rfloor
    \PB\left(|X_1-X_2|>2\sqrt{2}K\varepsilon r\right)
    &\ge cr\PB\left(
    |X_1-X_2|>4\sqrt{2}K\frac{r}{a_\varepsilon}b(a_\varepsilon)
    \right)\\
    &\ge C\left(\frac{r}{a_\varepsilon}\right)^{1-\alpha+\delta}.
\end{align*}
Substituting the preceding bound into the second sum gives
\begin{align*}
&\sum_{\substack{0\le j\le J_\varepsilon:\,
\varepsilon r_j>M}}
\exp\left(-\left\lfloor\frac{r_j}{2}\right\rfloor
\PB\left(|X_1-X_2|>2\sqrt{2}K\varepsilon r_j\right)\right)\\
\le&\sum_{\substack{0\le j\le J_\varepsilon:\,
\varepsilon r_j>M}}
\exp\left(-C\left(\frac{r_j}{a_\varepsilon}
\right)^{1-\alpha+\delta}\right)\\
=&\sum_{\substack{0\le k\le J_\varepsilon:\,
\varepsilon r_{J_\varepsilon-k}>M}}
\exp\left(-C\left(\frac{r_{J_\varepsilon-k}}{a_\varepsilon}
\right)^{1-\alpha+\delta}\right).
\end{align*}
Since $r_{J_\varepsilon}\le\eta a_\varepsilon$,
$\frac{r_{J_\varepsilon-k}}{a_\varepsilon}\le\eta2^{-k}$, and therefore
\begin{align*}
\sum_{\substack{0\le k\le J_\varepsilon:\,
\varepsilon r_{J_\varepsilon-k}>M}}
\exp\left(-C\left(\frac{r_{J_\varepsilon-k}}{a_\varepsilon}
\right)^{1-\alpha+\delta}\right)
&\le\sum_{k=0}^{J_\varepsilon}
\exp\left(-C\eta^{-(\alpha-1-\delta)}
2^{k(\alpha-1-\delta)}\right)\\
&\le\sum_{k=0}^{\infty}
\exp\left(-C\eta^{-(\alpha-1-\delta)}
2^{k(\alpha-1-\delta)}\right).
\end{align*}
Combining these bounds yields
\begin{align*}
    \limsup_{\varepsilon\downarrow0}
    \PB\left(
    \inf_{\substack{n\in\mathbb N:\,
    \varepsilon^{-\nu}\le n\le\eta a_\varepsilon}}
    \frac{\sqrt{Z_2(n)}}{n}\le K\varepsilon
    \right)
    \le\sum_{k=0}^{\infty}
    \exp\left(-C\eta^{-(\alpha-1-\delta)}
    2^{k(\alpha-1-\delta)}\right).
\end{align*}
Since $C$ is independent of $\eta\in(0,\eta_0]$, the series on the right
converges to zero as $\eta\downarrow0$ by the dominated convergence theorem,
proving
Assumption~\ref{asmp: early_stopping}.

\subsubsection{Proof of Theorem~\ref{thm: replication}}
Independence gives joint convergence of the $m$ coordinates to
$(Y_1^{(1)},\ldots,Y_1^{(m)})$ in the product $M_2$ topology.
Since each path has countably many jumps, quasi-left-continuity and
independence imply that no two limiting coordinates jump together almost
surely. Corollary~12.11.4 of \cite{whitt2002stochastic} and the continuous
mapping theorem therefore strengthen the joint convergence to the
vector-valued $\mathrm{SM}_2$ topology, and
Theorem~12.11.4 of \cite{whitt2002stochastic} gives
\begin{align*}
    \varepsilon^{h}\ell(\varepsilon^{-1})\left(
    Z_1^{\mathrm{sec}}\left(\frac{\cdot}{\varepsilon}\right)
    -\idx\left(\frac{\cdot}{\varepsilon}\right)\mu
    \right)
    \overset{M_2}{\Rightarrow}
    \bar Y_1(\cdot).
\end{align*}
After the change of variables $s=\frac{r}{\varepsilon}$, the square of the second coordinate can be written as
\begin{align*}
    \bigl(\varepsilon^{h}\ell(\varepsilon^{-1})\bigr)^2
    Z_2^{*,\mathrm{sec}}
    \left(\frac{t}{\varepsilon}\right)
    =\frac{1}{m-1}\sum_{i=1}^m\frac{1}{t}
    \int_0^t
    \left(
    \varepsilon^{h}\ell(\varepsilon^{-1})
    \left[
    Z_1^{(i)}\left(\frac{r}{\varepsilon}\right)
    -\bar Z_1\left(\frac{r}{\varepsilon}\right)
    \right]
    \right)^2
    \,\mathrm dr.
\end{align*}
Theorem~11.5.1 of \cite{whitt2002stochastic}, with the continuous function
\begin{align*}
    (z_1,\ldots,z_m)\mapsto
    \frac{1}{m-1}\sum_{i=1}^m
    \left(z_i-\frac{1}{m}\sum_{j=1}^m z_j\right)^2,
\end{align*}
gives locally uniform convergence of the cumulative integrals. Division by
$t$ and the square-root map then give locally uniform convergence of the
second coordinate on every compact interval bounded away from zero. To handle
zero, consider any deterministic sequence
$(x_n^{(1)},\ldots,x_n^{(m)})\to(x^{(1)},\ldots,x^{(m)})$ in
$\mathrm{SM}_2$, where $x^{(i)}(0)=0$ for $1\le i\le m$. Then
\begin{align*}
\sup_{0<t\le\delta}\frac{1}{m-1}\sum_{i=1}^m\frac{1}{t}
    \int_0^t\left(x_n^{(i)}(s)-\frac{1}{m}
    \sum_{j=1}^m x_n^{(j)}(s)\right)^2\,\mathrm ds
&\le\frac{m}{m-1}\max_{1\le i\le m}
    \sup_{0\le s\le\delta}|x_n^{(i)}(s)|^2,
\end{align*}
whose right-hand side vanishes under
$\lim_{\delta\downarrow0}\limsup_{n\to\infty}$ by completed-graph convergence
and right-continuity at zero. The same bound holds for the limit. Hence the
continuous mapping theorem proves the asserted $\mathrm{WM}_2$ convergence.

The process $Y_1^{\mathrm{sec}}$ is $h$-self-similar, while a change of variables in the defining integral shows that $Y_2^{*,\mathrm{sec}}$ is $2h$-self-similar jointly with $Y_1^{\mathrm{sec}}$. The latter process is continuous. Moreover, $Y_1^{\mathrm{sec}}$ is quasi-left-continuous with respect to the filtration generated by the limiting vector and therefore also with respect to the smaller filtration generated by $(Y_1^{\mathrm{sec}},Y_2^{*,\mathrm{sec}})$. Since each copy starts at zero, $Z_2^{*,\mathrm{sec}}$ is nonnegative and continuous, including at time zero. The strong-law condition follows by averaging the corresponding condition for the $m$ copies. Write $e_i(t)=Z_1^{(i)}(t)-\mu t$. Then $e_i(t)=o(t)$ almost surely, and c\`adl\`ag local boundedness gives $\sup_{s\le t}|e_i(s)|=o(t)$ almost surely. Since $m$ is fixed, the same bound holds for every sectioning contrast, and therefore
\begin{align*}
    \frac{Z_2^{*,\mathrm{sec}}(t)}{t^2}
    \le \frac{C}{t^3}\int_0^t
    \max_{1\le i\le m}|e_i(s)|^2\,\mathrm ds
    \to0
    \qquad\text{almost surely}.
\end{align*}
This completes the verification of Assumption~\ref{asmp: fclt}.

\subsubsection{Proof of Theorem~\ref{thm: batch_mean}}
Because $\idx$ is the identity map, the linear centering terms cancel
exactly in every batch contrast. Consequently,
\begin{align*}
    \bigl(\varepsilon^{h}\ell(\varepsilon^{-1})\bigr)^2
    Z_2^{*,\mathrm{bm}}
    \left(\frac{t}{\varepsilon}\right)
    =\frac{1}{m-1}\sum_{k=1}^m\frac{1}{t}
    \int_0^t
    \left(
    \varepsilon^{h}\ell(\varepsilon^{-1})
    \left[
    Z_{1,k}^{\mathrm{bm}}\left(\frac{r}{\varepsilon}\right)
    -\frac{1}{m}Z_1\left(\frac{r}{\varepsilon}\right)
    \right]
    \right)^2
    \,\mathrm dr.
\end{align*}
Deterministic linear time changes preserve $M_2$ convergence.
The proof of Theorem~11.5.1 of \cite{whitt2002stochastic} applies to the
resulting finite product $M_2$ convergence and the continuous quadratic
function defining the batch contrasts, giving locally uniform convergence of
the cumulative integrals.
Division by $t$ and the square-root map apply away from zero, and the same
near-zero bound as in the proof of Theorem~\ref{thm: replication} handles
zero. Together with the first-coordinate convergence, this proves the claimed
$\mathrm{WM}_2$ convergence.

A change of variables shows that $(Y_1^{\mathrm{bm}},Y_2^{*,\mathrm{bm}})$ is jointly self-similar with indices $h$ and $2h$. The second coordinate is continuous and is adapted to the natural filtration of $Y_1$. Hence the augmented natural filtration of the pair does not enlarge that of $Y_1$, and the quasi-left-continuity of $Y_1^{\mathrm{bm}}=Y_1$ is preserved. Since $Z_1(0)=0$, the prelimit time average is also nonnegative and continuous at zero. If $e(t)=Z_1(t)-\mu t$ and $M(t)=\sup_{s\le t}|e(s)|$, then $M(t)=o(t)$ almost surely. Every batch contrast is bounded by $3M(t)$, and consequently
\begin{align*}
    \frac{Z_2^{*,\mathrm{bm}}(t)}{t^2}
    \le \frac{9m}{m-1}\frac{M(t)^2}{t^2}\to0
    \qquad\text{almost surely}.
\end{align*}
This completes the verification of Assumption~\ref{asmp: fclt}.

\subsubsection{Proof of Theorem~\ref{thm: random_scaling}}
The linear centering terms cancel exactly in each bridge because $\idx$ is
the identity map. After a change of variables, the rescaled pointwise random
scaling is
\begin{align*}
    Q_\varepsilon(u)
    &:=\bigl(\varepsilon^{h}\ell(\varepsilon^{-1})\bigr)^2
    Z_2^{\mathrm{rs}}
    \left(\frac{u}{\varepsilon}\right)\\
    &=\frac{1}{u}\int_0^u
    \left(
    \varepsilon^{h}\ell(\varepsilon^{-1})
    \left[
    Z_1\left(\frac{r}{\varepsilon}\right)
    -\frac{r}{u}Z_1\left(\frac{u}{\varepsilon}\right)
    \right]
    \right)^2\,\mathrm dr.
\end{align*}
To verify the continuous-mapping step, let $x_n\to x$ in $M_2$ on
compact time intervals, with $x_n(0)=x(0)=0$, and let $Q_n$ and $Q$ be
the corresponding pointwise random-scaling paths. At every continuity
point $u$ of $x$, completed-graph convergence gives $x_n(u)\to x(u)$,
and the same convergence holds for almost every $r$. The paths are uniformly
bounded on compact intervals, so dominated convergence gives
$Q_n(u)\to Q(u)$ for almost every $u>0$ and, for every $T<\infty$,
\begin{align*}
    \int_0^T|Q_n(u)-Q(u)|\,\mathrm du\to0.
\end{align*}
A second change of variables yields
\begin{align*}
    \bigl(\varepsilon^{h}\ell(\varepsilon^{-1})\bigr)^2
    Z_2^{*,\mathrm{rs}}
    \left(\frac{t}{\varepsilon}\right)
    =\frac{1}{t}\int_0^t Q_\varepsilon(u)\,\mathrm du.
\end{align*}
For every $0<\delta<T$, the preceding $L_1$ convergence gives
\begin{align*}
&\sup_{\delta\le t\le T}
\left|\frac{1}{t}\int_0^tQ_n(u)\,\mathrm du
-\frac{1}{t}\int_0^tQ(u)\,\mathrm du\right|
\le\frac{1}{\delta}\int_0^T|Q_n(u)-Q(u)|\,\mathrm du\to0.
\end{align*}
Near zero,
\begin{align*}
\sup_{0<t\le\delta}
    \left|\frac{1}{t}\int_0^tQ_n(u)\,\mathrm du
    -\frac{1}{t}\int_0^tQ(u)\,\mathrm du\right|
&\le4\sup_{0\le r\le\delta}|x_n(r)|^2
    +4\sup_{0\le r\le\delta}|x(r)|^2.
\end{align*}
Since $x$ is right-continuous at zero, completed-graph convergence gives
\begin{align*}
    \lim_{\delta\downarrow0}\limsup_{n\to\infty}
    \sup_{0\le r\le\delta}|x_n(r)|=0.
\end{align*}
Thus the time-smoothed random-scaling map is continuous from $M_2$ to the
local uniform topology on the subspace of paths starting from zero. The
normalized prelimit paths and their limit all lie in this subspace. The
continuous mapping theorem, followed by the square-root map, proves the
asserted $\mathrm{WM}_2$ convergence.

A change of variables in the two defining integrals shows that $(Y_1^{\mathrm{rs}},Y_2^{*,\mathrm{rs}})$ is jointly self-similar with indices $h$ and $2h$. Moreover, $Y_2^{*,\mathrm{rs}}$ is continuous and adapted to the natural filtration of $Y_1$. Since the first coordinate of the pair is $Y_1$ itself, the augmented natural filtrations generated by $Y_1$ and by the pair coincide, so quasi-left-continuity is preserved. Since $Z_1(0)=0$, both prelimit scaling processes are nonnegative and continuous at zero. With $e(t)=Z_1(t)-\mu t$ and $M(t)=\sup_{s\le t}|e(s)|$, the strong law and c\`adl\`ag local boundedness give $M(t)=o(t)$ almost surely. Hence
\begin{align*}
    \frac{Z_2^{\mathrm{rs}}(t)}{t^2}
    \le4\frac{M(t)^2}{t^2}\to0
    \qquad\text{almost surely},
\end{align*}
and
\begin{align*}
    \frac{Z_2^{*,\mathrm{rs}}(t)}{t^2}
    \le\frac{4}{t^3}\int_0^tM(u)^2\,\mathrm du\to0
    \qquad\text{almost surely}.
\end{align*}
Indeed, for any $\delta>0$, the tail of the integral is bounded by
$\frac{\delta^2t^3}{3}$ for all sufficiently large $t$, whereas its integral over
any fixed initial interval is $o(t^3)$.
This completes the verification of Assumption~\ref{asmp: fclt}.

\subsubsection{Proof of Theorem~\ref{thm: stop_prop_rep}}
Under the substitution
$(Z_1,Z_2,Y_1,Y_2)=(Z_1^\iota,Z_2^{*,\iota},Y_1^\iota,
Y_2^{*,\iota})$, the corresponding construction theorem verifies
Assumption~\ref{asmp: fclt}. The result follows from
Theorem~\ref{thm: stop_prop}.

\subsubsection{Proof of Theorem~\ref{thm: asymp_valid_rep}}
The convergence in~\eqref{eq: limit_dist_rep}, invariance in $c$, and the
asserted oracle coverage follow from Theorem~\ref{thm: asymp_valid} under
the same substitution.

\subsection{Proofs for Section~\ref{sec: eg}}
\subsubsection{Proof of Theorem~\ref{thm: time}}
\paragraph{Verification of Assumption~\ref{asmp: fclt}.}
Let $b$ be a regularly varying norming function of index $\frac{1}{\alpha}$
associated with $\xi_0$, and take $\ell(t)=\frac{t^{\frac{1}{\alpha}}}{b(t)}$. The
theorem is applied to the centered innovations $\xi_t-\mu_\xi$. The
geometric coefficients are nonnegative, summable to every positive power,
and have partial sums between zero and their total sum, so they satisfy the
finite- and infinite-order coefficient conditions of Theorem~2 in
\cite{krizmanic2023functional}. That theorem, together with the
sample-centering argument in the proof of
Theorem~\ref{thm: iid_infinite}, gives
\eqref{eq: fclt} for the pair $(Z_1,Z_2)$. Here $h=\frac{1}{\alpha}$, and the
limiting pair $(Y_1,Y_2)$ is a jointly self-similar L\'evy process whose
first coordinate is an $\alpha$-stable L\'evy process and whose second
coordinate is an $\frac{\alpha}{2}$-stable subordinator. The first coordinate is
quasi-left-continuous with respect to the augmented joint filtration, and
the second coordinate is almost surely positive at every positive time and
has no negative jumps. Thus Assumption~\ref{asmp: fclt}(a) and (c)--(e)
hold. Birkhoff's ergodic theorem gives
$\frac{Z_1(t)-\mu t}{t}\to0$ almost surely, while
Lemma~\ref{lem:subquadratic_squares} and the least-squares inequality give
$\frac{Z_2(t)}{t^2}\to0$ almost surely. This proves
Assumption~\ref{asmp: fclt}(b).

\paragraph{Verification of Assumption~\ref{asmp: early_stopping}.}
Let
\begin{align*}
    \Delta_n:=\max_{1\le i\le n}X_i-\min_{1\le i\le n}X_i.
\end{align*}
As in~\eqref{eq: iid_range_bound}, $Z_2(n)\ge\frac{\Delta_n^2}{2}$. We notice
that, for every $t\ge2$,
\begin{align*}
    X_t-X_{t-1}
    =\xi_t-(1-\varphi)X_{t-1}
    -\varphi^{p+1}\xi_{t-p-1},
\end{align*}
where the last term is interpreted as zero when $p=\infty$. Successive
conditioning on $\sigma(\xi_s:s\le t-1)$ therefore gives, for every
$x>0$,
\begin{align*}
    \PB(\Delta_n\le x)
    &\le\PB\left(\max_{2\le t\le n}|X_t-X_{t-1}|\le x\right)\\
    &\le\left(\sup_{y\in\RB}\PB(|\xi_0-y|\le x)\right)^{n-1}\\
    &\le\exp\left(-\frac{n-1}{2}
    \PB(|\xi_1-\xi_2|>2x)\right).
\end{align*}
The last inequality follows from
\begin{align*}
    \left(\sup_{y\in\RB}\PB(|\xi_0-y|\le x)\right)^2
    \le\PB(|\xi_1-\xi_2|\le2x),
\end{align*}
together with $(1-u)^a\le\exp(-au)$.
Fix $0<\nu<\frac{\alpha}{\alpha-1}$, $K>0$, and $\eta>0$. Following the
dyadic argument in the proof of Theorem~\ref{thm: iid_infinite}, let
$m_\varepsilon=\lceil\varepsilon^{-\nu}\rceil$,
$r_j=2^jm_\varepsilon$, and let $J_\varepsilon$ be the largest integer
such that $r_{J_\varepsilon}\le\eta a_\varepsilon$. The preceding bounds
give
\begin{align*}
\PB\left(\inf_{\varepsilon^{-\nu}\le n\le\eta a_\varepsilon}
    \frac{\sqrt{Z_2(n)}}{n}\le K\varepsilon\right)
&\le\sum_{j=0}^{J_\varepsilon}
    \exp\left(-\frac{r_j-1}{2}
    \PB\left(|\xi_1-\xi_2|>4\sqrt{2}K\varepsilon r_j\right)\right).
\end{align*}
Since $A(t)=\frac{t}{b(t)}$ and the tail of $|\xi_1-\xi_2|$ is regularly varying
with index $-\alpha$, the remainder of that argument gives, for every
$0<\delta<\alpha-1$, all sufficiently small $\eta$, and some $C>0$
independent of $\eta$,
\begin{align*}
\limsup_{\varepsilon\downarrow0}
    \PB\left(\inf_{\varepsilon^{-\nu}\le n\le\eta a_\varepsilon}
    \frac{\sqrt{Z_2(n)}}{n}\le K\varepsilon\right)
&\le\sum_{k=0}^{\infty}
    \exp\left(-C\eta^{-(\alpha-1-\delta)}
    2^{k(\alpha-1-\delta)}\right).
\end{align*}
The series on the right converges to zero as $\eta\downarrow0$, proving
Assumption~\ref{asmp: early_stopping}.

\subsubsection{Proof of Theorem~\ref{thm: pa_limit}}
Let $b$ be a regularly varying norming function of index
    $\frac{1}{\alpha}$ such that $t\PB(|\xi_1|>b(t))\to1$, and choose the slowly
varying function in the statement as
    $\ell(t)=\frac{t^{\frac{1}{\alpha}}}{b(t)}$, so
    $b(t)^{-1}=t^{-\frac{1}{\alpha}}\ell(t)$.
For $r\ge1$, iteration of \eqref{eq: sgd_linear} and summation over the
iterates give
\begin{align}
    \sum_{i=1}^r(\theta_i-\theta^*)
    =&(\theta_0-\theta^*)\sum_{i=1}^r
    \prod_{l=0}^{i-1}(1-H\eta_l)
    -\frac{1}{H}\sum_{k=1}^r q_{k-1,r}\xi_k,
    \label{eq: pa_exact_decomposition}
\end{align}
where $q_{k,j}$ is defined in Lemma~\ref{lem: q_est}. Since $\varrho<1$,
\begin{align*}
    \sum_{i=1}^{\infty}\prod_{l=0}^{i-1}(1-H\eta_l)
    \le\sum_{i=1}^{\infty}
    \exp\left(-H\sum_{l=0}^{i-1}\eta_l\right)<\infty.
\end{align*}
Thus the first term on the right-hand side of
\eqref{eq: pa_exact_decomposition} is negligible on the $b(n)$ scale.

\paragraph{The first coordinate.}
First, define, for $0\le t\le1$,
\begin{align}
    X_n(t)&=\frac{1}{Hb(n)}\sum_{k=1}^{\lfloor nt\rfloor}
    q_{k-1,\lfloor nt\rfloor}\xi_k,
    &Z_n(t)&=\frac{1}{Hb(n)}
    \sum_{k=1}^{\lfloor nt\rfloor}\xi_k.
    \label{eq: pa_weighted_processes}
\end{align}
For $\Delta>0$, let
\begin{align*}
    \mathcal J_n(\Delta)
    :=\{1\le j\le n:|\xi_j|>\Delta b(n)\}.
\end{align*}
Choose $C_w>0$ sufficiently large, set
$w_n=\lfloor C_wn^\varrho\log n\rfloor$, and define
\begin{align*}
    \bar q_{k,j}:=
    \begin{cases}
        q_{k,j},&j\le k+w_n,\\
        1,&j>k+w_n.
    \end{cases}
\end{align*}
We decompose $X_n$ in the following:
\begin{align}
    X_n(t)
    &=X_n^>(t)+X_n^<(t)\notag\\
    &=\underbrace{\bar X_n^>(t)+Z_n^<(t)}_{Y_n(t)}
    +\underbrace{X_n^>(t)-\bar X_n^>(t)}_{\operatorname{Err}_{n,1}(t)}
    +\underbrace{X_n^<(t)-Z_n^<(t)}_{\operatorname{Err}_{n,2}(t)},
    \label{eq: pa_first_decomposition}
\end{align}
where
\begin{align*}
    X_n^>(t)
    &=\frac{1}{Hb(n)}\sum_{k=1}^{\lfloor nt\rfloor}
    q_{k-1,\lfloor nt\rfloor}\xi_k
    \mathbbm{1}\{|\xi_k|>\Delta b(n)\},\\
    X_n^<(t)
    &=\frac{1}{Hb(n)}\sum_{k=1}^{\lfloor nt\rfloor}
    q_{k-1,\lfloor nt\rfloor}\xi_k
    \mathbbm{1}\{|\xi_k|\le\Delta b(n)\},\\
    \bar X_n^>(t)
    &=\frac{1}{Hb(n)}\sum_{k=1}^{\lfloor nt\rfloor}
    \bar q_{k-1,\lfloor nt\rfloor}\xi_k
    \mathbbm{1}\{|\xi_k|>\Delta b(n)\},\\
    Z_n^>(t)
    &=\frac{1}{Hb(n)}\sum_{k=1}^{\lfloor nt\rfloor}
    \xi_k\mathbbm{1}\{|\xi_k|>\Delta b(n)\},\\
    Z_n^<(t)
    &=\frac{1}{Hb(n)}\sum_{k=1}^{\lfloor nt\rfloor}
    \xi_k\mathbbm{1}\{|\xi_k|\le\Delta b(n)\}.
\end{align*}

\paragraph*{$\operatorname{Err}_{n,1}(t)$ is negligible.}
Put $r=\lfloor nt\rfloor$. By the definition of $\bar q_{k-1,r}$,
\begin{align*}
    \operatorname{Err}_{n,1}(t)
    =\frac{1}{Hb(n)}\sum_{\substack{k\le r\\r-k\ge w_n}}
    (q_{k-1,r}-1)\xi_k
    \mathbbm{1}\{|\xi_k|>\Delta b(n)\}.
\end{align*}
For $r-k\ge w_n$, Lemma~\ref{lem: q_est} gives
\begin{align*}
    |q_{k-1,r}-1|
    \le |q_{k-1,r}-1+P_{k-1,r}|+P_{k-1,r}
    \lesssim k^{\varrho-1}+P_{k-1,r}.
\end{align*}
Moreover, since $\eta_l\ge\eta_n$ for $l\le r\le n$,
\begin{align*}
    P_{k-1,r}
    &\le\exp\left(-H\sum_{l=k}^r\eta_l\right)
    \le\exp(-Hw_n\eta_n)
    \le n^{-d}
\end{align*}
for some $d>0$ and all sufficiently large $n$. Hence
\begin{align*}
    |\operatorname{Err}_{n,1}(t)|
    \lesssim\frac{1}{b(n)}\sum_{k=1}^n
    (k^{\varrho-1}+n^{-d})|\xi_k|
    \mathbbm{1}\{|\xi_k|>\Delta b(n)\}.
\end{align*}
Karamata's theorem gives
\begin{align}
    \EB[|\xi_1|\mathbbm{1}\{|\xi_1|>\Delta b(n)\}]
    \lesssim\frac{b(n)}{n}\Delta^{1-\alpha}.
    \label{eq: pa_truncated_first_moment}
\end{align}
Taking the supremum and expectations for
$\operatorname{Err}_{n,1}(t)$ and using
\eqref{eq: pa_truncated_first_moment}, we obtain
\begin{align}
    \EB\left[\sup_{0\le t\le1}
    |\operatorname{Err}_{n,1}(t)|\right]
    &\lesssim\frac{1}{b(n)}\sum_{k=1}^n
    (k^{\varrho-1}+n^{-d})
    \EB[|\xi_1|\mathbbm{1}\{|\xi_1|>\Delta b(n)\}]\notag\\
    &\lesssim\Delta^{1-\alpha}
    \left(\frac{1}{n}\sum_{k=1}^nk^{\varrho-1}+n^{-d}\right)\notag\\
    &\lesssim\Delta^{1-\alpha}
    (n^{\varrho-1}+n^{-d})\to0.
    \label{eq: err1}
\end{align}

\paragraph*{$\operatorname{Err}_{n,2}(t)$ is negligible.}
Again put $r=\lfloor nt\rfloor$. We have
\begin{align}
    \operatorname{Err}_{n,2}(t)
    =&\frac{1}{Hb(n)}\sum_{k=1}^r
    (q_{k-1,r}-1+P_{k-1,r})\xi_k
    \mathbbm{1}\{|\xi_k|\le\Delta b(n)\}\notag\\
    &-\frac{1}{Hb(n)}\sum_{k=1}^rP_{k-1,r}\xi_k
    \mathbbm{1}\{|\xi_k|\le\Delta b(n)\}.
    \label{eq: err_2}
\end{align}
Let
\begin{align*}
    \tilde{\xi}_k
    :=\xi_k\mathbbm{1}\{|\xi_k|\le\Delta b(n)\}
    -\EB[\xi_k\mathbbm{1}\{|\xi_k|\le\Delta b(n)\}].
\end{align*}
For the first sum in~\eqref{eq: err_2}, centering and
$\EB\xi_1=0$ give
\begin{gather*}
\frac{1}{Hb(n)}\sum_{k=1}^r
    (q_{k-1,r}-1+P_{k-1,r})\xi_k
    \mathbbm{1}\{|\xi_k|\le\Delta b(n)\}\\
=\frac{1}{Hb(n)}\sum_{k=1}^r(q_{k-1,r}-1+P_{k-1,r})\tilde{\xi}_k-\frac{1}{Hb(n)}\sum_{k=1}^r(q_{k-1,r}-1+P_{k-1,r})\EB[\xi_1\mathbbm{1}\{|\xi_1|>\Delta b(n)\}].
\end{gather*}
For the bias term, Lemma~\ref{lem: q_est} and
\eqref{eq: pa_truncated_first_moment} yield
\begin{align}
&\max_{r\le n}\left|\frac{1}{Hb(n)}\sum_{k=1}^r
    (q_{k-1,r}-1+P_{k-1,r})
    \EB[\xi_1\mathbbm{1}\{|\xi_1|>\Delta b(n)\}]
    \right|\notag\\
\le&\frac{C}{b(n)}\sum_{k=1}^nk^{\varrho-1}
    \EB[|\xi_1|\mathbbm{1}\{|\xi_1|>\Delta b(n)\}]\notag\\
\lesssim&\frac{\Delta^{1-\alpha}}{n}
    \sum_{k=1}^nk^{\varrho-1}
    \lesssim\Delta^{1-\alpha}n^{\varrho-1}\to0.
    \label{eq: small_0}
\end{align}
For the mean-zero term, the identity in
Lemma~\ref{lem: q_est} gives
\begin{align*}
\frac{1}{Hb(n)}\sum_{k=1}^r(q_{k-1,r}-1+P_{k-1,r})\tilde{\xi}_k
=&\frac{1}{b(n)}\sum_{i=1}^r\sum_{k=1}^i(\eta_{k-1}-\eta_i)P_{k-1,i-1}\tilde{\xi}_k.
\end{align*}
Moreover, Karamata's theorem gives
\begin{align}
    \EB[\tilde{\xi}_1^2]
    \le\EB[\xi_1^2\mathbbm{1}\{|\xi_1|\le\Delta b(n)\}]
    \lesssim\frac{b(n)^2}{n}\Delta^{2-\alpha}.
    \label{eq: pa_truncated_second_moment}
\end{align}
Thus, Markov's inequality, Cauchy--Schwarz, and
Lemma~\ref{lem: q_square} give, for every $\varepsilon>0$,
\begin{align}
&\PB\left(\max_{r\le n}\left|
    \frac{1}{Hb(n)}\sum_{k=1}^r
    (q_{k-1,r}-1+P_{k-1,r})\tilde{\xi}_k
    \right|>\varepsilon\right)\notag\\
\le&\frac{1}{\varepsilon b(n)}\sum_{i=1}^n
    \EB\left|\sum_{k=1}^i
    (\eta_{k-1}-\eta_i)P_{k-1,i-1}\tilde{\xi}_k
    \right|\notag\\
\le&\frac{1}{\varepsilon b(n)}\sum_{i=1}^n
    \left(\EB[\tilde{\xi}_1^2]
    \sum_{k=1}^i(\eta_{k-1}-\eta_i)^2
    P_{k-1,i-1}^2\right)^{\frac{1}{2}}\notag\\
\lesssim&\frac{\Delta^{1-\frac{\alpha}{2}}}{\sqrt n}
    \sum_{i=1}^ni^{\frac{\varrho-2}{2}}\notag\\
\lesssim&\Delta^{1-\frac{\alpha}{2}}n^{\frac{\varrho-1}{2}}\to0.
    \label{eq: small_1}
\end{align}
Here independence and centering eliminate the cross terms in the
second-moment calculation. For the second sum in~\eqref{eq: err_2}, we
decompose it into a mean-zero term and a bias term:
\begin{gather*}
\frac{1}{Hb(n)}\sum_{k=1}^rP_{k-1,r}\xi_k
    \mathbbm{1}\{|\xi_k|\le\Delta b(n)\}\\
=\frac{1}{Hb(n)}\sum_{k=1}^rP_{k-1,r}\tilde{\xi}_k
    -\frac{1}{Hb(n)}
    \sum_{k=1}^rP_{k-1,r}\EB[\xi_1\mathbbm{1}\{|\xi_1|>\Delta b(n)\}].
\end{gather*}
For the bias term, \eqref{eq: pa_truncated_first_moment} gives
\begin{align}
    \frac{|\EB[\xi_1\mathbbm{1}\{|\xi_1|>\Delta b(n)\}]|}{Hb(n)}
    \max_{r\le n}\sum_{k=1}^rP_{k-1,r}
    &\lesssim\frac{\max_{r\le n}\sum_{k=1}^rP_{k-1,r}}{nH}\\
    &\le\frac{1}{nH^2}\max_{r\le n}\eta_r^{-1}
    \left(H\eta_0P_{0,r}+\sum_{k=2}^r(P_{k-1,r}-P_{k-2,r})\right)\\
    &=\frac{1}{nH^2}\max_{r\le n}\eta_r^{-1}
    \left(P_{r-1,r}-(1-H\eta_0)P_{0,r}\right)\\
    &\le\frac{\max_{r\le n}\eta_r^{-1}}{nH^2}
    \lesssim n^{\varrho-1}\to0.
    \label{eq: small_2}
\end{align}
For the mean-zero term, Abel's lemma gives
\begin{align}
    \left|\frac{1}{Hb(n)}\sum_{k=1}^rP_{k-1,r}\tilde{\xi}_k\right|
    =&\frac{1}{Hb(n)}\left|P_{r-1,r}\sum_{j=1}^r\tilde{\xi}_j
    +\sum_{k=1}^{r-1}(P_{k-1,r}-P_{k,r})
    \sum_{j=1}^k\tilde{\xi}_j\right|\\
    \le&\frac{2}{Hb(n)}\max_{k\le r}
    \left|\sum_{j=1}^k\tilde{\xi}_j\right|.
\end{align}
Doob's maximal inequality and
\eqref{eq: pa_truncated_second_moment} therefore give
\begin{align}
\PB\left(\max_{r\le n}\left|\frac{1}{Hb(n)}
    \sum_{k=1}^rP_{k-1,r}\tilde{\xi}_k\right|>\varepsilon\right)
\le&\PB\left(\max_{k\le n}\left|
    \sum_{j=1}^k\tilde{\xi}_j\right|>
    \frac{\varepsilon Hb(n)}{2}\right)\notag\\
\le&\frac{4n\EB[\tilde{\xi}_1^2]}
    {\varepsilon^2H^2b(n)^2}\notag\\
    \lesssim&\Delta^{2-\alpha}.
    \label{eq: small_3}
\end{align}
Combining~\eqref{eq: small_0}, \eqref{eq: small_1},
\eqref{eq: small_2}, and~\eqref{eq: small_3}, we obtain
\begin{align}
    \limsup_{n\to\infty}\PB\left(
    \sup_{0\le t\le1}|\operatorname{Err}_{n,2}(t)|>\varepsilon
    \right)\lesssim\Delta^{2-\alpha}.
    \label{eq: err2}
\end{align}

\paragraph*{$M_1$ convergence between $Y_n$ and $Z_n$.}
For $K>0$, we consider the following good event $G_n(\Delta)$:
\begin{align}
    G_n(\Delta)=\left\{|\mathcal J_n(\Delta)|\le K,
    \min_{i\neq j\in\mathcal J_n(\Delta)}|i-j|>2w_n
    ,\ j\le n-w_n\text{ for every }j\in\mathcal J_n(\Delta)\right\}.
\end{align}
We use the conventions $\min\varnothing=+\infty$ and
$\max\varnothing=0$.
Under the good event $G_n(\Delta)$, there exists disjoint intervals:
\begin{align}
    \mathcal I_n:=\bigcup_{j\in\mathcal J_n(\Delta)}
    \left[\frac{j}{n},\frac{j+w_n}{n}\right]
    :=\bigcup_{j\in\mathcal J_n(\Delta)}I_{j,n}.
\end{align}
For $t\notin\mathcal I_n$, we have:
\begin{align}
    Y_n(t)=Z_n(t).
\end{align}
For $t\in I_{j,n}$ and $j\in\mathcal J_n(\Delta)$, we have:
\begin{align}
    &Y_n(t)=Z_n^<(t)+\frac{1}{Hb(n)}
    \sum_{\substack{k<j,\,k\in\mathcal J_n(\Delta)}}\xi_k
    +\frac{1}{Hb(n)}\bar q_{j-1,\lfloor nt\rfloor}\xi_j,\\
    &Z_n(t)=Z_n^<(t)+\frac{1}{Hb(n)}
    \sum_{\substack{k<j,\,k\in\mathcal J_n(\Delta)}}\xi_k
    +\frac{1}{Hb(n)}\xi_j.
\end{align}
Since the intervals $I_{j,n}$, $j\in\mathcal J_n(\Delta)$, are
disjoint, for $t\in I_{j,n}$ define
\begin{align}
    B_n(t):=Z_n^<(t)+\frac{1}{Hb(n)}
    \sum_{\substack{k<j,\,k\in\mathcal J_n(\Delta)}}\xi_k.
\end{align}
\begin{figure}[H]
    \centering
    \includegraphics[width=0.7\linewidth]{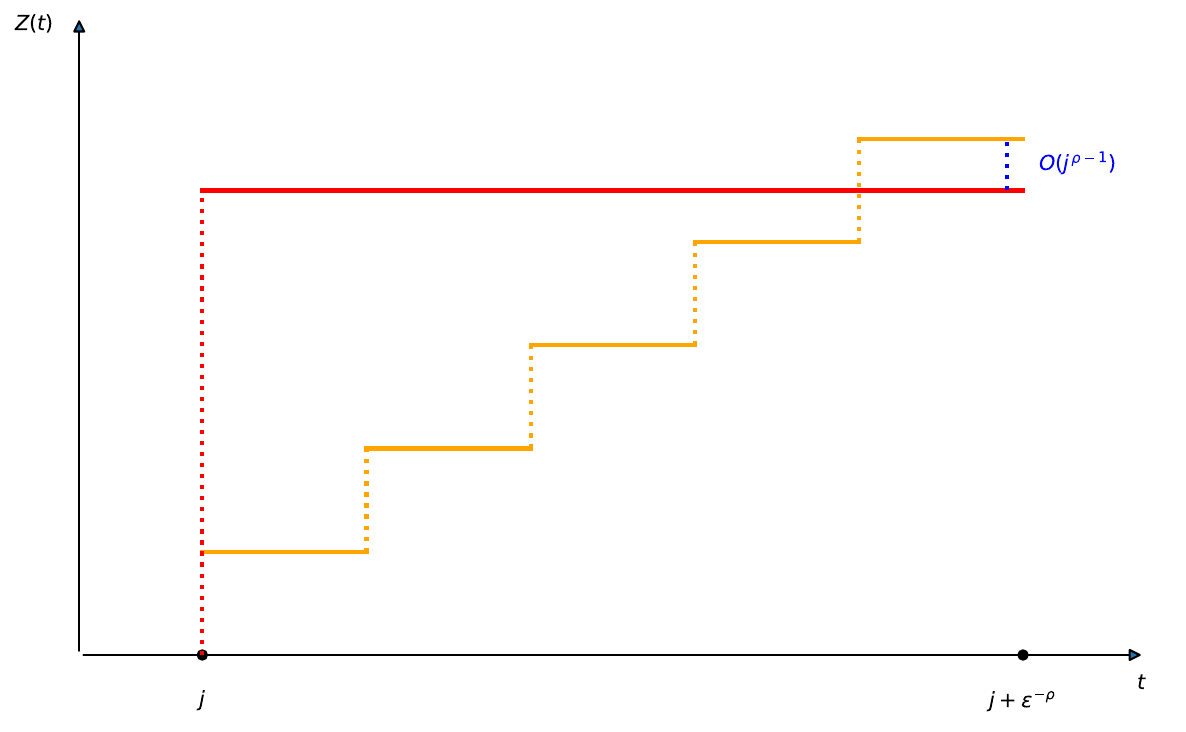}
    \caption{On the large-jump event $
\left\{
|g(\theta_j;\xi_{j+1})|
>
	\frac{n^{\frac{1}{\alpha}}}{\ell(n)}
\right\}$, the realized contribution of this jump to \(Z_1(\cdot)\) is represented by the orange step function.
The corresponding averaged linear response is given by $-H^{-1}g(\theta_j;\xi_{j+1})$,
which is represented by the red step function.
Over a local time window of order \(n^{\varrho}\log n\), these two step functions coincide up to a controllable error of order \(O(j^{\varrho-1})\).}
    \label{fig: m1sgd}
\end{figure}
We construct ordered parametric representations
$(\lambda_Y,\rho_Y)$ and $(\lambda_Z,\rho_Z)$ of the completed graphs of
$Y_n$ and $Z_n$, respectively, restricted to $I_{j,n}$ and indexed by
$[0,1]$. After an affine rescaling, these local parameter
intervals can be identified with intervals $U_j=[u_j,u_{j+1}]$ having
disjoint interiors for $j\in\mathcal J_n(\Delta)$. For
$u\notin\bigcup_{j\in\mathcal J_n(\Delta)}U_j$, the parametrizations
for $Y_n$ and $Z_n$ should be exactly the same as the complete graphs
for these two processes are the same. We consider two phases
decomposition for some $u^*\in(0,1)$:
\begin{align}
    [0,1]=[0,u^*]\cup[u^*,1].
\end{align}
In $u\in[0,u^*]$, we set that parametrization $(\lambda_Y,\rho_Y)$
will go over the entire complete graph of $Y_n(\cdot)$ in $I_{j,n}$,
while $(\lambda_Z,\rho_Z)$ will stay at the vertical segment at
$\frac{j}{n}$. Thus, we have:
\begin{align}
    &\sup_{u\in[0,u^*]}|\lambda_Y(u)-\lambda_Z(u)|
    \le\frac{w_n}{n},\\
    &\sup_{u\in[0,u^*]}|\rho_Y(u)-\rho_Z(u)|
    \le\max_{j\le r\le j+w_n}\left|B_n\left(\frac{r}{n}\right)
    -B_n\left(\frac{j}{n}\right)\right|
    +\sup_{u\in[0,u^*]}|\rho_{\tilde Y}(u)
    -\rho_{\tilde Z}(u)|.
\end{align}
Here $\rho_{\tilde Y}(u)$ and $\rho_{\tilde Z}(u)$ are the
parametrizations for processes $(t\in I_{j,n})$
\begin{align}
    &\tilde Y_n(t)=
    \begin{cases}
        0, &t=\frac{j}{n}-,\\
        \frac{1}{Hb(n)}\bar q_{j-1,\lfloor nt\rfloor}\xi_j,
        &t\in I_{j,n},
    \end{cases}\\
    &\tilde Z_n(t)=
    \begin{cases}
        0, &t=\frac{j}{n}-,\\
        \frac{1}{Hb(n)}\xi_j,&t\in I_{j,n},
    \end{cases}
\end{align}
and $\lambda_{\tilde Y}=\lambda_Y$,
$\lambda_{\tilde Z}=\lambda_Z$. We notice that $q_{j-1,k}$ is
non-decreasing with $k$ and satisfies
$q_{j-1,k}\le1+Cj^{\varrho-1}$.
Let $\Pi_{n,j}$ be the projection onto the closed line segment
with endpoints $0$ and $\frac{\xi_j}{Hb(n)}$, and choose
$\rho_{\tilde Z}(u)=\Pi_{n,j}(\rho_{\tilde Y}(u))$ during this
phase.
Thus, we can control the error by:
\begin{align}
    \sup_{u\in[0,u^*]}|\rho_{\tilde Y}(u)
    -\rho_{\tilde Z}(u)|
    \le\frac{|\xi_j|}{Hb(n)}\max_{j\le r\le j+w_n}
    (\bar q_{j-1,r}-1)_+
    \le\frac{Cj^{\varrho-1}|\xi_j|}{Hb(n)},
\end{align}
where the last inequality is due to Lemma~\ref{lem: q_est}. Then, we
have the error:
\begin{align}
    &\sup_{u\in[0,u^*]}|\lambda_Y(u)-\lambda_Z(u)|
    \le\frac{w_n}{n},\\
    &\sup_{u\in[0,u^*]}|\rho_Y(u)-\rho_Z(u)|
    \le\max_{j\le r\le j+w_n}\left|B_n\left(\frac{r}{n}\right)
    -B_n\left(\frac{j}{n}\right)\right|
    +\frac{Cj^{\varrho-1}|\xi_j|}{Hb(n)}.
\end{align}
In $u\in[u^*,1]$, we froze the process $Y_n$ as the entire complete
graph has been visited in the last stage. And we go over the entire
complete graph of $Z_n$ from point
$\left(\frac{j}{n},B_n\left(\frac{j}{n}\right)+\frac{\xi_j}{Hb(n)}\right)$:
\begin{align}
    &\lambda_Y(u)=\frac{j+w_n}{n},\\
    &\rho_Y(u)=B_n\left(\frac{j+w_n}{n}\right)
    +\frac{1}{Hb(n)}\xi_j,
\end{align}
which is due to the final point in the complete graph of $Y_n$ is
$Y_n\left(\frac{j+w_n}{n}\right)$ and $\bar q_{j-1,j+w_n}=1$. Thus, we can control the
error by:
\begin{align}
    \sup_{u\in[u^*,1]}|\lambda_Y(u)-\lambda_Z(u)|
    &\le\frac{w_n}{n},\\
    \sup_{u\in[u^*,1]}|\rho_Y(u)-\rho_Z(u)|
    &\le\max_{j\le r\le j+w_n}
    \left|B_n\left(\frac{r}{n}\right)
    -B_n\left(\frac{j+w_n}{n}\right)\right|\\
    &\le2\max_{j\le r\le j+w_n}
    \left|B_n\left(\frac{r}{n}\right)-B_n\left(\frac{j}{n}\right)\right|.
\end{align}
The last inequality follows from the triangle inequality.
For each $j\in\mathcal J_n(\Delta)$, the local graph pieces have the same initial and terminal values:
\begin{align*}
    Y_n\left(\frac{j}{n}-\right)
    &=Z_n\left(\frac{j}{n}-\right),
    &Y_n\left(\frac{j+w_n}{n}\right)
    &=Z_n\left(\frac{j+w_n}{n}\right).
\end{align*}
Moreover, monotonicity of $q_{j-1,r}$ and the definition of $\Pi_{n,j}$ ensure that the projected spatial component traverses the vertical segment of the completed graph of $Z_n$ in graph order, possibly with a constant portion. Thus, the preceding constructions are ordered continuous parametric representations of the two local graph pieces.
On $G_n(\Delta)$, order the intervals $I_{j,n}$ by increasing $j$ and concatenate the preceding local representations in that order. On the complement of their union, use the same ordered representation of the common completed graph of $Y_n$ and $Z_n$, inserting flat spots in either representation when necessary. The endpoint choices in the two phases above agree with the adjacent common pieces. The concatenated representations are therefore continuous, have nondecreasing time components, preserve the completed-graph order, and are onto the completed graphs of $Y_n$ and $Z_n$. Thus, they form a valid pair of global parametric representations.
Combining over $j\in\mathcal J_n(\Delta)$, we have:
\begin{align}
    d_{M_1,[0,1]}(Y_n,Z_n)&\le\frac{w_n}{n}
    +2\max_{\substack{j\in\mathcal J_n(\Delta)\\j\le r\le j+w_n}}
    \left|B_n\left(\frac{r}{n}\right)-B_n\left(\frac{j}{n}\right)\right|
    +\max_{j\in\mathcal J_n(\Delta)}
    \frac{Cj^{\varrho-1}|\xi_j|}{Hb(n)}.
    \label{eq: bias}
\end{align}
For the second term in Eqn~\eqref{eq: bias}, for any
$\varepsilon>0$, under the good event $G_n(\Delta)$, we have:
\begin{align}
    &\PB\left(
    \max_{\substack{j\in\mathcal J_n(\Delta)\\j\le r\le j+w_n}}
    \left|B_n\left(\frac{r}{n}\right)-B_n\left(\frac{j}{n}\right)\right|
    >\varepsilon;G_n(\Delta)\right)\\
    \le&\ \sum_{j=1}^{n-w_n}
    \PB\left(|\xi_j|>\Delta b(n),
    \frac{1}{Hb(n)}\max_{r\le w_n}
    \left|\sum_{k=j+1}^{j+r}\xi_k
    \mathbbm{1}\{|\xi_k|\le\Delta b(n)\}\right|>\varepsilon\right)\\
    =&\ (n-w_n)\PB(|\xi_1|>\Delta b(n))
    \PB\left(\frac{1}{Hb(n)}\max_{r\le w_n}
    \left|\sum_{k=1}^r\xi_k
    \mathbbm{1}\{|\xi_k|\le\Delta b(n)\}\right|>\varepsilon\right)\\
    \lesssim&\ 
    \PB\left(\frac{1}{Hb(n)}\max_{r\le w_n}
    \left|\sum_{k=1}^r\xi_k
    \mathbbm{1}\{|\xi_k|\le\Delta b(n)\}\right|>\varepsilon\right).
\end{align}
We notice that:
\begin{align}
    \frac{\sum_{k=1}^{w_n}\EB[|\xi_1|
    \mathbbm{1}\{|\xi_1|>\Delta b(n)\}]}{b(n)}
    \lesssim\frac{w_n}{n}\to0.
\end{align}
Recall that
$\tilde{\xi}_k=\xi_k\mathbbm{1}\{|\xi_k|\le\Delta b(n)\}
-\EB[\xi_k\mathbbm{1}\{|\xi_k|\le\Delta b(n)\}]$. When $n$ is
large, by Doob's maximal inequality, we have:
\begin{align}
    \PB\left(\frac{1}{Hb(n)}\max_{r\le w_n}
    \left|\sum_{k=1}^r\xi_k
    \mathbbm{1}\{|\xi_k|\le\Delta b(n)\}\right|>\varepsilon\right)
    &\le\PB\left(\frac{1}{Hb(n)}\max_{r\le w_n}
    \left|\sum_{k=1}^r\tilde{\xi}_k\right|>\frac{\varepsilon}{2}\right)\\
    &\le\frac{4w_n\EB\tilde{\xi}_1^2}
    {b(n)^2H^2\varepsilon^2}\\
    &\lesssim\frac{w_n}{n}\to0,
\end{align}
which implies:
\begin{align}
    \label{eq: bias_1}
    \PB\left(
    \max_{\substack{j\in\mathcal J_n(\Delta)\\j\le r\le j+w_n}}
    \left|B_n\left(\frac{r}{n}\right)-B_n\left(\frac{j}{n}\right)\right|
    >\varepsilon;G_n(\Delta)\right)
    \lesssim\frac{w_n}{n}\to0.
\end{align}
For the third term in Eqn~\eqref{eq: bias}, we have
\begin{align}
    \PB\left(\max_{j\in\mathcal J_n(\Delta)}
    \frac{Cj^{\varrho-1}|\xi_j|}{Hb(n)}>\varepsilon;
    G_n(\Delta)\right)&\le\PB\left(\max_{j\le n}
    \frac{Cj^{\varrho-1}|\xi_j|
    \mathbbm{1}\{|\xi_j|>\Delta b(n)\}}{Hb(n)}>\varepsilon\right)\\
    &\le\sum_{j=1}^n\PB\left(
    \frac{Cj^{\varrho-1}|\xi_j|
    \mathbbm{1}\{|\xi_j|>\Delta b(n)\}}{Hb(n)}>\varepsilon\right)\\
    &\le\sum_{j=1}^n\frac{Cj^{\varrho-1}\EB[|\xi_j|
    \mathbbm{1}\{|\xi_j|>\Delta b(n)\}]}
    {\varepsilon Hb(n)}\\
    &\lesssim\frac{\sum_{j=1}^nj^{\varrho-1}}{n}\\
    &\lesssim\frac{1}{n^{1-\varrho}}\to0.
    \label{eq: bias_2}
\end{align}
Thus, for Eqn~\eqref{eq: bias}, we combine Eqn~\eqref{eq: bias_1}
and~\eqref{eq: bias_2}:
\begin{align}
    \limsup_{n\to\infty}\PB\left(
    d_{M_1,[0,1]}(Y_n,Z_n)>\varepsilon;G_n(\Delta)\right)=0.
    \label{eq: y_1}
\end{align}
We also have:
\begin{align}
    \label{eq: bad}
    \PB(G_n(\Delta)^c)
    \le&\PB\left(|\mathcal J_n(\Delta)|>K\right)
    +\PB\left(\exists i\neq j\in\mathcal J_n(\Delta),
    |i-j|\le2w_n\right)\notag\\
    &+\PB\left(\mathcal J_n(\Delta)\cap
    \{n-w_n+1,\ldots,n\}\neq\varnothing\right).
\end{align}
For the first term in Eqn~\eqref{eq: bad}, we have:
\begin{align}
    \PB\left(|\mathcal J_n(\Delta)|>K\right)
    \le\frac{\EB|\mathcal J_n(\Delta)|}{K}
    =\frac{n\PB(|\xi_1|>\Delta b(n))}{K}
    \lesssim\frac{1}{\Delta^\alpha K}.
\end{align}
For the second term in Eqn~\eqref{eq: bad}, we have:
\begin{align}
    \PB\left(\exists i\neq j\in\mathcal J_n(\Delta),
    |i-j|\le2w_n\right)
    &\le\sum_{j=1}^n\sum_{i=1}^{2w_n}
    \PB(j\in\mathcal J_n(\Delta),j+i\in\mathcal J_n(\Delta))\\
    &\le2nw_n\PB(|\xi_1|>\Delta b(n))^2\\
    &\lesssim\frac{w_n}{n}.
\end{align}
For the third term in Eqn~\eqref{eq: bad}, we have:
\begin{align}
    \PB\left(\mathcal J_n(\Delta)\cap
    \{n-w_n+1,\ldots,n\}\neq\varnothing\right)
    \le w_n\PB(|\xi_1|>\Delta b(n))
    \lesssim\frac{w_n}{n}.
\end{align}
Thus, we conclude:
\begin{align}
    \limsup_{n\to\infty}\PB(G_n(\Delta)^c)
    \lesssim\frac{1}{\Delta^\alpha K}.
    \label{eq: y_2}
\end{align}
Combining Eqn~\eqref{eq: y_1} and~\eqref{eq: y_2}, we have:
\begin{align}
    \limsup_{n\to\infty}\PB\left(
    d_{M_1,[0,1]}(Y_n,Z_n)>\varepsilon\right)
    \lesssim\frac{1}{\Delta^\alpha K}.
\end{align}
By the error control for $\operatorname{Err}_{n,1}$ and
$\operatorname{Err}_{n,2}$ in probability by Eqn~\eqref{eq: err1}
and~\eqref{eq: err2}, we have:
\begin{align}
    &\limsup_{n\to\infty}\PB\left(
    d_{M_1,[0,1]}(X_n,Z_n)>3\varepsilon\right)\\
    \le&\limsup_{n\to\infty}\left(
    \PB\left(d_{M_1,[0,1]}(Y_n,Z_n)>\varepsilon\right)
    \!+\!\PB\left(\sup_{t\in[0,1]}
    |\operatorname{Err}_{n,1}(t)|>\varepsilon\right)
    \!+\!\PB\left(\sup_{t\in[0,1]}
    |\operatorname{Err}_{n,2}(t)|>\varepsilon\right)\right)\\
    \lesssim&\Delta^{2-\alpha}+\frac{1}{\Delta^\alpha K}.
\end{align}
As $\Delta$ is arbitrarily small, $K$ is arbitrarily large, and
$\alpha\in(1,2)$, we can choose $K=\Delta^{-2}$ and conclude:
\begin{align}
    d_{M_1,[0,1]}(X_n,Z_n)\topb0.
    \label{eq: pa_weighted_m1}
\end{align}
The same argument applies on every compact interval. Consequently, by~\eqref{eq: pa_exact_decomposition},
\begin{align}
&d_{M_1,[0,T]}\left(
    \frac{Z_1(n\,\cdot)-\theta^*\lfloor n\,\cdot\rfloor}{b(n)},
    -\frac{1}{Hb(n)}\sum_{k=1}^{\lfloor n\,\cdot\rfloor}\xi_k
    \right)\topb0.
    \label{eq: pa_first_coordinate_approximation}
\end{align}

We first prove the almost-sure convergence of the iterates:
\begin{align}
    \theta_n\to\theta^*
    \qquad\text{a.s.}
    \label{eq: pa_as_convergence}
\end{align}
Indeed, choose $\gamma\in(1,\alpha)$ such that $\gamma\varrho>1$. Since $\EB|\xi_1|^\gamma<\infty$ and $\sum_{i\ge1}\eta_{i-1}^\gamma<\infty$, the martingale convergence theorem and the von Bahr--Esseen inequality show that $\sum_{i=1}^{\infty}\eta_{i-1}\xi_i$ converges almost surely, which is also similar with the proof in \cite{krasulina1969stochastic}. For fixed $N<n$, iteration of~\eqref{eq: sgd_linear} gives
\begin{align*}
    \theta_n-\theta^*
    =P_{N-1,n-1}(\theta_N-\theta^*)
    -\sum_{i=N+1}^nP_{i-1,n-1}\eta_{i-1}\xi_i.
\end{align*}
Abel summation gives
\begin{align*}
    \left|\sum_{i=N+1}^nP_{i-1,n-1}\eta_{i-1}\xi_i\right|
    \le2\sup_{k\ge N}\left|\sum_{i=N+1}^k\eta_{i-1}\xi_i\right|.
\end{align*}
For every fixed $N$, $P_{N-1,n-1}\to0$ as $n\to\infty$, while the right-hand side converges to zero almost surely as $N\to\infty$. This proves~\eqref{eq: pa_as_convergence}. For $Z_2(\cdot)$, the functional CLT follows
directly from \citet[Lemma~S1]{blanchet2026statistical} because the
residual in the decomposition there is negligible uniformly on compact
intervals.

\paragraph{Verification of Assumption~\ref{asmp: fclt}.}
All conditions other than~(b) are immediate. For $Z_1(\cdot)$, \eqref{eq: pa_as_convergence} and Ces\`aro's lemma give, for every $h_0\ge1$,
\begin{align*}
    \frac{|Z_1(t)-\theta^*t|}{t^{h_0}}\to0
    \qquad\text{a.s.}
\end{align*}
For almost sure convergence for $\frac{Z_2(t)}{t^2}$, it is an instant corollary from the decomposition in \citet[Lemma~S1]{blanchet2026statistical}.

\paragraph{Verification of Assumption~\ref{asmp: early_stopping}.}
The assumption can be verified directly by the same dyadic
argument as in the i.i.d. infinite-variance case, since,
conditionally on the past, $g(\theta_{i-1};\xi_i)$ is $\xi_i$
plus a deterministic shift.
Let $\mathcal F_i=\sigma(\xi_1,\ldots,\xi_i)$ and define
\begin{align*}
    Q_\xi(x):=\sup_{y\in\RB}\PB(|\xi_1-y|\le x),
    \qquad x>0.
\end{align*}
Since $\theta_{i-1}$ is $\mathcal F_{i-1}$-measurable and $g(\theta_{i-1};\xi_i)=\xi_i+H(\theta_{i-1}-\theta^*)$, successive conditioning gives, for every integer $r\ge1$,
\begin{align*}
    \PB\left(\max_{1\le i\le r}|g(\theta_{i-1};\xi_i)|\le x\right)
    \le Q_\xi(x)^r.
\end{align*}
Furthermore,
\begin{align*}
    Q_\xi(x)^2\le\PB(|\xi_1-\xi_2|\le2x),
\end{align*}
and hence
\begin{align}
    \PB\left(\max_{1\le i\le r}|g(\theta_{i-1};\xi_i)|\le x\right)
    \le\exp\left(-\frac{r}{2}\PB(|\xi_1-\xi_2|>2x)\right).
    \label{eq: pa_conditional_concentration}
\end{align}
The verification can be obtained by the same argument of the proof of Theorem~\ref{thm: iid_infinite}.

\subsubsection{Proof of Theorem~\ref{thm: queue_reward_tail}}
    The proof for this theorem is an extension from the proof of
    Proposition~1 of \cite{denisov2021tail}. For completeness consideration, we detail the proof here. Set $S_0=0$ and $S_n=\sum_{i=1}^nX_i$.
    Before time $\beta_1$, Lindley's recursion gives $W_i=S_i>0$,
    while $W_{\beta_1}=0$. Hence, with
    \begin{align*}
        R_p=\sum_{i=1}^{\beta_1-1}S_i^p,
        \qquad
        M=\max_{0\le i<\beta_1}S_i,
    \end{align*}
    It can also be verified that
    $\sum_{i=1}^{\beta_1}W_i^p=R_p$. Since $A_1$ is
    exponentially distributed and independent of $U_0$, regular
    variation of the service-time tail gives
    \begin{align}
        \PB(X_1>x)\sim\PB(U_0>x),
        \qquad x\to\infty.
        \label{eq: increment_tail}
    \end{align}
    Thus, the right tail of $X_1$ is regularly varying with index
    $-\alpha$.
    We first record the one-big-jump estimates for a negative-drift
    random walk with a regularly varying right tail. Proposition~4 of
    \cite{denisov2021tail} gives
    \begin{align}
        \PB(M>y)&\sim \EB[\beta_1]\PB(X_1>y),
        \label{eq: queue_max_tail}\\
        \PB(\beta_1>t)&\sim
        \EB[\beta_1]\PB(X_1>at).
        \label{eq: queue_length_tail}
    \end{align}
    We next establish the fluid approximation for the reward collected
    after a large jump. For each deterministic starting level $v>0$,
    define, using the same increment sequence $\{X_i\}$,
    \begin{align*}
        S_j^{(v)}=v+S_j,\qquad
        \tau(v)=\inf\{j\ge1:S_j^{(v)}\le0\},\qquad
        H_p(v)=\sum_{j=0}^{\tau(v)-1}\left(S_j^{(v)}\right)^p.
    \end{align*}
    We study this family as the starting level $v\to\infty$; all the
    limits in the following argument are taken in this sense.
    Consider the centered random walk associated with $S_j$,
    \begin{align*}
        C_j:=S_j-\EB[S_j]=S_j+aj
        =\sum_{i=1}^j(X_i-\EB[X_i]).
    \end{align*}
    Since $C_j$ has integrable mean-zero increments, the strong law of
    large numbers gives, for every fixed $T>0$,
    \begin{align*}
        \frac{1}{v}\max_{0\le j\le\lfloor vT\rfloor}|C_j|\toas0.
    \end{align*}
    Using $S_j^{(v)}=v+C_j-aj$, we consequently obtain
    \begin{align}
        \sup_{0\le t\le T}\left|\frac{S_{\lfloor vt\rfloor}^{(v)}}{v}-(1-at)\right|
        &\le
        \frac{1}{v}\max_{0\le j\le\lfloor vT\rfloor}|C_j|
        +\frac{a}{v}\toas0.
        \label{eq: queue_fluid_path}
    \end{align}
    We first use this path convergence to locate the first-passage
    time. Fix $\delta\in\left(0,\frac{1}{a}\right)$, take $T>\frac{1}{a}+\delta$, and let
    \begin{align*}
        E_v(\delta)=\left\{
        \sup_{0\le t\le T}
        \left|
        \frac{S_{\lfloor vt\rfloor}^{(v)}}{v}-(1-at)
        \right|<\frac{a\delta}{2}
        \right\}.
    \end{align*}
    Equation~\eqref{eq: queue_fluid_path} implies that
    $E_v(\delta)$ occurs eventually almost surely. Put
    $j_v=\lceil v\left(\frac{1}{a}+\delta\right)\rceil$. For all sufficiently large $v$,
    $\frac{j_v}{v}\le T$, and on $E_v(\delta)$,
    \begin{align*}
        \frac{S_j^{(v)}}{v}
        &\ge \frac{a\delta}{2}>0,
        &&0\le j\le\lfloor v\left(\frac{1}{a}-\delta\right)\rfloor,\\
        \frac{S_{j_v}^{(v)}}{v}
        &\le-\frac{a\delta}{2}<0.
    \end{align*}
    Therefore,
    \begin{align*}
        \lfloor v\left(\frac{1}{a}-\delta\right)\rfloor
        <\tau(v)\le
        \lceil v\left(\frac{1}{a}+\delta\right)\rceil
    \end{align*}
    eventually almost surely. Since $\delta$ is arbitrary,
    \begin{align}
        \frac{\tau(v)}{v}\toas\frac{1}{a}.
        \label{eq: queue_fluid_hitting_time}
    \end{align}
    We now turn to the reward. Since the integrand below is constant on
    each interval $[\frac{j}{v},\frac{j+1}{v})$,
    \begin{align*}
        \frac{H_p(v)}{v^{p+1}}
        =\int_0^{\frac{\tau(v)}{v}}
        \left(\frac{S_{\lfloor vt\rfloor}^{(v)}}{v}\right)^p
        \mathrm{d}t.
    \end{align*}
    Let $(x)_+=\max\{x,0\}$ and set
    \begin{align*}
        g_v(t)=\left(
        \frac{S_{\lfloor vt\rfloor}^{(v)}}{v}
        \right)_+^p,
        \qquad
        g(t)=(1-at)_+^p.
    \end{align*}
    For $j<\tau(v)$, $S_j^{(v)}>0$, so replacing the integrand by
    $g_v$ does not change the preceding identity. Moreover,
    \eqref{eq: queue_fluid_path} and continuity of
    $x\mapsto(x)_+^p$ give
    \begin{align*}
        \sup_{0\le t\le T}|g_v(t)-g(t)|\toas0.
    \end{align*}
    By \eqref{eq: queue_fluid_hitting_time}, $\frac{\tau(v)}{v}<T$
    eventually almost surely, and hence
    \begin{align*}
        \left|\frac{H_p(v)}{v^{p+1}}-\int_0^{\frac{1}{a}}g(t)\,\mathrm{d}t\right|
        &\le T\sup_{0\le t\le T}|g_v(t)-g(t)|
        +\left|\int_{\frac{\tau(v)}{v}}^{\frac{1}{a}}g(t)\,\mathrm{d}t\right|\toas0.
    \end{align*}
    Thus,
    \begin{align}
        \frac{H_p(v)}{v^{p+1}}
        \xrightarrow[v\to\infty]{\mathrm{a.s.}}
        \int_0^{\frac{1}{a}}(1-at)^p\,\mathrm{d}t
        =\frac{1}{a(p+1)}.
        \label{eq: queue_fluid_reward}
    \end{align}
    Define
    \begin{align*}
        q_x=(a(p+1)x)^{\frac{1}{p+1}}.
    \end{align*}
    The choice of $q_x$ is dictated by
    \eqref{eq: queue_fluid_reward}: an excursion starting from height
    $q_x$ accumulates approximately $x$ units of reward. The remainder
    of the proof therefore has the following two targets. For every
    fixed $\varepsilon\in(0,1)$, we first show
    \begin{align*}
        \PB(R_p>x,\,M>\varepsilon q_x)
        \sim \EB[\beta_1]\PB(X_1>q_x),
    \end{align*}
    and we then show that
    \begin{align*}
        \lim_{\varepsilon\downarrow0}\limsup_{x\to\infty}
        \frac{\PB(R_p>x,\,M\le\varepsilon q_x)}
        {\EB[\beta_1]\PB(X_1>q_x)}=0.
    \end{align*}
    Together, these two statements give
    \eqref{eq: queue_reward_tail}.
    We begin with the first target. Fix $\varepsilon\in(0,1)$, set
    $y=\varepsilon q_x$, and let
    \begin{align*}
        \sigma_y=\inf\{n<\beta_1:S_n>y\}.
    \end{align*}
    On $\{M>y\}$, $\sigma_y$ is the time at which the excursion first
    exceeds $y$. We first keep this time fixed. The key step is to
    replace the reward event $\{R_p>x\}$ by the height event
    $\{S_k>q_x\}$. More precisely, for every fixed $k$, we will prove
    \begin{align}
        \PB(R_p>x,\,\sigma_y=k)&=\PB(S_k>q_x,\,\sigma_y=k)
        +o\bigl(\PB(X_1>q_x)\bigr).
        \label{eq: queue_reward_height_replacement}
    \end{align}
    We will then evaluate the probability on the right as
    \begin{align}
        \PB(S_k>q_x,\,\sigma_y=k)
        \sim
        \PB(\beta_1>k-1)\PB(X_1>q_x).
        \label{eq: queue_fixed_jump_height}
    \end{align}
    Thus, all the bounds below serve only to justify the replacement
    in \eqref{eq: queue_reward_height_replacement}.
    On $\{\sigma_y=k\}$, the reward accumulated before time $k$ is
    \begin{align*}
        R_{p,k}^{-}=\sum_{j=1}^{k-1}S_j^p,
    \end{align*}
    and, on $\{\sigma_y=k\}$,
    \begin{align}
        0\le R_{p,k}^{-}
        \le(k-1)y^p=o(x)
        \qquad\text{for every fixed }k.
        \label{eq: queue_pre_jump_reward}
    \end{align}
    To describe the reward after time $k$, for $v>0$ define
    \begin{align*}
        S_j^{(k,v)}
        &=v+\sum_{\ell=1}^jX_{k+\ell},
        \qquad S_0^{(k,v)}=v,\\
        \tau_k(v)
        &=\inf\{j\ge1:S_j^{(k,v)}\le0\},\\
        H_p^{(k)}(v)
        &=\sum_{j=0}^{\tau_k(v)-1}\left(S_j^{(k,v)}\right)^p.
    \end{align*}
    On $\{\sigma_y=k\}$, the part of the original excursion beginning
    at time $k$ is exactly the walk
    $\{S_j^{(k,S_k)}\}_{j\ge0}$ stopped at $\tau_k(S_k)$. Therefore,
    \begin{align*}
        R_p=R_{p,k}^{-}+H_p^{(k)}(S_k),
    \end{align*}
    on $\{\sigma_y=k\}$. Moreover, if
    $\mathcal F_k=\sigma(X_1,\ldots,X_k)$, then
    $\{X_{k+\ell}\}_{\ell\ge1}$ is independent of $\mathcal F_k$ and
    has the same law as $\{X_\ell\}_{\ell\ge1}$. Hence, for every
    deterministic $v>0$,
    \begin{align*}
        \mathcal L\left(H_p^{(k)}(v)\mid\mathcal F_k\right)
        =\mathcal L\left(H_p(v)\right).
    \end{align*}
    We now bound the two errors that can occur in the replacement
    \eqref{eq: queue_reward_height_replacement}. Fix
    $\eta\in(0,1-\varepsilon)$. A false negative occurs if
    $S_k>(1+\eta)q_x$ but $R_p\le x$, whereas a false positive occurs
    if $S_k\le(1-\eta)q_x$ but $R_p>x$. Since
    $x=\frac{q_x^{p+1}}{a(p+1)}$, equation
    \eqref{eq: queue_fluid_reward} implies
    \begin{gather*}
        \sup_{v\ge(1+\eta)q_x}\PB(H_p(v)\le x)\to0,\\
        \sup_{y\le v\le(1-\eta)q_x}
        \PB\left(H_p(v)>x-(k-1)y^p\right)\to0.
    \end{gather*}
    For the second relation, \eqref{eq: queue_pre_jump_reward} and
    $v\le(1-\eta)q_x$ imply
    \begin{align*}
        \frac{x-(k-1)y^p}{v^{p+1}}
        \ge
        \frac{1+o(1)}
        {a(p+1)(1-\eta)^{p+1}}
        >
        \frac{1}{a(p+1)}.
    \end{align*}
    The first relation follows analogously because, when
    $v\ge(1+\eta)q_x$,
    $\frac{x}{v^{p+1}}\le\frac{1}{a(p+1)(1+\eta)^{p+1}}$.
    Moreover, by \eqref{eq: queue_max_tail},
    \begin{align*}
        \PB(\sigma_y=k)\le\PB(M>y)
        =O(\PB(X_1>y))
        =O(\PB(X_1>q_x)).
    \end{align*}
    Let
    \begin{align*}
        B_{k,x}^{+}&=\{S_k>(1+\eta)q_x\},\qquad
        B_{k,x}^{-}=\{S_k>(1-\eta)q_x\}.
    \end{align*}
    Since $R_{p,k}^{-}\ge0$, on
    $B_{k,x}^{+}\cap\{\sigma_y=k\}$ the event
    $\{R_p\le x\}$ implies
    $\{H_p^{(k)}(S_k)\le x\}$. Conditioning on $\mathcal F_k$ and using
    the conditional-law identity above,
    \begin{align*}
        \PB(B_{k,x}^{+},\,\sigma_y=k,\,R_p\le x)
        &\le\PB(\sigma_y=k)\sup_{v\ge(1+\eta)q_x}\PB(H_p(v)\le x)
        =o(\PB(X_1>q_x)).
    \end{align*}
    Therefore,
    \begin{align*}
        \PB(R_p>x,\,\sigma_y=k)
        \ge\PB(B_{k,x}^{+},\,\sigma_y=k)
        -o(\PB(X_1>q_x)).
    \end{align*}
    For the upper bound, on
    $\{\sigma_y=k\}\cap(B_{k,x}^{-})^c$ we have
    $y<S_k\le(1-\eta)q_x$. If $R_p>x$, then
    \eqref{eq: queue_pre_jump_reward} implies
    \begin{align*}
        H_p^{(k)}(S_k)
        =R_p-R_{p,k}^{-}
        >x-(k-1)y^p.
    \end{align*}
    Conditioning explicitly on $\mathcal F_k$ therefore gives
    \begin{align*}
        \PB(R_p>x,\,\sigma_y=k,\,(B_{k,x}^{-})^c)
        &\le\EB\left[
        \mathbbm{1}_{\{\sigma_y=k\}\cap(B_{k,x}^{-})^c}
        \PB\left(
        H_p^{(k)}(S_k)>x-(k-1)y^p
        \mid\mathcal F_k
        \right)
        \right]\\
        &\le\sup_{y\le v\le(1-\eta)q_x}
        \PB\left(H_p(v)>x-(k-1)y^p\right)
        \PB(\sigma_y=k)\\
        &=o(\PB(X_1>q_x)).
    \end{align*}
    Consequently, for every fixed $k$,
    \begin{align*}
        \PB(B_{k,x}^{+},\,\sigma_y=k)-o(\PB(X_1>q_x))
        \le \PB(R_p>x,\,\sigma_y=k)
        \le \PB(B_{k,x}^{-},\,\sigma_y=k)
        +o(\PB(X_1>q_x)).
    \end{align*}
    We next evaluate the height probabilities in this bound. For fixed
    $k$ and $c>\varepsilon$, let
    $M_{k-1}=\max_{1\le j\le k-1}S_j$, with $M_0=0$. Independence of
    $X_k$ from $\mathcal F_{k-1}$ gives
    \begin{align*}
        \PB(S_k>cq_x,\,\sigma_y=k)
        =\EB\left[
        \mathbbm{1}_{\{\beta_1>k-1,\,M_{k-1}\le y\}}
        \PB(X_1>cq_x-S_{k-1})
        \right].
    \end{align*}
    On the event inside the expectation,
    $0\le S_{k-1}\le y=\varepsilon q_x$, and hence
    \begin{align*}
        (c-\varepsilon)q_x
        \le cq_x-S_{k-1}\le cq_x.
    \end{align*}
    Since the tail probability is nonincreasing,
    \begin{align*}
        1
        \le
        \frac{\PB(X_1>cq_x-S_{k-1})}
        {\PB(X_1>cq_x)}
        \le
        \frac{\PB(X_1>(c-\varepsilon)q_x)}
        {\PB(X_1>cq_x)}.
    \end{align*}
    The right-hand side converges, by regular variation, to the finite
    constant $\left(\frac{c-\varepsilon}{c}\right)^{-\alpha}$. The ratio in the middle
    is therefore bounded uniformly over the event inside the
    expectation for all sufficiently large $x$.
    Moreover, because $k$ is fixed, $M_{k-1}$ is the maximum of
    finitely many almost surely finite random variables. Since
    $y=\varepsilon q_x\to\infty$, it follows that
    $\mathbbm{1}_{\{M_{k-1}\le y\}}\to1$ almost surely. For every fixed
    sample path, the tail ratio above also converges to one. Dominated
    convergence therefore yields
    \begin{align*}
        \PB(S_k>cq_x,\,\sigma_y=k)
        \sim \PB(\beta_1>k-1)\PB(X_1>cq_x).
    \end{align*}
    Taking $c=1$ proves \eqref{eq: queue_fixed_jump_height}. It follows
    from the preceding bounds that
    \begin{align*}
        \PB(R_p>x,\,\sigma_y=k)-\PB(S_k>q_x,\,\sigma_y=k)
        &\le
        \PB((1-\eta)q_x<S_k\le q_x,\,\sigma_y=k)
        +o(\PB(X_1>q_x)),\\
        \PB(S_k>q_x,\,\sigma_y=k)-\PB(R_p>x,\,\sigma_y=k)
        &\le
        \PB(q_x<S_k\le(1+\eta)q_x,\,\sigma_y=k)
        +o(\PB(X_1>q_x)).
    \end{align*}
    Hence, regular variation yields
    \begin{align*}
        \limsup_{x\to\infty}
        \frac{
        \PB(R_p>x,\,\sigma_y=k)
        -\PB(S_k>q_x,\,\sigma_y=k)
        }{\PB(X_1>q_x)}
        &\le \PB(\beta_1>k-1)
        \left((1-\eta)^{-\alpha}-1\right),\\
        \limsup_{x\to\infty}
        \frac{
        \PB(S_k>q_x,\,\sigma_y=k)
        -\PB(R_p>x,\,\sigma_y=k)
        }{\PB(X_1>q_x)}
        &\le \PB(\beta_1>k-1)
        \left(1-(1+\eta)^{-\alpha}\right).
    \end{align*}
    Both right-hand sides tend to zero as $\eta\downarrow0$, which
    proves
    \eqref{eq: queue_reward_height_replacement}. Combining
    \eqref{eq: queue_reward_height_replacement} and
    \eqref{eq: queue_fixed_jump_height} yields
    \begin{align}
        \PB(R_p>x,\,\sigma_y=k)
        \sim
        \PB(\beta_1>k-1)\PB(X_1>q_x).
        \label{eq: queue_fixed_jump}
    \end{align}
    Since the events
    $\{\sigma_y=k\}$ are disjoint, summing
    \eqref{eq: queue_fixed_jump} over $k=1,\ldots,N$ gives, for every
    fixed $N$,
    \begin{align}
        \PB(R_p>x,\,\sigma_y\le N)
        \sim
        \EB[\beta_1\wedge N]\PB(X_1>q_x).
        \label{eq: queue_early_jump}
    \end{align}
    Although $\EB[\beta_1\wedge N]\to\EB[\beta_1]$, we cannot let
    $N\to\infty$ directly in \eqref{eq: queue_early_jump}, because
    this asymptotic equivalence was obtained with $N$ fixed and its
    error is not known to be uniform in $N$. We therefore need to show
    that a large jump occurring after time $N$ is negligible on the
    $\PB(X_1>q_x)$ scale.
    To see this, note that, for each fixed $k$,
    \begin{align*}
        \PB(\sigma_y=k)
        &=
        \PB\left(\max_{n\le\beta_1\wedge k}S_n>y\right)
        -
        \PB\left(\max_{n\le\beta_1\wedge(k-1)}S_n>y\right).
    \end{align*}
    Applying the truncated-maximum asymptotics in Proposition~4 of
    \cite{denisov2021tail} to the two terms gives
    \begin{align*}
        \PB(\sigma_y=k)
        \sim\PB(\beta_1>k-1)\PB(X_1>y).
    \end{align*}
    By definition, $\{M>y\}=\{\sigma_y<\infty\}$ and
    $\{\sigma_y=k\}\subseteq\{M>y\}$. Hence, together with
    \eqref{eq: queue_max_tail},
    \begin{align*}
        \PB(M>y,\,\sigma_y>N)
        &=\PB(M>y)-\sum_{k=1}^N
        \PB(M>y,\,\sigma_y=k)\\
        &=\PB(M>y)-\sum_{k=1}^N\PB(\sigma_y=k)\\
        &\sim\left(\EB[\beta_1]-\EB[\beta_1\wedge N]\right)
        \PB(X_1>y).
    \end{align*}
    Since
    $\frac{\PB(X_1>y)}{\PB(X_1>q_x)}\to\varepsilon^{-\alpha}$,
    \begin{align*}
        \lim_{N\to\infty}\limsup_{x\to\infty}
        \frac{\PB(M>y,\,\sigma_y>N)}
        {\PB(X_1>q_x)}=0.
    \end{align*}
    Finally, $\{M>y\}=\{\sigma_y<\infty\}$, and
    \begin{align*}
        \PB(R_p>x,\,\sigma_y\le N)
        &\le\PB(R_p>x,\,M>y)\\
        &\le\PB(R_p>x,\,\sigma_y\le N)
        +\PB(M>y,\,\sigma_y>N).
    \end{align*}
    Dividing this sandwich by $\PB(X_1>q_x)$, applying
    \eqref{eq: queue_early_jump} for fixed $N$, and then letting
    $N\to\infty$ gives
    \begin{align}
        \label{eq: queue_large_max}
        \PB(R_p>x,\,M>\varepsilon q_x)
        &\sim \EB[\beta_1]\PB(X_1>q_x).
    \end{align}
    It remains to rule out a large reward produced without a large
    maximum. Since $R_p\le \beta_1 M^p$, equations
    \eqref{eq: queue_length_tail} and $q_x^{p+1}=a(p+1)x$ give
    \begin{align*}
        \PB(R_p>x,\,M\le\varepsilon q_x)
        &\le\PB\left(\beta_1>
        \frac{x}{(\varepsilon q_x)^p}\right)\\
        &\sim\EB[\beta_1]\PB\left(
        X_1>\frac{q_x}{(p+1)\varepsilon^p}\right).
    \end{align*}
    By regular variation and combining with
    \eqref{eq: queue_large_max}, we obtain
    \begin{align*}
        1\le\liminf_{x\to\infty}
        \frac{\PB(R_p>x)}{\EB[\beta_1]\PB(X_1>q_x)}
        \le\limsup_{x\to\infty}
        \frac{\PB(R_p>x)}{\EB[\beta_1]\PB(X_1>q_x)}
        \le1+((p+1)\varepsilon^p)^{\alpha}.
    \end{align*}
    Letting $\varepsilon\downarrow0$ proves
    \eqref{eq: queue_reward_tail}. The claimed regular-variation
    index follows from \eqref{eq: increment_tail}.\hfill\qedsymbol

\subsubsection{Proof of Theorem~\ref{thm: queue}}
Let $\beta_0=0$ and define the successive regeneration times by
$\beta_{k+1}=\inf\{j>\beta_k:W_j=0\}$. We decompose
$Z_1(n)-n\mu_W$ as
\begin{align*}
    Z_1(n)-n\mu_W
    =\sum_{k=1}^{\chi(n)}(D_k-\mu_W\tau_k)
    +\sum_{i=\beta_{\chi(n)}+1}^n(W_i-\mu_W),
\end{align*}
where $\chi(n)=\max\{k:\beta_k\le n\}$,
$\tau_k=\beta_k-\beta_{k-1}$,
$D_k=\sum_{i=\beta_{k-1}+1}^{\beta_k}W_i$. The pairs
$\{(D_k,\tau_k)\}_{k\ge1}$ are i.i.d. It follows from
Theorem~\ref{thm: queue_reward_tail} and~\eqref{eq: queue_length_tail}
that
$\PB(D_1>x)\in\mathrm{RV}_{-\frac{\alpha}{2}}$ and
$\PB(\tau_1>x)\in\mathrm{RV}_{-\alpha}$, respectively. Hence
$D_1-\mu_W\tau_1$ is $\frac{\alpha}{2}$-regularly varying. Moreover,
$\EB[D_1-\mu_W\tau_1]=0$. Thus, for a suitable slowly varying function
$\ell$,
\begin{align}
    n^{-\frac{2}{\alpha}}\ell(n)
    \sum_{k=1}^{\lfloor n\,\cdot\rfloor}(D_k-\mu_W\tau_k)
    \overset{J_1}{\Rightarrow}L_{\frac{\alpha}{2}}(\cdot).
    \label{eq: queue_cycle_fclt}
\end{align}
Moreover, by the definition of $\chi(n)$,
\begin{align*}
    \frac{\chi(nt)}{n}
    &=\max\left\{\frac{k}{n}:
    \sum_{i=1}^k\tau_i\le nt\right\}\\
    &=\max\left\{s\in n^{-1}\mathbb N_0:
    \frac{1}{n}\sum_{i=1}^{ns}\tau_i\le t\right\}
    \toas\frac{t}{\EB[\tau_1]}.
\end{align*}
The convergence holds locally uniformly in $t$. Since the limit is
deterministic, this convergence holds
jointly with~\eqref{eq: queue_cycle_fclt}. Thus, by the
random-time-change theorem,
\begin{align}
    n^{-\frac{2}{\alpha}}\ell(n)
    \sum_{k=1}^{\chi(n\,\cdot)}(D_k-\mu_W\tau_k)
    \overset{J_1}{\Rightarrow}
    L_{\frac{\alpha}{2}}\left(\frac{\cdot}{\EB[\tau_1]}\right).
    \label{eq: queue_renewal_reward_fclt}
\end{align}
However, the normalized remainder term
$n^{-\frac{2}{\alpha}}\ell(n)
\sum_{i=\beta_{\chi(\lfloor nt\rfloor)}+1}^{\lfloor nt\rfloor}W_i$
does not converge uniformly to zero in probability. Instead, for every
$T<\infty$, we prove
\begin{align}
&d_{M_1,[0,T]}(\widetilde W_n,\widetilde Y_n)\\
:=&d_{M_1,[0,T]}\left(
    n^{-\frac{2}{\alpha}}\ell(n)
    \left(Z_1(n\,\cdot)-n\mu_W\,\cdot\right),
    n^{-\frac{2}{\alpha}}\ell(n)
    \sum_{k=1}^{\chi(n\,\cdot)}(D_k-\mu_W\tau_k)
    \right)\topb0.
    \label{eq: queue_cycle_compression}
\end{align}
For cycle $k$ and $0\le s\le\tau_k$, let
$C_k(s)=\sum_{i=\beta_{k-1}+1}^{\beta_{k-1}+\lfloor s\rfloor}W_i$,
and let $(r_k(u),a_k(u))$, $0\le u\le1$, be an ordered continuous
parametric representation of its completed graph, running from $(0,0)$
to $(\tau_k,D_k)$. We first describe the two cycle-level representations
relative to their common spatial value at time $\frac{\beta_{k-1}}{n}$ and
before normalization by $n^{-\frac{2}{\alpha}}\ell(n)$. For
$0\le u\le\frac{1}{2}$, set
\begin{align*}
    (\lambda_{1,n}^{(k)}(u),\rho_{1,n}^{(k)}(u))
    &=\left(\frac{\beta_{k-1}}{n},0\right),\\
    (\lambda_{2,n}^{(k)}(u),\rho_{2,n}^{(k)}(u))
    &=\left(\frac{\beta_{k-1}+2u\tau_k}{n},0\right).
\end{align*}
For $\frac{1}{2}\le u\le1$, set
\begin{align*}
    \lambda_{1,n}^{(k)}(u)
    &=\frac{\beta_{k-1}+r_k(2u-1)}{n},
    &\rho_{1,n}^{(k)}(u)
    &=a_k(2u-1)-\mu_Wr_k(2u-1),\\
    \lambda_{2,n}^{(k)}(u)
    &=\frac{\beta_k}{n},
    &\rho_{2,n}^{(k)}(u)
    &=\begin{cases}
        \displaystyle
        \frac{D_k-\mu_W\tau_k}{D_k}a_k(2u-1),&D_k>0,\\[6pt]
        -\mu_Wr_k(2u-1),&D_k=0.
    \end{cases}
\end{align*}
The first half lets the second path traverse its horizontal segment while
the first path remains at the beginning of the cycle. In the second half,
the first path traverses the completed graph of
$C_k(s)-\mu_Ws$, while the second path traverses its vertical jump at
$\frac{\beta_k}{n}$. Since $W_i\ge0$, $\frac{a_k(u)}{D_k}\in[0,1]$ when $D_k>0$;
when $D_k=0$, we have $C_k\equiv0$. Hence
\begin{align*}
    \|\lambda_{1,n}^{(k)}-\lambda_{2,n}^{(k)}\|_\infty
    &\le\frac{\tau_k}{n},\\
    \|\rho_{1,n}^{(k)}-\rho_{2,n}^{(k)}\|_\infty
    &\le\mu_W\tau_k.
\end{align*}
After normalization by $n^{-\frac{2}{\alpha}}\ell(n)$, we have:
\begin{align*}
&d_{M_1,[0,T]}(\widetilde W_n,\widetilde Y_n)\\
\le&\max_{k\le\chi(\lfloor nT\rfloor)}
    \|\lambda_{1,n}^{(k)}-\lambda_{2,n}^{(k)}\|_\infty
    +n^{-\frac{2}{\alpha}}\ell(n)\left(
    \max_{k\le\chi(\lfloor nT\rfloor)}
    \|\rho_{1,n}^{(k)}-\rho_{2,n}^{(k)}\|_\infty
    +D_{\chi(\lfloor nT\rfloor)+1}
    +\mu_W\tau_{\chi(\lfloor nT\rfloor)+1}\right)\\
\le&\left(n^{-1}+\mu_Wn^{-\frac{2}{\alpha}}\ell(n)\right)
    \max_{k\le\chi(\lfloor nT\rfloor)+1}\tau_k
    +n^{-\frac{2}{\alpha}}\ell(n)
    (D_{\chi(\lfloor nT\rfloor)+1}
    +\mu_W\tau_{\chi(\lfloor nT\rfloor)+1}).
\end{align*}
Since $\PB(\tau_1>x)\in\mathrm{RV}_{-\alpha}$ by
\eqref{eq: queue_length_tail}, for every $\delta>0$, we have:
\begin{align*}
    \PB\left(
    \left(n^{-1}+\mu_Wn^{-\frac{2}{\alpha}}\ell(n)\right)
    \max_{k\le\chi(\lfloor nT\rfloor)+1}\tau_k>\delta\right)
    \le(\lfloor nT\rfloor+1)\!\PB\left(
    \tau_1>\frac{\delta}
    {n^{-1}+\mu_Wn^{-\frac{2}{\alpha}}\ell(n)}\right)\to0.
\end{align*}
To control the remaining term, note that
$\sum_{k\ge0}\PB(\beta_k=s)\le1$ for every integer $s$, since
$\{\beta_k\}$ is strictly increasing. Moreover,
$\{\chi(\lfloor nT\rfloor)=k\}
=\{\beta_k\le\lfloor nT\rfloor<\beta_{k+1}\}$. Since
$(D_{k+1},\tau_{k+1})$ is independent of $\beta_k$ and has the same
distribution as $(D_1,\tau_1)$, for every $x>0$,
\begin{align*}
\PB(D_{\chi(\lfloor nT\rfloor)+1}>x)
&=\sum_{k\ge0}\PB(D_{k+1}>x,
    \beta_k\le\lfloor nT\rfloor<\beta_{k+1})\\
&=\sum_{s=0}^{\lfloor nT\rfloor}\sum_{k\ge0}
    \PB(\beta_k=s,D_{k+1}>x,
    \tau_{k+1}>\lfloor nT\rfloor-s)\\
&=\sum_{s=0}^{\lfloor nT\rfloor}\sum_{k\ge0}\PB(\beta_k=s)
    \PB(D_1>x,\tau_1>\lfloor nT\rfloor-s)\\
&\le\sum_{s=0}^{\lfloor nT\rfloor}
    \PB(D_1>x,\tau_1>\lfloor nT\rfloor-s)\\
&\le\sum_{r=0}^{\infty}\PB(D_1>x,\tau_1>r)\\
&=\EB\left[\tau_1\mathbbm{1}\{D_1>x\}\right].
\end{align*}
Similarly, we have:
\begin{align*}
    \PB(\tau_{\chi(\lfloor nT\rfloor)+1}>x)
    \le\EB\left[\tau_1\mathbbm{1}\{\tau_1>x\}\right].
\end{align*}
Both bounds tend to zero as $x\to\infty$. Taking
$x=\frac{\delta n^{\frac{2}{\alpha}}}{\ell(n)}$ proves
\eqref{eq: queue_cycle_compression}. Combining
\eqref{eq: queue_renewal_reward_fclt} and
\eqref{eq: queue_cycle_compression} gives the $M_1$ convergence.
Taking $n=\varepsilon^{-1}$ proves~\eqref{eq: queue_fclt}. Finally, the
strong law gives $\frac{Z_1(t)}{t}\to\mu_W$ almost surely.
The limiting L\'evy process is $\frac{\alpha}{2}$-stable and hence
$\frac{2}{\alpha}$-self-similar and
quasi-left-continuous w.r.t.\ its natural filtration.

\subsubsection{Proof of Theorem~\ref{thm: queue_early_stopping}}
Fix $0<\nu<\frac{\alpha}{\alpha-2}$.
All implicit constants below are independent of $t$, $\varepsilon$, and
$\eta$.
Define the deterministic-time regenerative approximation
\begin{align*}
    \widehat Z_1(t)
    :=\sum_{k=1}^{\lfloor \frac{t}{\EB[\tau_1]}\rfloor}
    (D_k-\mu_W\tau_k),
    \qquad t\ge0,
\end{align*}
and use independent copies of this process for sectioning. For $t>0$,
define
\begin{align*}
    \Lambda_{\mathrm{sec}}(t)
    &:=\frac{2}{t}\int_{\frac{t}{2}}^t
    \left(Z_1^{(1)}(s)-Z_1^{(2)}(s)\right)\,\mathrm ds,\\
    \Lambda_{\mathrm{bm}}(t)
    &:=\frac{2}{t}\int_{\frac{t}{2}}^t
    \left(Z_1\left(\frac{s}{m}\right)-Z_1(s)
    +Z_1\left(\frac{m-1}{m}s\right)\right)\,\mathrm ds,\\
    \Lambda_{\mathrm{rs}}(t)
    &:=\frac{8}{t}\int_{\frac{t}{2}}^t\frac{1}{u}
    \int_{\frac{u}{4}}^{\frac{u}{2}}
    \left(Z_1(s)-\frac{s}{u}Z_1(u)\right)\,\mathrm ds\,\mathrm du.
\end{align*}
Let $\widehat\Lambda_\iota(t)$ denote the same functional with each
$Z_1$ replaced by $\widehat Z_1$. Replacing $Z_1(v)$ by
$Z_1(v)-\mu_Wv$ leaves all three functionals unchanged. For sectioning,
\begin{align}
    Z_2^{*,\mathrm{sec}}(t)
    &\overset{(a)}{=}\frac{1}{m(m-1)t}
    \sum_{1\le i<j\le m}\int_0^t
    \left(Z_1^{(i)}(s)-Z_1^{(j)}(s)\right)^2\,\mathrm ds\notag\\
    &\overset{(b)}{\ge}\frac{1}{m(m-1)t}\int_{\frac{t}{2}}^t
    \left(Z_1^{(1)}(s)-Z_1^{(2)}(s)\right)^2\,\mathrm ds\notag\\
    &\overset{(c)}{\ge}\frac{2}{m(m-1)t^2}
    \left(\int_{\frac{t}{2}}^t
    \left(Z_1^{(1)}(s)-Z_1^{(2)}(s)\right)\,\mathrm ds\right)^2\notag\\
    &\overset{(d)}{=}\frac{\Lambda_{\mathrm{sec}}(t)^2}{2m(m-1)},
    \label{eq: queue_sec_projection}
\end{align}
where (a) is due to the identity
$\sum_{i=1}^m(x_i-\bar x)^2
=m^{-1}\sum_{i<j}(x_i-x_j)^2$, (b) retains the first two copies and
restricts the integral to $[\frac{t}{2},t]$, (c) follows from
the Cauchy--Schwarz inequality, and (d) follows from the definition of
$\Lambda_{\mathrm{sec}}(t)$. Similarly, for batch means,
\begin{align}
    Z_2^{*,\mathrm{bm}}(t)
    &\overset{(a)}{=}\frac{1}{m(m-1)t}
    \sum_{1\le k<j\le m}\int_0^t
    \left(Z_{1,k}^{\mathrm{bm}}(s)-Z_{1,j}^{\mathrm{bm}}(s)\right)^2\,\mathrm ds\notag\\
    &\overset{(b)}{\ge}\frac{1}{m(m-1)t}\int_{\frac{t}{2}}^t
    \left(Z_1\left(\frac{s}{m}\right)-Z_1(s)
    +Z_1\left(\frac{m-1}{m}s\right)\right)^2\,\mathrm ds\notag\\
    &\overset{(c)}{\ge}\frac{2}{m(m-1)t^2}
    \left(\int_{\frac{t}{2}}^t
    \left(Z_1\left(\frac{s}{m}\right)-Z_1(s)
    +Z_1\left(\frac{m-1}{m}s\right)\right)\,\mathrm ds\right)^2\notag\\
    &\overset{(d)}{=}\frac{\Lambda_{\mathrm{bm}}(t)^2}{2m(m-1)},
    \label{eq: queue_bm_projection}
\end{align}
where (a) follows from the same pairwise identity, (b) retains the
first and last batches and restricts the integral to $[\frac{t}{2},t]$, (c)
follows from the Cauchy--Schwarz inequality, and (d) follows from the
definition of $\Lambda_{\mathrm{bm}}(t)$. For random scaling,
\begin{align}
    Z_2^{*,\mathrm{rs}}(t)
    &\overset{(a)}{\ge}\frac{1}{t}\int_{\frac{t}{2}}^t\frac{1}{u}
    \int_{\frac{u}{4}}^{\frac{u}{2}}
    \left(Z_1(s)-\frac{s}{u}Z_1(u)\right)^2
    \,\mathrm ds\,\mathrm du\notag\\
    &\overset{(b)}{\ge}\frac{8}{t^2}\left(\int_{\frac{t}{2}}^t\frac{1}{u}
    \int_{\frac{u}{4}}^{\frac{u}{2}}\left(Z_1(s)-\frac{s}{u}Z_1(u)\right)
    \,\mathrm ds\,\mathrm du\right)^2\notag\\
    &\overset{(c)}{=}\frac{\Lambda_{\mathrm{rs}}(t)^2}{8},
    \label{eq: queue_rs_projection}
\end{align}
where (a) restricts the defining integrals to
$\frac{t}{2}\le u\le t$ and $\frac{u}{4}\le s\le \frac{u}{2}$, (b) follows from
the Cauchy--Schwarz inequality and
\begin{align*}
    \int_{\frac{t}{2}}^t\frac{1}{u}\int_{\frac{u}{4}}^{\frac{u}{2}}1\,\mathrm ds\,\mathrm du
    =\frac{t}{8},
\end{align*}
and (c) follows from the definition of $\Lambda_{\mathrm{rs}}(t)$.
We next compare the preceding projections with their regenerative
approximations. For sectioning,
\begin{align}
|\Lambda_{\mathrm{sec}}(t)-\widehat\Lambda_{\mathrm{sec}}(t)|
&\le\frac{2}{t}\sum_{i=1}^2\int_{\frac{t}{2}}^t
    |Z_1^{(i)}(s)-\mu_Ws-\widehat Z_1^{(i)}(s)|\,\mathrm ds\notag\\
&\le\frac{2}{t}\sum_{i=1}^2\int_0^t
    |Z_1^{(i)}(s)-\mu_Ws-\widehat Z_1^{(i)}(s)|\,\mathrm ds.
    \label{eq: queue_sec_approximation}
\end{align}
For batch means,
\begin{align}
|\Lambda_{\mathrm{bm}}(t)-\widehat\Lambda_{\mathrm{bm}}(t)|
&\le\frac{2}{t}\int_{\frac{t}{2}}^t
    \left|Z_1\left(\frac{s}{m}\right)-\frac{\mu_Ws}{m}
    -\widehat Z_1\left(\frac{s}{m}\right)\right|\,\mathrm ds\notag\\
&+\frac{2}{t}\int_{\frac{t}{2}}^t
    |Z_1(s)-\mu_Ws-\widehat Z_1(s)|\,\mathrm ds\notag\\
&+\frac{2}{t}\int_{\frac{t}{2}}^t
    \left|Z_1\left(\frac{m-1}{m}s\right)
    -\frac{m-1}{m}\mu_Ws
    -\widehat Z_1\left(\frac{m-1}{m}s\right)\right|\,\mathrm ds\notag\\
&\le\frac{2}{t}\left(m+1+\frac{m}{m-1}\right)
    \int_0^t|Z_1(s)-\mu_Ws-\widehat Z_1(s)|\,\mathrm ds.
    \label{eq: queue_bm_approximation}
\end{align}
For random scaling,
\begin{align}
&|\Lambda_{\mathrm{rs}}(t)-\widehat\Lambda_{\mathrm{rs}}(t)|\notag\\
\le&\frac{8}{t}\int_{\frac{t}{2}}^t\frac{1}{u}
    \int_{\frac{u}{4}}^{\frac{u}{2}}\left(
    |Z_1(s)-\mu_Ws-\widehat Z_1(s)|
    +\frac{s}{u}|Z_1(u)-\mu_Wu-\widehat Z_1(u)|
    \right)\,\mathrm ds\,\mathrm du\notag\\
\overset{(a)}{=}&\frac{8}{t}\int_{\frac{t}{8}}^{\frac{t}{2}}
    |Z_1(s)-\mu_Ws-\widehat Z_1(s)|
    \int_{\max\left(\frac{t}{2},2s\right)}^{\min(t,4s)}\frac{1}{u}\,\mathrm du\,\mathrm ds
    +\frac{3}{4t}\int_{\frac{t}{2}}^t
    |Z_1(u)-\mu_Wu-\widehat Z_1(u)|\,\mathrm du\notag\\
\overset{(b)}{\le}&\frac{8}{t}\left(\log 2+\frac{3}{32}\right)
    \int_0^t|Z_1(s)-\mu_Ws-\widehat Z_1(s)|\,\mathrm ds,
    \label{eq: queue_rs_approximation}
\end{align}
where (a) follows from Fubini's theorem and
$u^{-2}\int_{\frac{u}{4}}^{\frac{u}{2}}s\,\mathrm ds=\frac{3}{32}$, and (b) follows because the
ratio of the endpoints in the inner integral is at most $2$.
Thus, \eqref{eq: queue_sec_approximation},
\eqref{eq: queue_bm_approximation}, and
\eqref{eq: queue_rs_approximation} imply that, for every
$\iota\in\{\mathrm{sec},\mathrm{bm},\mathrm{rs}\}$,
\begin{align}
    |\Lambda_\iota(t)-\widehat\Lambda_\iota(t)|
    \lesssim\frac{1}{t}\int_0^t
    |Z_1(s)-\mu_Ws-\widehat Z_1(s)|\,\mathrm ds.
    \label{eq: queue_projection_integral_bound}
\end{align}
Its right-hand side is the sum of the corresponding integrals for the
first two copies in the sectioning case. We next bound the integral on the right-hand side
of~\eqref{eq: queue_projection_integral_bound}. For every $u\ge0$,
\begin{align}
    Z_1(u)-\mu_Wu-\widehat Z_1(u)
    =&Z_1(u)-\mu_Wu
    -\sum_{k\ge1}(D_k-\mu_W\tau_k)
    \mathbbm{1}\{\beta_k\le u\}\notag\\
    &+\sum_{k\ge1}(D_k-\mu_W\tau_k)
    \left(\mathbbm{1}\{\beta_k\le u\}
    -\mathbbm{1}\{k\EB[\tau_1]\le u\}\right).
    \label{eq: queue_path_decomposition}
\end{align}
For $\beta_{j-1}\le u<\beta_j$,
\begin{align}
\left|Z_1(u)-\mu_Wu
    -\sum_{k\ge1}(D_k-\mu_W\tau_k)
    \mathbbm{1}\{\beta_k\le u\}\right|
&=\left|\sum_{i=\beta_{j-1}+1}^{\lfloor u\rfloor}W_i
    -\mu_W(u-\beta_{j-1})\right|\notag\\
&\le D_j+\mu_W\tau_j.
    \label{eq: queue_within_cycle_bound}
\end{align}
Moreover,
\begin{align}
    \int_0^t\left|
    \mathbbm{1}\{\beta_k\le u\}
    -\mathbbm{1}\{k\EB[\tau_1]\le u\}\right|\,\mathrm du
    \le|\beta_k-k\EB[\tau_1]|.
    \label{eq: queue_clock_shift_bound}
\end{align}
Thus, for every $t\ge1$,
\begin{align}
&\int_0^t\left|Z_1(u)-\mu_Wu-\widehat Z_1(u)\right|\,\mathrm du\notag\\
\overset{(a)}{\le}&\sum_{k=1}^{\lfloor t\rfloor+1}
    (D_k+\mu_W\tau_k)\tau_k
    +\sum_{k=1}^{\lfloor t\rfloor+1}|D_k-\mu_W\tau_k|
    |\beta_k-k\EB[\tau_1]|\notag\\
\le&\sum_{k=1}^{\lfloor t\rfloor+1}
    (D_k+\mu_W\tau_k)\tau_k
    +\max_{k\le\lfloor t\rfloor+1}
    |\beta_k-k\EB[\tau_1]|
    \sum_{k=1}^{\lfloor t\rfloor+1}|D_k-\mu_W\tau_k|,
    \label{eq: queue_regenerative_integral_bound}
\end{align}
where (a) follows from
\eqref{eq: queue_path_decomposition},
\eqref{eq: queue_within_cycle_bound}, and
\eqref{eq: queue_clock_shift_bound}.
Choose
\begin{align*}
    0<\delta<\frac{2}{\alpha}-\frac12,
    \qquad
    \frac{2}{3+2\delta}<s<\min\left(1,\frac{\alpha}{3}\right).
\end{align*}
The tail results used in the proof of Theorem~\ref{thm: queue} imply
\begin{align*}
    \EB[\tau_1^2]<\infty,\qquad
    \EB[|D_1-\mu_W\tau_1|]<\infty,\qquad
    \EB[((D_1+\mu_W\tau_1)\tau_1)^s]<\infty.
\end{align*}
The last assertion is due to subadditivity and H\"older's inequality:
\begin{align*}
\EB[((D_1+\mu_W\tau_1)\tau_1)^s]
&\overset{(a)}{\le}\EB[D_1^s\tau_1^s]
    +\mu_W^s\EB[\tau_1^{2s}]\\
&\overset{(b)}{\le}\EB[D_1^{\frac{3s}{2}}]^{\frac{2}{3}}\EB[\tau_1^{3s}]^{\frac{1}{3}}
    +\mu_W^s\EB[\tau_1^{2s}]<\infty,
\end{align*}
where (a) follows from subadditivity because $s<1$, and (b) follows
from H\"older's inequality with conjugate exponents $\frac{3}{2}$ and $3$.
The last expression is finite because $D_1$ and $\tau_1$ have regularly
varying tails with indices $-\frac{\alpha}{2}$ and $-\alpha$, respectively, and
$3s<\alpha$. We next bound separately the three quantities on the
right-hand side of~\eqref{eq: queue_regenerative_integral_bound}:
\begin{align*}
&\PB\left(\sum_{k=1}^{\lfloor t\rfloor+1}
    (D_k+\mu_W\tau_k)\tau_k>t^{\frac{3}{2}+\delta}\right)
    \overset{(a)}{\le}t^{-s\left(\frac{3}{2}+\delta\right)}
    \EB\left[\left(\sum_{k=1}^{\lfloor t\rfloor+1}
    (D_k+\mu_W\tau_k)\tau_k\right)^s\right]
    \overset{(b)}{\lesssim}t^{1-s\left(\frac{3}{2}+\delta\right)},\\
&\PB\left(\max_{k\le\lfloor t\rfloor+1}
    |\beta_k-k\EB[\tau_1]|>t^{\frac{1}{2}+\frac{\delta}{2}}\right)
    \overset{(c)}{\le}\frac{(\lfloor t\rfloor+1)\Var(\tau_1)}
    {t^{1+\delta}}\lesssim t^{-\delta},\\
&\PB\left(\sum_{k=1}^{\lfloor t\rfloor+1}
    |D_k-\mu_W\tau_k|>t^{1+\frac{\delta}{2}}\right)
    \overset{(d)}{\le}\frac{(\lfloor t\rfloor+1)
    \EB[|D_1-\mu_W\tau_1|]}{t^{1+\frac{\delta}{2}}}
    \lesssim t^{-\frac{\delta}{2}}.
\end{align*}
Here, (a) and (d) follow from Markov's inequality, (b) follows from
subadditivity, and (c) follows from Kolmogorov's maximal inequality.
Consequently, \eqref{eq: queue_projection_integral_bound},
\eqref{eq: queue_regenerative_integral_bound}, and the three probability
bounds above imply that, for some $c,C_1>0$ and every
$\iota\in\{\mathrm{sec},\mathrm{bm},\mathrm{rs}\}$,
\begin{align}
    \PB\left(
    |\Lambda_\iota(t)-\widehat\Lambda_\iota(t)|
    >C_1t^{\frac{1}{2}+\delta}\right)
    \lesssim t^{-c}.
    \label{eq: queue_projection_approximation}
\end{align}
Recalling the definitions of $\widehat\Lambda_\iota(t)$, for
$\iota\in\{\mathrm{bm},\mathrm{rs}\}$ we can write
\begin{align*}
    \widehat\Lambda_\iota(t)
    =\sum_{k\ge1}a_{\iota,k,t}(D_k-\mu_W\tau_k),
\end{align*}
whereas, for sectioning,
\begin{align*}
    \widehat\Lambda_{\mathrm{sec}}(t)
    =\sum_{q=1}^2\sum_{k\ge1}a_{\mathrm{sec},q,k,t}
    (D_k^{(q)}-\mu_W\tau_k^{(q)}).
\end{align*}
Here, the coefficients $a_{\iota,k,t}$ and
$a_{\mathrm{sec},q,k,t}$ are deterministic. And for every $\iota$, a sufficient
number of coefficients are bounded away from zero. For
sectioning, if
$k\EB[\tau_1]\le \frac{t}{2}$, then
\begin{align*}
    a_{\mathrm{sec},1,k,t}
    &=\frac{2}{t}\int_{\frac{t}{2}}^t
    \mathbbm{1}\{k\EB[\tau_1]\le s\}\,\mathrm ds=1.
\end{align*}
For batch means, if $k\EB[\tau_1]\le \frac{t}{2m}$, then
\begin{align*}
    a_{\mathrm{bm},k,t}
    =&\frac{2}{t}\int_{\frac{t}{2}}^t\left(
    \mathbbm{1}\left\{k\EB[\tau_1]\le\frac{s}{m}\right\}
    -\mathbbm{1}\{k\EB[\tau_1]\le s\}
    +\mathbbm{1}\left\{k\EB[\tau_1]\le\frac{m-1}{m}s\right\}
    \right)\,\mathrm ds\\
    =&\frac{2}{t}\int_{\frac{t}{2}}^t(1-1+1)\,\mathrm ds=1.
\end{align*}
For random scaling, if $k\EB[\tau_1]\le \frac{t}{8}$, then
\begin{align*}
    a_{\mathrm{rs},k,t}
    =\frac{8}{t}\int_{\frac{t}{2}}^t\int_{\frac{1}{4}}^{\frac{1}{2}}(1-v)\,\mathrm dv\,\mathrm du
    =4\int_{\frac{1}{4}}^{\frac{1}{2}}(1-v)\,\mathrm dv=\frac58.
\end{align*}
Hence, for all sufficiently large $t$, there are at least $c_0t$
coefficients in $[\frac{5}{8},1]$, where
$c_0=(2\EB[\tau_1])^{-1}\min\{\frac{1}{2m},\frac{1}{8}\}$.
Let $\operatorname{med}(D_1-\mu_W\tau_1)$ denote a fixed median. For
$\frac{5}{8}\le a\le1$ and every $x\in\RB$, if
$\operatorname{med}(D_1-\mu_W\tau_1)<\frac{x-z}{a}$, then
\begin{align*}
    \PB\left(|a(D_1-\mu_W\tau_1)-x|>z\right)
    \ge\PB\left(D_1-\mu_W\tau_1
    \le\operatorname{med}(D_1-\mu_W\tau_1)\right)\ge\frac12.
\end{align*}
Otherwise,
\begin{align*}
    \frac{x+z}{a}
    =\frac{x-z}{a}+\frac{2z}{a}
    \le\operatorname{med}(D_1-\mu_W\tau_1)+\frac{16z}{5}
    \le4z
\end{align*}
for all sufficiently large $z$, and hence
\begin{align*}
    \PB\left(|a(D_1-\mu_W\tau_1)-x|>z\right)
    \ge\PB(D_1-\mu_W\tau_1>4z).
\end{align*}
Taking the infimum over $x$ gives
\begin{align*}
1-\sup_{x\in\RB}\PB\left(
    |a(D_1-\mu_W\tau_1)-x|\le z\right)
&=\inf_{x\in\RB}\PB\left(
    |a(D_1-\mu_W\tau_1)-x|>z\right)\\
&\ge\min\left\{\frac12,
    \PB(D_1-\mu_W\tau_1>4z)\right\}\\
&=\PB(D_1-\mu_W\tau_1>4z).
\end{align*}
The last equality follows because
$\PB(D_1-\mu_W\tau_1>4z)\to0$ as $z\to\infty$.
Thus, applying Lemma~\ref{lem: kolmogorov_rogozin}, we obtain, for all
sufficiently large $t$ and $z$,
\begin{align}
    \sup_{x\in\RB}\PB\left(|\widehat\Lambda_\iota(t)-x|\le z\right)
    \lesssim\frac{1}{\sqrt{c_0t\PB(D_1-\mu_W\tau_1>4z)}}.
    \label{eq: queue_projection_concentration}
\end{align}
Fix $K>0$, and set
$d_{\mathrm{sec}}=d_{\mathrm{bm}}=\sqrt{2m(m-1)}$,
$d_{\mathrm{rs}}=2\sqrt{2}$, and
$z_\iota(t):=d_\iota K\varepsilon t+C_1t^{\frac{1}{2}+\delta}$.
Let $A(t)=t^{1-\frac{2}{\alpha}}\ell(t)$. There exists
$0<\kappa<\min\{c,\frac{\alpha}{4}\}$ such that, for every
$0<\eta\le1$ and all
sufficiently small $\varepsilon$, uniformly for
$\lceil\varepsilon^{-\nu}\rceil\le t\le\eta a_\varepsilon$,
\begin{align}
\PB\left(
    \frac{\sqrt{Z_2^{*,\iota}(t)}}{t}\le K\varepsilon\right)
    &\overset{(a)}{\le}\PB\left(|\Lambda_\iota(t)|
    \le d_\iota K\varepsilon t\right)\notag\\
    &\overset{(b)}{\lesssim}t^{-c}
    +\PB\left(|\widehat\Lambda_\iota(t)|\le z_\iota(t)\right)\notag\\
    &\overset{(c)}{\lesssim}t^{-c}
    +\frac{1}{\sqrt{t\PB(
    D_1-\mu_W\tau_1>4z_\iota(t))}}\notag\\
    &\overset{(d)}{\lesssim}
    (\varepsilon A(t))^\kappa+t^{-\kappa}.
    \label{eq: queue_scaling_small_ball}
\end{align}
(a) follows from \eqref{eq: queue_sec_projection},
\eqref{eq: queue_bm_projection}, and~\eqref{eq: queue_rs_projection},
(b) follows from~\eqref{eq: queue_projection_approximation}, and (c)
follows by taking $x=0$ in~\eqref{eq: queue_projection_concentration}.
To verify (d), Potter's bound and
$A(a_\varepsilon)\sim\varepsilon^{-1}$ give, uniformly for
$\lceil\varepsilon^{-\nu}\rceil\le t\le\eta a_\varepsilon$,
\begin{align*}
    \varepsilon A(t)
    \asymp\frac{A(t)}{A(a_\varepsilon)}
    \lesssim\left(\frac{t}{a_\varepsilon}
    \right)^{\frac{\alpha-2}{2\alpha}}\le1.
\end{align*}
Moreover, the normalization in~\eqref{eq: queue_cycle_fclt} implies
$t\PB\left(D_1-\mu_W\tau_1>\frac{t^{\frac{2}{\alpha}}}{\ell(t)}\right)\asymp1$, while
\begin{align*}
    \frac{z_\iota(t)\ell(t)}{t^{\frac{2}{\alpha}}}
    =d_\iota K\varepsilon A(t)
    +C_1t^{-\left(\frac{2}{\alpha}-\frac{1}{2}-\delta\right)}\ell(t).
\end{align*}
Thus, for some $\kappa'\in\left(0,\frac{\alpha}{4}\right)$, Potter's bound gives
\begin{align*}
    \frac{1}{\sqrt{t\PB(D_1-\mu_W\tau_1>4z_\iota(t))}}
    \lesssim\left(\frac{z_\iota(t)\ell(t)}{t^{\frac{2}{\alpha}}}\right)^{\kappa'}
    \lesssim(\varepsilon A(t))^\kappa+t^{-\kappa},
\end{align*}
where the last inequality holds for the sufficiently small $\kappa$
chosen above.
Finally, set $t_j=2^j\lceil\varepsilon^{-\nu}\rceil$, and let
$J_\varepsilon$ be the largest integer such that
$t_{J_\varepsilon}\le\eta a_\varepsilon$. If
$t_j\le n<2t_j$ and
$\frac{\sqrt{Z_2^{*,\iota}(n)}}{n}\le K\varepsilon$, then
\begin{align*}
    \sqrt{Z_2^{*,\iota}(t_j)}
    \overset{(a)}{\le}
    \sqrt{\frac{n}{t_j}}\sqrt{Z_2^{*,\iota}(n)}
    \le2\sqrt{2}K\varepsilon t_j,
\end{align*}
where (a) follows because $tZ_2^{*,\iota}(t)$ is nondecreasing. Therefore,
\begin{align*}
\PB\left(
    \inf_{\varepsilon^{-\nu}\le n\le\eta a_\varepsilon}
    \frac{\sqrt{Z_2^{*,\iota}(n)}}{n}\le K\varepsilon\right)
&\overset{(a)}{\lesssim}\sum_{j=0}^{J_\varepsilon}
    \left((\varepsilon A(t_j))^\kappa+t_j^{-\kappa}\right)\\
&\overset{(b)}{\lesssim}\eta^{\frac{\kappa(\alpha-2)}{2\alpha}}
    +\varepsilon^{\nu\kappa},
\end{align*}
where (a) follows from the dyadic covering and
\eqref{eq: queue_scaling_small_ball}, while (b) follows from
$\varepsilon A(t_j)\lesssim
\left(\frac{t_j}{a_\varepsilon}\right)^{\frac{\alpha-2}{2\alpha}}
\le\eta^{\frac{\alpha-2}{2\alpha}}$.
Letting first $\varepsilon\downarrow0$ and then $\eta\downarrow0$
proves Assumption~\ref{asmp: early_stopping}.

%% file: tex/conclusion.tex
This paper studies fixed-width sequential stopping when the scaling process has a nondegenerate random limit. We show that the confidence interval remains asymptotically valid at the resulting random stopping time. To make the procedure practical, we use the invariance of the terminal distribution and develop a sequential subsampling method to estimate its unknown quantiles. We first consider scaling statistics based on the conventional sum of squares, which apply to i.i.d.\ mean estimation, short-range dependent time series, and stochastic approximation. We also develop alternative choices based on sectioning, batch means, and random scaling for settings in which the scaling statistic and estimation error have different stochastic orders. These alternatives apply to the M/G/1 queue studied here and can also accommodate certain long-range dependent models when the required assumptions are verified. Together, the results extend fixed-width sequential inference beyond settings with a deterministic scaling limit.